\documentclass[sn-mathphys-num]{sn-jnl}

\usepackage{graphicx}%
\usepackage{multirow}%
\usepackage{amsmath,amssymb,amsfonts}%
\usepackage{amsthm}%
\usepackage{mathrsfs}%
\usepackage[title]{appendix}%
\usepackage{xcolor}%
\usepackage{textcomp}%
\usepackage{manyfoot}%
\usepackage{booktabs}%
\usepackage{algorithm}%
\usepackage{algorithmicx}%
\usepackage{algpseudocode}%
\usepackage{listings}%
\usepackage{comment}

\def\b{{\mathbf{b}}}

\def\x{{\mathbf{x}}}

\def\u{{\mathbf{u}}}
\def\v{{\mathbf{v}}}
\def\z{{\mathbf{z}}}
\def\w{{\mathbf{w}}}

\def\n{{\mathbf{n}}}
\def\y{{\mathbf{y}}}

\def\b{{\mathbf{b}}}
\def\X{{\mathbf{X}}}
\def\Y{{\mathbf{Y}}}

\def\A{{\mathbf{A}}}
\def\D{{\mathbf{D}}}
\def\M{{\mathbf{M}}}
\def\N{{\mathbf{N}}}
\def\I{{\mathbf{I}}}
\def\B{{\mathbf{B}}}

\def\V{{\mathbf{V}}}

\def\Z{{\mathbf{Z}}}
\def\W{{\mathbf{W}}}
\def\U{{\mathbf{U}}}
\def\Q{{\mathbf{Q}}}
\def\P{{\mathbf{P}}}

\def\G{{\mathbf{G}}}

\def\Sn{{\mathcal{S}_n}}

\DeclareMathOperator*{\argmin}{arg\,min}
\DeclareMathOperator*{\argmax}{arg\,max}

\newcommand{\mY}{\mathcal{Y}}

\newcommand{\mA}{\mathcal{A}}

\newcommand{\mK}{\mathcal{K}}
\newcommand{\mS}{\mathcal{S}}
\newcommand{\mX}{\mathcal{X}}

\newcommand{\mbS}{\mathbb{S}}

\newcommand{\trace}{\textnormal{\textrm{Tr}}}
\newcommand{\rank}{\textnormal{\textrm{rank}}}

\newcommand{\diag}{\textnormal{\textrm{diag}}}
\newcommand{\support}{\textnormal{\textrm{support}}}
\newcommand{\reals}{\mathbb{R}}
\newcommand{\sign}{\textnormal{\textrm{sign}}}
\newcommand{\breg}{B}
\newcommand{\bLambda}{\boldsymbol{\Lambda}}

\usepackage{subcaption}

\newtheorem{theorem}{Theorem}

\newtheorem{lemma} {Lemma}
\newtheorem{definition} {Definition}
\newtheorem{corollary} {Corollary}

\newtheorem{assumption} {Assumption}

\newtheorem{remark} {Remark}

\begin{document}

\title[Article Title]{Low-Rank Mirror-Prox Methods for Nonsmooth and Low-Rank Matrix Optimization Problems\footnote{A preliminary version of this work which includes only the results for the Euclidean setup as appeared in the conference NeurIPS 2021 \cite{kaplan2021low}. This version substantially extends it by considering also the mirror-prox method with exponentiated gradient updates which enjoys better theoretical convergence rates (in terms of Lipschitz constants) and requires a different and more involved analysis. This version also includes additional numerical experiments.}}


\author*[1]{\fnm{Dan} \sur{Garber}}\email{dangar@technion.ac.il}

\author[1]{\fnm{Atara} \sur{Kaplan}}\email{ataragold@campus.technion.ac.il}
\equalcont{This author contributed equally to this work.}


\affil[1]{\orgdiv{Faculty of Data and Decision Sciences}, \orgname{Technion}, \orgaddress{  \postcode{32000} \city{Haifa}, \country{Israel}}}

%
%
%
%
%


\abstract{
Low-rank and nonsmooth matrix optimization problems capture many fundamental tasks in statistics and machine learning.
While significant progress has been made in developing efficient methods for \textit{smooth} problems that avoid computing expensive high-rank SVDs, advances for nonsmooth problems have been slow paced. 

In this paper we consider a standard convex relaxation: minimizing a convex nonsmooth objective function over the spectrahedron (set of real positive-semidefinite matrices with unit trace), under a  plausible strict complementarity (SC) condition. Following an observation that, even arbitrarily close to a low-rank optimal solution which satisfies SC, the standard projected subgradient descent method may fail to produce low-rank iterates,  we focus on nonsmooth objectives that can be written as a maximum of smooth functions and consider  the corresponding saddle-point formulation.
Mainly, we prove that  (approximated) variants of two popular \textit{mirror-prox} methods: the Euclidean extragradient method and mirror-prox with matrix exponentiated gradient updates,  when initialized with a ``warm-start'', converge to an optimal solution with the standard $O(1/t)$ rate, while requiring only two \textit{low-rank} SVDs per iteration. For the Euclidean method we also consider relaxed versions of SC which yield a trade-off between the rank of SVDs and the radius of the ball in which we need to initialize.
We support our theoretical findings with empirical experiments on several tasks, demonstrating both the plausibility of the SC assumption, and the efficient convergence of our proposed mirror-prox variants.}




\maketitle

\section{Introduction}

Low-rank and nonsmooth matrix optimization problems have many important applications in statistics, machine learning, and related fields, such as \textit{sparse PCA} \cite{sparsePCA1, sparsePCA2}, \textit{robust PCA} \cite{robustPCA1, robustPCA2, robustPCA3, robustPCA4, robustPCA5},  \textit{phase synchronization} \cite{phaseSyncronization1, phaseSyncronization2, phaseSyncronization3}, \textit{community detection and stochastic block models} \cite{abbe2017community}\footnote{In \cite{phaseSyncronization1, phaseSyncronization2, phaseSyncronization3} and \cite{abbe2017community} the authors consider SDPs with linear objective function and affine constraints of the form $\mA(\X)=\b$. By incorporating the linear constraints into the objective function via a $\ell_2$ penalty term of the form $\lambda\Vert{\mA(\X)-\b}\Vert_2$, $\lambda > 0$, we obtain a nonsmooth objective function, though choosing an appropriate value of $\lambda$ may be difficult.},  \textit{low-rank and sparse covariance matrix recovery} \cite{lowRankAndSparse}, \textit{robust matrix completion} \cite{klopp2017robust, chen2011robust}, and more. For many of these problems, convex relaxations, in which one replaces the nonconvex low-rank constraint with a trace-norm constraint, have been demonstrated in numerous papers to be highly effective both in theory (under suitable assumptions) and empirically (see references above).
These convex relaxations can be formulated as the following general nonsmooth optimization problem:
\begin{align} \label{nonSmoothProblem}
\min_{\X\in\Sn} g(\X),
\end{align} 
where $g:\mathbb{S}^n\rightarrow\reals$ is convex but nonsmooth, and $\mathcal{S}_n=\{\X\in\mathbb{S}^n\ \vert\ \trace(\X)=1,\ \X\succeq0\}$ is the spectrahedron in $\mathbb{S}^n$, $\mbS^n$ being the space of $n\times n$ real symmetric matrices.

Problem \eqref{nonSmoothProblem}, despite being convex, is notoriously difficult to solve in large scale. Perhaps the simplest and most general approach applicable to it are \textit{mirror decent} (MD) methods \cite{beckOptimizationBook, sebastienMD}. However, the MD setups of interest for Problem \eqref{nonSmoothProblem}, which are mostly the Euclidean setup and the von Neumann entropy setup (see discussions in the sequel), require in worst case $O(n^3)$ runtime per iteration. In many applications $g(\X)$ follows a composite model, i.e., $g(\X) = h(\X) + w(\X)$, where $h(\cdot)$ is convex and smooth and $w(\cdot)$ is convex and nonsmooth but admits a simple structure (e.g., nonsmooth regularizer). For such composite objectives, without the spectrahedron constraint, proximal methods such as FISTA \cite{FISTA} or splitting methods such as ADMM  \cite{ADMM} are often very effective. However, with the spectrahderon constraint, all such methods require on each iteration to apply a subprocedure (e.g., computing the proximal mapping) which in worst case amounts to at least $O(n^3)$ runtime. 
A third type of off-the-shelf methods include those which are based on the \textit{conditional gradient} method (aka Frank-Wolfe algorithm) and adapted to nonsmooth problems, see for instance \cite{odor2016frank, lowRankAndSparseOurs, NEURIPS2020_8f468c87, locatello2019stochastic}. The advantage of such methods is that  no expensive high-rank SVD computations are needed. Instead, only a single leading eigenvector computation (i.e., a rank-one SVD) per iteration is required. However, these methods typically suffer from slow convergence rates compared to proximal methods \cite{lan2013complexity}. 

In the recent works \cite{garberLowRankSmooth, garberStochasticLowRank} it was established that for smooth objective functions, the high-rank SVD computations required for Euclidean projections onto the spectrahedron in standard Euclidean gradient methods, can be replaced with low-rank SVDs in the  close proximity of  a low-rank optimal solution. This is significant since the runtime to compute a rank-$r$ SVD of a  $n\times n$ matrix using efficient iterative methods typically scales with $rn^2$ (and further improves when the matrix is sparse or admits a low-rank factorization), instead of $n^3$ for a full-rank SVD.
These results assume the existence of eigen-gaps in the gradient of the optimal solution, which we refer to in this work as a \textit{generalized strict complementarity condition} (see definition in the sequel).  
These results also hinge on a unique property of the Euclidean projection onto the spectrahedron. The projection onto the spectrahedron of a matrix $\X\in\mathbb{S}^n$, which admits an eigen-decomposition $\X=\sum_{i=1}^n\lambda_i\v_i\v_i^{\top}$, is given by 
\begin{align}\label{eq:euclidProj}
\Pi_{\mathcal{S}_n}[\X]=\sum_{i=1}^n\max\{0,\lambda_i-\lambda\}\v_i\v_i^{\top},
\end{align}
where $\lambda\in\reals$ is the unique scalar satisfying  $\sum_{i=1}^n\max\{0,\lambda_i-\lambda\}=1$. This operation thus truncates all eigenvalues that are smaller than $\lambda$, while leaving the eigenvectors unchanged, thereby returning a matrix with rank equal to the number of eigenvalues greater than $\lambda$. Importantly, when the projection of $\X$ onto $\mS_n$ is of rank $r$, only the first $r$ components in the eigen-decomposition of $\X$ are required to compute it in the first place, and thus, only a rank-$r$ SVD of $\X$ is required.
In other words and simplifying, \cite{garberLowRankSmooth, garberStochasticLowRank}  show that under (generalized) strict complementary, at the proximity of an optimal solution of rank $r$, the exact Euclidean  projection  equals the rank-$r$ truncated projection given by:
\begin{align} \label{truncatedProjection}
\widehat{\Pi}_{\mathcal{S}_n}^r[\X]:=\Pi_{\mathcal{S}_n}\left[\sum_{i=1}^r\lambda_i\v_i\v_i^{\top}\right].
\end{align}

In our recent work \cite{garber2020efficient}, similar results were obtained for smooth objective functions, when using Non-Euclidean von Neumann entropy-based gradient methods, a.k.a  matrix exponentiated gradient methods \cite{BregmanMatrices}. The importance of these methods lies in the fact that they allow to measure the Lipschitz parameters (either of the function or its gradients) with respect to the matrix spectral norm, which can lead to significantly better convergence rates (e.g., in terms of dependence on the dimension), than when measuring these with respect to the Euclidean norm. A standard matrix exponentiated gradient (MEG) step from a matrix $\X\in\mathbb{S}^n$ and with step-size $\eta>0$, can be written as
\begin{align} \label{exponentiatedGradientUpdate}
\X_{+}:=\frac{1}{\tau}\exp(\log(\X)-\eta\nabla{}g(\X)),\qquad\tau:=\trace(\exp(\log(\X)-\eta\nabla{}g(\X))).
\end{align}
Computing the matrix logarithm and matrix exponent for the MEG update requires computing a full-rank SVD\footnote{Only one SVD is required since $\X$ is stored in a factored form. See discussion in Section \ref{sec:MEG}.} and returns a full-rank matrix, unlike the truncation of the lower eigenvalues in Euclidean projections. Nevertheless, as shown in \cite{garber2020efficient}, under a  strict complementarity assumption, and when close to an optimal rank-$r$ solution, it is possible to approximate these steps with updates that only require rank-$r$ SVDs, while sufficiently controlling the errors resulting from the approximations. These approximated steps can be written as
\begin{align}\label{lowrankexponentiatedGradientUpdate}
\widehat{\X}_{+}^r:=\frac{1-\varepsilon}{\tau^r}\Y^r+\frac{\varepsilon}{n-r}(\I-\V^r{\V^r}^{\top}),
\end{align}
where $\varepsilon\in[0,1]$, $\Y_r$, which admits the eigen-decomposition $\Y^r=\V^r\bLambda^r{\V^r}^{\top}$, is the rank-$r$ approximation of $\Y:=\exp(\log(\X)-\eta\nabla{}g(\X))=\V\bLambda\V^{\top}$, and $\tau^r=\trace(\bLambda^r)$. That is, somewhat similarly to the low-rank Euclidean projection in \eqref{truncatedProjection}, the low-rank update in \eqref{lowrankexponentiatedGradientUpdate} only uses the top-$r$ components in the eigen-decomposition of matrix $\exp(\log\X-\eta\nabla{}g(\X))$ (as opposed to the exact MEG step \eqref{exponentiatedGradientUpdate} which requires all components), however, differently from the Euclidean case, the lower eigenvalues are not truncated to zero but to the same (small) value $\varepsilon/(n-r)$.

Extending the results of \cite{garberLowRankSmooth, garberStochasticLowRank,garber2020efficient} to the nonsmooth setting is difficult since the smoothness assumption is crucial to the analysis. 
Moreover, while  \cite{garberLowRankSmooth, garberStochasticLowRank,garber2020efficient} rely on certain eigen-gaps in the gradients at optimal points, for nonsmooth problems, since the subdifferential set is often not a singleton, it is not likely that a similar eigen-gap property holds for all subgradients of an optimal solution. In fact,  as we show in Section \ref{sec:subgradfail}, the standard projected subgradient descent method, which is arguably the simplest and most general method applicable to Problem \ref{nonSmoothProblem}, may fail to produce low-rank iterates even when arbitrarily close to a low-rank optimal solution which satisfies strict complementarity.

Motivated by the inability of the projected subgradient descent method to guarantee low-rank iterates near minimizers, we consider an additional mild assumption that Problem \eqref{nonSmoothProblem} can be formulated as a smooth convex-concave saddle-point problem, i.e., the nonsmooth term can be written as a maximum over (possibly infinite number of) smooth convex functions.
We then establish that results in the spirit of \cite{garberLowRankSmooth, garberStochasticLowRank,garber2020efficient} (low-rank iterates near minimizers) could be obtained when considering the corresponding saddle-point formulation. Concretely, we show that if a strict complementarity (SC) assumption holds for a low-rank optimal solution (see Assumption \ref{ass:strictcompNonSmooth} in the sequel), \textit{(approximated) mirror-prox methods} for smooth convex-concave saddle-point problems (see Algorithm \ref{alg:ApproxMP} below), which are either based on Euclidean projected gradient updates or MEG updates, and when initialized in the proximity of  an optimal solution, converge with their original ergodic (i.e., w.r.t. the average of iterates) convergence rate of $O(1/t)$, while requiring only two low-rank SVDs per iteration\footnote{note that mirror-prox methods compute two primal steps on each iteration, and thus two SVDs are needed per iteration.}. It is important to recall that while mirror-prox methods require two SVDs per iteration, they have the benefit of a fast $O(1/t)$ convergence rate, while simpler saddle-point methods such as mirror-descent-based, only guarantee a $O(1/\sqrt{t})$ rate \cite{sebastienMD}.
Our contributions can be summarized  as follows.
\paragraph*{Contributions (informally stated):}
\begin{itemize}

\item
We prove that even under (standard) strict complementarity, the projected subgradient method, when initialized with a ``warm-start", may produce iterates with rank higher than that of the optimal solution. This phenomenon  motivates our saddle-point approach. See Section \ref{sec:subgradfail}. 

\item
We suggest a generalized strict complementarity (GSC) condition for saddle-point problems and prove that when $g(\cdot)$ --- the objective function in Problem \eqref{nonSmoothProblem}, admits a highly popular saddle-point structure (one which captures all applications we mentioned in this paper), GSC w.r.t. an optimal solution to  Problem  \eqref{nonSmoothProblem} implies GSC (with the same parameters) w.r.t. a corresponding optimal solution of the equivalent saddle-point problem (the other direction always holds). See Section \ref{sec:smooth2Saddle}.

\item
\textbf{Main result:} we prove that for a smooth convex-concave saddle-point problem and an optimal solution which satisfies SC, mirror-prox methods \cite{extragradientK,Nemirovski} such as the Euclidean extragradient method and mirror-prox with approximated MEG updates, when initialized with a ``warm-start" (i.e., a point within a certain distance from the saddle-point), converge with their original rate of $O(1/t)$ while requiring only two low-rank SVDs per iteration. Moreover, for the Euclidean extragradient method we prove that a weaker GSC assumption is sufficient, and that GSC facilitates a precise and powerful tradeoff: increasing the rank of SVD computations (beyond the rank of the optimal solution) can significantly increase the radius of the ball in which the method needs to be initialized. See Theorem \ref{theroem:allPutTogether} and Theorem \ref{thm:allPutTogetherVonNeumann}. 

\item
From a more practical point of view, we discuss simple and efficient procedures for verifying that I. for the Euclidean extragradient method our low-rank updates indeed coincide with the \textit{exact} (full-rank SVD based)  updates, and II.  for the mirror-prox method with MEG updates our low-rank updates provide sufficient approximation for the exact (full-rank SVD based) MEG updates. These procedures could  be used to certify the efficient convergence of the methods during runtime. See Sections \ref{sec:certificate} and  \ref{sec:certificateMP}.

\item We present extensive numerical evidence that demonstrate both the plausibility of the strict complementarity assumption in various nonsmooth low-rank matrix recovery tasks, and  the efficient convergence of our proposed  mirrox-prox methods. See Section \ref{sec:expr}.
\end{itemize}

As already mentioned earlier, the mirror-prox method with MEG updates demonstrates better theoretical dependence on the Lipschitz parameters of the gradient, as these are measured with respect to the spectral norm, in contrast to the Frobenius norm in the Euclidean method (see Section \ref{sec:typesOfBregDistances} and further discussions in \cite{Nemirovski, garber2020efficient}). On the other-hand, as it shall be evident in the sequel, the implementation of the mirror-prox method with low-rank MEG updates is somewhat more complex than the Euclidean method, as it involves explicitly controlling certain approximation errors in the computation of its update steps.

\subsection{Additional related work}\label{sec:addrelatedwork}
Since, as in the works  \cite{garberLowRankSmooth, garberStochasticLowRank} mentioned before which deal with smooth objectives,  strict complementarity plays a key role in our analysis, we refer the interested reader to the recent works \cite{garber2019linear, spectralFrankWolfe, ding2020kfw} which also exploit this property for efficient smooth and convex optimization over the spectrahedron.  In \cite{relatedWork4} the authors show that strict complementarity is sufficient for bounding the measurement and optimization errors caused by noisy data in semidefinite programs.
In \cite{relatedWork2} the authors suggest a primal-dual method to solve certain semidefinite programs with optimal storage using low-rank updates under strict complementarity.
Strict complementarity has also played an instrumental role in two recent and very influential works which used it to prove linear convergence rates for proximal gradient methods \cite{zhou2017unified, drusvyatskiy2018error}. 
It is also a standard generic assumption that is assumed in many constrained optimization problems, see for instance \cite{strictComplementarity}.

For the smooth setting there have been many more advances in developing efficient methods for solving smooth low-rank matrix optimization problems.
In \cite{relatedWork5} the authors suggest a conditional gradient method with minimal storage requirements by using a matrix sketching technique. Also the work \cite{relatedWork6} uses a matrix sketching technique together with a primal-dual method which requires storing only low-rank matrices in memory to solve semidefinite programming problems.
In \cite{relatedWork1} the authors proposed an efficient method for certain nonsmooth low-rank spectral optimization problems via gauge duality which only uses low-rank updates.

Besides convex relaxations such as Problem \eqref{nonSmoothProblem}, considerable advances have been made in the past several yeas in developing efficient \textit{nonconvex} methods with global convergence guarantees for low-rank matrix problems. 
In \cite{nonConvexFactorizedAssumptions} the authors consider semidefinite programs and prove that under a smooth manifold assumption on the constraints, such methods converge to the optimal global solution.
In \cite{nonconvexFactorizedGlobal} the authors prove global convergence of factorized nonconvex gradient descent from a ``warm-start'' initialization point for non-linear smooth minimization on the positive semidefinite cone.  Very recently, \cite{charisopoulos2021low} has established, under statistical conditions, fast convergence results from ``warm-start'' initialization of nonconvex first-order methods, when applied to nonsmooth nonconvex matrix recovery problems which are based on the explicit factorization of the low-rank matrix. A result of similar flavor concerning nonsmooth and nonconvex formulation of robust recovery of low-rank matrices from random linear measurements was presented in \cite{li2020nonconvex}. In \cite{relatedWork7}, under certain geometric assumptions, the authors propose a subgradient method with fast convergence for a nonsmooth formulation of certain low-rank matrix recovery problems.
Several recent works have also considered nonconvex low-rank regularizers which result in nonconvex nonsmooth optimization problems, but guarantee convergence only to a stationary point \cite{lu2014generalized, yao2018large}. 
The work \cite{relatedWork8} provides an overview of the statistical models that enable
efficient nonconvex optimization with provable global convergence.

\subsection{Organization of this paper}
In Section \ref{sec:strictComp} we present the strict complementarity condition for Problem \eqref{nonSmoothProblem} and its relaxed version (generalized strict complementarity) which underly all of our results. In this section we also provide a negative result:  strict complementarity alone is not sufficient to guarantee low-rank iterates for the projected subgradient descent method. This observation motivates our investigation of saddle-point formulations of Problem \eqref{nonSmoothProblem}, which is discussed in Section \ref{sec:smooth2Saddle}. Section \ref{sec:approximatedMP} introduces the mirror-prox method for saddle-point optimization and gives a template for constructing \textit{approximate mirror-prox} algorithms, which will be used to construct our low-rank mirror-prox methods. Section \ref{sec:EucledianCase} introduces our first main result: the Euclidean extragradient method with low-rank updates and its convergence guarantees. Section \ref{sec:MEG} introduces our second main result: a mirror-prox method with low-rank matrix exponentiated gradient updates and its convergence guarantees. Finally, Section \ref{sec:expr} brings empirical evidence in support of our our theoretical findings. For ease of presentation some of the proofs are deferred to the appendix.

\section{Strict Complementarity for Nonsmooth Optimization and Difficulty of Using Low-Rank Projected Subgradient Steps}
\label{sec:strictComp}

The analysis of the nonsmooth Problem \eqref{nonSmoothProblem} naturally depends on  certain subgradients of an optimal solution which, in many aspects, behave like the gradients of smooth functions. The existence of such a subgradient is guaranteed from the first-order optimality condition for constrained convex minimization problems:
\begin{lemma}[first-order optimality condition, see \cite{beckOptimizationBook}] \label{optCondition}
Let $g:\mathbb{S}^n\rightarrow\reals$ be a convex function. Then $\X^*\in\Sn$ minimizes $g$ over $\Sn$ if and only if there exists a subgradient $\G^*\in\partial g(\X^*)$ such that $\langle \X-\X^*,\G^*\rangle\ge0$ for all $\X\in\Sn$.
\end{lemma}

For some $\G^*\in\partial g(\X^*)$ which satisfies the first-order optimality condition for an optimal solution $\X^*$, if the multiplicity of the smallest eigenvalue equals $r^*=\rank(\X^*)$, then it can be shown that the optimal solution satisfies a strict complementarity assumption. The equivalence between a standard strict complementarity assumption of some low-rank optimal solution of a smooth optimization problem over the spectrahedron and an eigen-gap in the gradient of the optimal solution was established in \cite{spectralFrankWolfe}. We generalize this equivalence to also include nonsmooth problems. The proof follows similar arguments and is given in Appendix \ref{sec:appendix:proofLemma2}. 
Throughout the paper we let $\lambda_i(\X)$ denote the $i$th largest eigenvalue of a symmetric real matrix $\X\in\mbS^n$.

\begin{definition}[strict complementarity] An optimal solution $\X^*\in\Sn$ of rank $r^*$ for Problem \eqref{nonSmoothProblem} satisfies the strict complementarity assumption with parameter $\delta >0$, if there exists an optimal solution of the dual problem\footnote{Denote $q(\Z,s)=\min_{\X\in\mathbb{S}^n} \lbrace g(\X)+s(1-\trace(\X))-\langle\Z,\X\rangle\rbrace$.
The dual problem to Problem \eqref{nonSmoothProblem} can be written as:
$ \max_{\lbrace\Z\succeq0,\ s\in\reals\rbrace} \lbrace q(\Z,s) \ \vert \ (\Z,s)\in \textrm{dom}(q)\rbrace$.}
 $(\Z^*,s^*)\in\mbS^n\times\reals$ such that $\rank(\Z^*)=n-r^*$, and $\lambda_{n-r^*}(\Z^*)\ge\delta$. 
\end{definition}

\begin{lemma} \label{lemma:kkt}
Let $\X^*\in\Sn$ be a rank-$r^*$ optimal solution to Problem \eqref{nonSmoothProblem}. $\X^*$ satisfies the (standard) strict complementarity assumption with parameter $\delta>0$, if and only if there exists a subgradient $\G^*\in\partial g(\X^*)$ such that $\langle \X-\X^*,\G^*\rangle\ge0$ for all $\X\in\mathcal{S}_n$ and $\lambda_{n-r^*}(\G^*)-\lambda_{n}(\G^*)\ge\delta$.
\end{lemma}

We also consider a weaker and more general assumption than strict complementarity. Namely, we assume generalized strict complementarity (GSC) holds with respect to at least one subgradient of the optimal solution for which the first-order optimality condition holds.

\begin{assumption}[generalized strict complementarity]\label{ass:strictcompNonSmooth}
We say an optimal solution $\X^*$ to Problem \eqref{nonSmoothProblem} satisfies the generalized strict complementarity assumption with parameters $r,\delta$, if there exists a subgradient $\G^*\in\partial g(\X^*)$ such that $\langle \X-\X^*,\G^*\rangle\ge0$ for all $\X\in\mathcal{S}_n$ and $\lambda_{n-r}(\G^*)-\lambda_{n}(\G^*) \geq \delta$.
\end{assumption}

In \cite{garberLowRankSmooth} the author presents several characteristic properties of the gradient of the  optimal solution in optimization problems over the spectrahedron. 
Using the existence of subgradients which satisfy the condition in Lemma \ref{optCondition}, we can extend these properties also to the nonsmooth setting.
The following lemma shows that GSC with parameters $(r,\delta)$ for some $\delta>0$ (Assumption \ref{ass:strictcompNonSmooth}) is a sufficient condition for the optimal solution to be of rank at most $r$. The proof follows immediately from the proof of the analogous Lemma 7 in \cite{garberLowRankSmooth}, by replacing the gradient of the optimal solution with a subgradient for which the first-order optimality condition holds. This lemma can also be viewed as a special case of the alignment property proposed in \cite{relatedWork3} (see example 4.9).
\begin{lemma} \label{lemma:Xopt-subgradEigen}
Let $\X^*$ be an optimal solution to Problem \eqref{nonSmoothProblem} and write its eigen-decomposition as $\X^*=\sum_{i=1}^{r^*}{\lambda_i\v_i\v_i^T}$. Then, any subgradient $\G^*\in\partial g(\X^*)$ which satisfies $\langle \X-\X^*,\G^*\rangle\ge0$ for all $\X\in\mathcal{S}_n$, admits an eigen-decomposition such that the set of vectors $\{\v_i\}_{i=1}^{r^*}$ is a set of leading eigenvectors of $(-\G^*)$ which corresponds to the eigenvalue $\lambda_1(-\G^*)=-\lambda_n(\G^*)$. Furthermore, there exists at least one such subgradient. 
\end{lemma}

\subsection{The challenge of applying low-rank projected subgradient steps}\label{sec:subgradfail}

Projected subgradient descent is the simplest and perhaps most general approach to solve nonsmooth optimization problems given by a first-order oracle. However, we now demonstrate the difficulty of replacing the full-rank SVD computations required in projected subgradient steps over the spectrahedron, with their low-rank SVD counterparts when attempting to solve Problem \eqref{nonSmoothProblem}. This motivates our saddle-point approach which we will present in Section \ref{sec:smooth2Saddle}. We prove that a projected subgradient descent step from a point arbitrarily close to a low-rank optimal solution --- even one that satisfies strict complementarity, may result in a higher rank matrix.  The problem on which we demonstrate this phenomenon is a well known convex formulation of the \textit{sparse PCA} problem \cite{d2007direct}. The proof is given in Appendix \ref{sec:appendix:proofLemma4}.

\begin{lemma}[failure of subgradient descent with low-rank projections on sparse PCA] \label{lemma:negativeExampleNonSmooth}
Consider the problem
\begin{align*}
\min_{\X\in\Sn}\{g(\X):=-\left\langle\z\z^{\top}+\z_{\perp}\z_{\perp}^{\top},\X\right\rangle+\frac{1}{2k}\Vert\X\Vert_1\},
\end{align*}
where $\z=(1/\sqrt{k},\ldots,1/\sqrt{k},0,\ldots,0)^{\top}$ is supported on the first $k$ entries,  $\z_{\perp}=(0,\ldots,0, 1/\sqrt{n-k},\ldots,1/\sqrt{n-k})^{\top}$ is supported on the last $n-k$ entries, and $k\le n/4$. Then, $\z\z^{\top}$ is a rank-one optimal solution for which  strict complementarity holds.
However, for any $\eta<\frac{2}{3}$ and any $\v\in\reals^n$ such that $\Vert\v\Vert=1$, $\support(\v)\subseteq\support(\z)$, and $\langle\z,\v\rangle^2=1 - \frac{1}{2}\Vert{\v\v^{\top}-\z\z^{\top}}\Vert_F^2 \ge1-\frac{1}{2k^2}$, it holds that
\begin{align*}
\rank\left(\Pi_{\Sn}[\v\v^{\top}-\eta\G_{\v\v^{\top}}]\right)>1,
\end{align*}
where $\G_{\v\v^{\top}}=-\z\z^{\top}-\z_{\perp}\z_{\perp}^{\top}+\frac{1}{2k}\sign(\v\v^{\top})\in\partial g(\v\v^{\top})$.
\end{lemma}
Note that the subgradient of the $\ell_1$-norm which we choose for the projected subgradient step simply corresponds to the sign function, which is arguably the most natural choice.

\section{From Nonsmooth to Saddle-Point Problems}\label{sec:smooth2Saddle}

To circumvent the difficulty demonstrated in Lemma \ref{lemma:negativeExampleNonSmooth} in incorporating low-rank SVDs into standard subgradient methods for solving Problem \eqref{nonSmoothProblem}, we propose tackling the nonsmooth problem with saddle-point methods.

We assume the nonsmooth Problem \eqref{nonSmoothProblem} can be written as a maximum of smooth functions, i.e., $g(\X)=\max_{\y\in\mathcal{K}}f(\X,\y)$, where $\mathcal{K}\subset\mathbb{V}$ is some compact and convex subset of the finite linear space over the reals $\mathbb{V}$ onto which it is efficient to compute projections. We assume $f(\cdot,\y)$ is convex for all $\y\in\mathcal{K}$ and $f(\X,\cdot)$ is concave for all $\X\in\Sn$. That is, we rewrite Problem \eqref{nonSmoothProblem} as the following equivalent saddle-point problem:
\begin{align} \label{problem1}
 \min_{\X\in\mS_n}\max_{\y\in\mathcal{K}}{f(\X,\y)}.
\end{align}

Finding an optimal solution to problem \eqref{problem1} is equivalent to finding a saddle-point $(\X^*,\y^*)\in{\Sn\times\mathcal{K}}$ such that for all $\X\in\Sn$ and $\y\in\mathcal{K}$, 
\begin{align*}
f(\X^*,\y) \le f(\X^*,\y^*) \le f(\X,\y^*).
\end{align*}

We make a standard assumption that $f(\cdot,\cdot)$ is smooth with respect to all the components. That is, we assume there exist $\beta_X,\beta_y,\beta_{Xy},\beta_{yX}\ge0$ such that for any $\X,\tilde{\X}\in\Sn$ and $\y,\tilde{\y}\in\mathcal{K}$ the following four inequalities hold:
\begin{align} \label{betas}
& \Vert\nabla_{\X}f(\X,\y)-\nabla_{\X}f(\tilde{\X},\y)\Vert_{\mathcal{X^*}} \le\beta_{X}\Vert\X-\tilde{\X}\Vert_{\mathcal{X}}, \nonumber \\
&\Vert\nabla_{\y}f(\X,\y)-\nabla_{\y}f(\X,\tilde{\y})\Vert_{\mathcal{Y^*}} \le\beta_{y}\Vert\y-\tilde{\y}\Vert_{\mathcal{Y}}, \nonumber \\
& \Vert\nabla_{\X}f(\X,\y)-\nabla_{\X}f(\X,\tilde{\y})\Vert_{\mathcal{X^*}} \le\beta_{Xy}\Vert\y-\tilde{\y}\Vert_{\mathcal{Y}}, \nonumber \\
& \Vert\nabla_{\y}f(\X,\y)-\nabla_{\y}f(\tilde{\X},\y)\Vert_{\mathcal{Y^*}} \le\beta_{yX}\Vert\X-\tilde{\X}\Vert_{\mathcal{X}},
\end{align}
where $\nabla_{\X}f=\frac{\partial f}{\partial \X}$, $\nabla_{\y}f=\frac{\partial f}{\partial \y}$, $\Vert\cdot\Vert_{\mathcal{X}}$ is a norm equipped upon $\mathbb{S}^n$ and $\Vert\cdot\Vert_{\mathcal{X^*}}$ is its dual norm, and $\Vert\cdot\Vert_{\mathcal{Y}}$ is a norm equipped upon $\mathbb{V}$ and $\Vert\cdot\Vert_{\mathcal{Y^*}}$ is its dual norm. For a pair of variables $(\X,\y)\in\mathbb{S}^n\times\mathbb{V}$ we denote by $\Vert{(\X,\y)}\Vert$ the Euclidean norm over the product space $\mathbb{S}^n\times\mathbb{V}$.

The following lemma highlights a connection between the gradient of a saddle-point of \eqref{problem1} and subgradients of an optimal solution to \eqref{nonSmoothProblem} for which the first order optimality condition holds. One of the connections we will be interested in, is that GSC for Problem \eqref{nonSmoothProblem} implies GSC (with the same parameters) for Problem \eqref{problem1}. However, to prove this specific connection we require an additional structural assumption on the objective function $g(\cdot)$. We note that this assumption holds for all applications mentioned in this paper.
\begin{assumption}\label{ass:struct}
$g(\X)$ is of the form $g(\X) = h(\X) + \max_{\y\in\mK}\y^{\top}(\mA(\X)-\b)$, where $h(\cdot)$ is smooth and convex, and $\mA$ is a linear map.
\end{assumption}

\begin{lemma} \label{lemma:connectionSubgradientNonSmoothAndSaddlePoint}
If $(\X^*,\y^*)$ is a saddle-point of Problem \eqref{problem1} then $\X^*$ is an optimal solution to Problem \eqref{nonSmoothProblem},  $\nabla_{\X}f(\X^*,\y^*)\in\partial g(\X^*)$,  and for all $\X\in\Sn$ it holds that $\langle\X-\X^*,\nabla_{\X}f(\X^*,\y^*)\rangle\ge0$.
Conversely, under Assumption \ref{ass:struct}, if $\X^*$ is an optimal solution to Problem \eqref{nonSmoothProblem}, and $\G^*\in\partial{}g(\X^*)$ which satisfies $\langle\X-\X^*,\G^*\rangle\ge0$  for all $\X\in\Sn$, then there exists $\y^*\in\argmax_{\y\in\mathcal{K}}f(\X^*,\y)$ such that $(\X^*,\y^*)$ is a saddle-point of Problem \eqref{problem1}, and $\nabla_{\X}f(\X^*,\y^*)=\G^*$.
\end{lemma}

The proof is given in Appendix \ref{sec:proofSubgradEquiv}. 
The connection between the gradient of an optimal solution to the saddle-point problem and a subgradient of a corresponding optimal solution in the equivalent nonsmooth problem established in Lemma \ref{lemma:connectionSubgradientNonSmoothAndSaddlePoint}, naturally
leads to the formulation of the following generalized strict complementarity assumption for saddle-point problems.

\begin{assumption}[(generalized) strict complementarity for saddle-points]\label{ass:strictcompSaddlePoint}
We say a saddle-point $(\X^*,\y^*)\in\Sn\times\mathcal{K}$ of Problem \eqref{problem1} with $\rank(\X^*)=r^*$ satisfies the strict complementarity assumption with parameters $\delta>0$, if $\lambda_{n-r^*}(\nabla_{\X}f(\X^*,\y^*))-\lambda_{n}(\nabla_{\X}f(\X^*,\y^*))\ge\delta$.
Moreover, we say $(\X^*,\y^*)$ satisfies the generalized strict complementarity assumption with parameters $r\geq r^*, \delta>0$, if $\lambda_{n-r}(\nabla_{\X}f(\X^*,\y^*))-\lambda_{n}(\nabla_{\X}f(\X^*,\y^*))\ge\delta$.
\end{assumption}

\begin{remark}
Note that under Assumption \ref{ass:struct}, due to Lemma \ref{lemma:connectionSubgradientNonSmoothAndSaddlePoint},  GSC with parameters $r,\delta$ for some optimal solution $\X^*$ to Problem \eqref{nonSmoothProblem} implies GSC with parameters $r,\delta$ to a corresponding saddle-point $(\X^*,\y^*)$ of Problem \eqref{problem1}. Nevertheless, Assumption \ref{ass:struct} is not necessary for proving our convergence results for Problem \eqref{problem1}, which are directly stated in terms of Assumption \ref{ass:strictcompSaddlePoint}. 
\end{remark}

\section{Approximated Mirror-Prox  for Saddle-Point Problems} \label{sec:approximatedMP}
In this section we present a general template for deriving approximate mirror-prox methods for saddle-point problems over the spectrahedron, and present its convergence analysis. We first give preliminaries on the mirror-prox method with a particular focus on the Euclidean setup which gives rise to the extragradient method, and on the von Neumann entropy setup which gives rise to mirror-prox with matrix exponentiated gradient updates.

\subsection{Bregman distances and mirror-prox methods}

\begin{definition}[Bregman distance]
Let $\omega$ be a real-valued, proper and strongly convex function over a nonempty, closed and convex subset of the parameter domain that is continuously differentiable. 
Then the Bregman distance is defined by: $B_{\omega}(\x,\y)=\omega(\x)-\omega(\y)-\left<\x-\y,\nabla\omega(\y)\right>$.
\end{definition}

Let $\omega_{\X}$ be a strongly convex function over a subset of $\Sn$ with respect to $\Vert\cdot\Vert_{\mathcal{X}}$ and $\omega_{\y}$ be a strongly convex function over a subset of $\mathcal{K}$ with respect to $\Vert\cdot\Vert_{\mathcal{Y}}$. We denote $\breg_{\X}(\X,\widetilde{\X}):=\breg_{\omega_{\X}}(\X,\widetilde{\X})$ to be the bregman distance corresponding to $\omega_{\X}$ and $\breg_{\y}(\y,\widetilde{\y}):=\breg_{\omega_{\y}}(\y,\widetilde{\y})$ to be the bregman distance corresponding to $\omega_{\y}$.

Without loss of generality, throughout this paper we assume that the strong convexity parameters are all equal to 1. That is, for all $\X,\widetilde{\X}\in\Sn\cap\textrm{dom}(\omega_{\X})$ and all $\y,\widetilde{\y}\in\mathcal{K}\cap\textrm{dom}(\omega_{\y})$,
\begin{align} \label{ineq:strongConvexityOfBregmanDistance}
\breg_{\X}(\X,\widetilde{\X})\ge\frac{1}{2}\Vert \X-\widetilde{\X}\Vert_{\mathcal{X}}^2,\quad \breg_{\y}(\y,\widetilde{\y})\ge\frac{1}{2}\Vert \y-\widetilde{\y}\Vert_{\mathcal{Y}}^2.
\end{align}

An important property of bregman distances is the three point identity (see \cite{beckMD}):
\begin{align} \label{lemma:threePointLemma}
\breg_{\omega}(\x,\y)+\breg_{\omega}(\y,\z)-\breg_{\omega}(\x,\z)=\langle \nabla\omega(\z)-\nabla\omega(\y),\x-\y\rangle.
\end{align}

A well-known method that can be used to solve saddle-point optimization problems as in \eqref{problem1} with convergence rate $\mathcal{O}(1/t)$ is the mirror-prox method \cite{Nemirovski}. This method is brought in Algorithm \ref{alg:MP}.
\begin{algorithm}[H]
	\caption{Mirror-prox for saddle-point problems (see also \cite{Nemirovski})}\label{alg:MP}
	\begin{algorithmic}
		\State \textbf{Input:} sequence of step-sizes $\{\eta_t\}_{t\geq 1}$ 
		\State \textbf{Initialization:} $(\X_1,\y_1)\in\Sn\cap\textrm{dom}(\omega_{\X})\times\mathcal{K}\cap\textrm{dom}(\omega_{\y})$
		\For{$t = 1,2,...$} 
            \State $\Z_{t+1}=\argmin_{\X\in\Sn}\lbrace\langle\eta_t\nabla_{\X}f(\X_t,\y_t),\X\rangle+\breg_{\X}(\X,\X_t)\rbrace$   
			\State $\w_{t+1}=\argmin_{\y\in\mathcal{K}}\lbrace\langle-\eta_t\nabla_{\y}f(\X_t,\y_t),\y\rangle+\breg_{\y}(\y,\y_t)\rbrace$
			\State $\X_{t+1}=\argmin_{\X\in\Sn}\lbrace\langle\eta_t\nabla_{\X}f(\Z_{t+1},\w_{t+1}),\X\rangle+\breg_{\X}(\X,\X_t)\rbrace$   
			\State $\y_{t+1}=\argmin_{\y\in\mathcal{K}}\lbrace\langle-\eta_t\nabla_{\y}f(\Z_{t+1},\w_{t+1}),\y\rangle+\breg_{\y}(\y,\y_t)\rbrace$
        \EndFor
	\end{algorithmic}
\end{algorithm}

\subsection{Bregman distances for the spectrahedron}
\label{sec:typesOfBregDistances}

There are two popular bregman distances that can be used in optimization problems over the spectrahedron. 
\paragraph{Euclidean distance:} Define
\begin{align*}
\omega_{\X}(\X) = \frac{1}{2}\Vert\X\Vert_F^2,
\end{align*}
which is a 1-strongly convex function with respect to the Frobenius norm. By equipping $\mathbb{S}^n$ with the Frobenius norm, the primal Lipschitz parameters of the gradient in \eqref{betas}, $\beta_X,\beta_{Xy}$, are also measured w.r.t the Frobenius norm as it is a self dual norm, i.e., the inequalities in \eqref{betas} can be written as
\begin{align*} 
& \Vert\nabla_{\X}f(\X,\y)-\nabla_{\X}f(\tilde{\X},\y)\Vert_{F} \le\beta_{X}\Vert\X-\tilde{\X}\Vert_{F}, \nonumber \\
&\Vert\nabla_{\y}f(\X,\y)-\nabla_{\y}f(\X,\tilde{\y})\Vert_{2} \le\beta_{y}\Vert\y-\tilde{\y}\Vert_{2}, \nonumber \\
& \Vert\nabla_{\X}f(\X,\y)-\nabla_{\X}f(\X,\tilde{\y})\Vert_{F} \le\beta_{Xy}\Vert\y-\tilde{\y}\Vert_{2}, \nonumber \\
& \Vert\nabla_{\y}f(\X,\y)-\nabla_{\y}f(\tilde{\X},\y)\Vert_{2} \le\beta_{yX}\Vert\X-\tilde{\X}\Vert_{F}.
\end{align*}
The Euclidean distance which corresponds to $\omega_{\X}(\X)$ is defined as
\begin{align*}
\breg_{\X}(\X,\widetilde{\X})=\frac{1}{2}\Vert\X-\widetilde{\X}\Vert_F^2.
\end{align*}
Using a Euclidean distance, the primal update step of the mirror-prox method in Algorithm \ref{alg:MP} for some $\widetilde{\X}\in\Sn$ and $\nabla_{\X}\in\mathbb{S}^n$, can be written as:
\begin{align*}
\X_+=\argmin_{\X\in\Sn}\lbrace\langle\eta\nabla_{\X},\X\rangle+\frac{1}{2}\Vert\X-\widetilde{\X}\Vert_F^2\rbrace = \Pi_{\mathcal{S}_n}[\widetilde{\X}-\eta\nabla_{\X}].
\end{align*}

\paragraph{Bregman distance corresponding to the von Neumann entropy:}
The von Neumann entropy is a 1-strongly convex function with respect to the nuclear norm $\Vert\cdot\Vert_*$ over the spectrahedron and is defined as:
\begin{align*}
\omega_{\X}(\X) = \trace(\X\log(\X)-\X),
\end{align*}
where the convention $0\log(0):=0$ is used to include also positive semidefinite matrices.
By equipping $\mathbb{S}^n$ with the nuclear norm, the primal Lipschitz parameters of the gradient in \eqref{betas}, $\beta_X,\beta_{Xy}$, are measured w.r.t the spectral norm, which is the dual norm of the nuclear norm. 
Therefore, the inequalities in \eqref{betas} can be written as
\begin{align*} 
& \Vert\nabla_{\X}f(\X,\y)-\nabla_{\X}f(\tilde{\X},\y)\Vert_{2} \le\beta_{X}\Vert\X-\tilde{\X}\Vert_{*}, \nonumber \\
&\Vert\nabla_{\y}f(\X,\y)-\nabla_{\y}f(\X,\tilde{\y})\Vert_{2} \le\beta_{y}\Vert\y-\tilde{\y}\Vert_{2}, \nonumber \\
& \Vert\nabla_{\X}f(\X,\y)-\nabla_{\X}f(\X,\tilde{\y})\Vert_{2} \le\beta_{Xy}\Vert\y-\tilde{\y}\Vert_{2}, \nonumber \\
& \Vert\nabla_{\y}f(\X,\y)-\nabla_{\y}f(\tilde{\X},\y)\Vert_{2} \le\beta_{yX}\Vert\X-\tilde{\X}\Vert_{*}.
\end{align*}
The corresponding bregman distance over $\Sn$ is defined as
\begin{align} \label{eq:bregmanVNE}
\breg_{\X}(\X,\widetilde{\X})=\trace(\X\log(\X)-\X\log(\widetilde{\X})).
\end{align}
Using this distance, 
the primal update step of the mirror-prox method in Algorithm \ref{alg:MP} for some $\widetilde{\X}\in\Sn$ and $\nabla_{\X}\in\mathbb{S}^n$, can be written as (see Section 3 in \cite{BregmanMatrices}):
\begin{align*}
\X_+=\argmin_{\X\in\Sn}\lbrace\langle\eta\nabla_{\X},\X\rangle+\breg_{\X}(\X,\widetilde{\X})\rbrace = \frac{1}{\tau}\exp(\log(\widetilde{\X})-\eta\nabla_{\X}),
\end{align*}
where $\tau:=\trace(\exp(\log(\widetilde{\X})-\eta\nabla_{\X}))$.

\subsection{Approximated mirror-prox method}
The updates of the primal variables $\X_{t+1}$ and $\Z_{t+1}$ in the mirror-prox method using either Euclidean updates as in \eqref{eq:euclidProj}, or MEG updates as in \eqref{exponentiatedGradientUpdate}, require in worst-case a full-rank SVD computation. To avoid these expensive steps, we consider replacing these exact updates in Algorithm \ref{alg:MP}, with certain approximations. In Algorithm \ref{alg:ApproxMP} we introduce a template for such approximated mirror-prox methods and in the sequel we will derive concrete methods based on this template.
We denote by $\widehat{\X}_{t+1}$ and $\widehat{\Z}_{t+1}$ the exact primal mirror-prox updates on iteration $t$, and we denote by $\X_{t+1}$ and $\Z_{t+1}$  the corresponding approximations of $\widehat{\X}_{t+1}, \widehat{\Z}_{t+1}$. The dual variables $\y_{t+1}$ and $\w_{t+1}$ are computed exactly, as we assume that their computation w.r.t. the set $\mK$ is computationally efficient, and so approximations are only considered with respect to the primal variables.

\begin{algorithm}[H]
	\caption{Template of approximated mirror-prox  for saddle-point problems}\label{alg:ApproxMP}
	\begin{algorithmic}
		\State \textbf{Input:} sequence of step-sizes $\{\eta_t\}_{t\geq 1}$ 
		\State \textbf{Initialization:} $(\X_1,\y_1)\in\Sn\cap\textrm{dom}(\omega_{\X})\times\mathcal{K}\cap\textrm{dom}(\omega_{\y})$
		\For{$t = 1,2,...$} 
            \State $\widehat{\Z}_{t+1}=\argmin_{\X\in\Sn}\lbrace\langle\eta_t\nabla_{\X}f(\X_t,\y_t),\X\rangle+\breg_{\X}(\X,\X_t)\rbrace$   
            	\State $\Z_{t+1}\in\Sn\cap\textrm{dom}(\omega_{\X})$ \textrm{\{approximation of $\widehat{\Z}_{t+1}$\}}
			\State $\w_{t+1}=\argmin_{\y\in\mathcal{K}}\lbrace\langle-\eta_t\nabla_{\y}f(\X_t,\y_t),\y\rangle+\breg_{\y}(\y,\y_t)\rbrace$
			\State $\widehat{\X}_{t+1}=\argmin_{\X\in\Sn}\lbrace\langle\eta_t\nabla_{\X}f(\Z_{t+1},\w_{t+1}),\X\rangle+\breg_{\X}(\X,\X_t)\rbrace$
			\State $\X_{t+1}\in\Sn\cap\textrm{dom}(\omega_{\X})$  \textrm{\{approximation of $\widehat{\X}_{t+1}$\}}
			\State $\y_{t+1}=\argmin_{\y\in\mathcal{K}}\lbrace\langle-\eta_t\nabla_{\y}f(\Z_{t+1},\w_{t+1}),\y\rangle+\breg_{\y}(\y,\y_t)\rbrace$
        \EndFor
	\end{algorithmic}
\end{algorithm}

In the following lemma we give a generic convergence result for Algorithm \ref{alg:ApproxMP} with any approximated sequences $\lbrace\X_t\rbrace_{t\ge1}$ and $\lbrace\Z_{t}\rbrace_{t\ge2}$. The proof of the lemma follows the general layout used for proving the convergence of the standard mirror-prox method, with the addition of error terms that are introduced due to the approximated sequences. The full proof is given in Appendix \ref{sec:appendix:proofLemma6}.

\begin{lemma} \label{lemma:convergenceApproxMethod}
Let $\lbrace(\X_t,\y_t)\rbrace_{t\ge1}$, $\lbrace(\Z_{t},\w_t)\rbrace_{t\ge2}$, $\lbrace\widehat{\X}_t\rbrace_{t\ge2}$, and $\lbrace\widehat{\Z}_{t}\rbrace_{t\ge2}$ be the sequences generated by Algorithm \ref{alg:ApproxMP} with a fixed step-size $\eta_t=\eta\le\min\left\lbrace\frac{1}{\beta_{X}+\beta_{Xy}},\frac{1}{(1+\theta)(\beta_{X}+\beta_{yX})},\frac{1}{\beta_{y}+\beta_{yX}},\frac{1}{\beta_{y}+\beta_{Xy}}\right\rbrace$, such that $\theta=\bigg\lbrace\begin{array}{ll} 0 & if\ \Z_{t+1}=\widehat{\Z}_{t+1}~\forall t
\\ 1 & otherwise \end{array}$. Then,
\begin{align} \label{ineq:convergenceRate}
& \max_{\y\in\mathcal{K}} f\left(\frac{1}{T}\sum_{t=1}^T \Z_{t+1},\y\right) -  \min_{\X\in\Sn} f\left(\X,\frac{1}{T}\sum_{t=1}^T\w_{t+1}\right) \nonumber
\\ & \le \frac{D^2}{\eta T}
+ \frac{1}{\eta T}\sum_{t=1}^T\max_{\X\in\Sn}\left(\breg_{\X}(\X,\X_{t+1})-\breg_{\X}(\X,\widehat{\X}_{t+1})\right)
 \nonumber
\\ & \ \ \ + \frac{2(\beta_{X}+\beta_{yX})}{T}\sum_{t=1}^T\breg_{\X}(\widehat{\Z}_{t+1},\Z_{t+1})+\frac{\sqrt{2}G}{T}\sum_{t=1}^T\sqrt{\breg_{\X}(\widehat{\Z}_{t+1},\Z_{t+1})},
\end{align}
where $D^2:=\sup_{(\X,\y)\in{\Sn\times\mathcal{K}}}\left(\breg_{\X}(\X,\X_1)+\breg_{\y}(\y,\y_1)\right)$ and  $G=\sup_{(\X,\y)\in{\Sn\times\mathcal{K}}}\Vert\nabla_{\X}f(\X,\y)\Vert_{\mathcal{X^*}}$.
\end{lemma}
As can be seen, if the error terms  $\sum_{t=1}^T\sqrt{\breg_{\X}(\widehat{\Z}_{t+1},\Z_{t+1})}$, $\sum_{t=1}^T\breg_{\X}(\widehat{\Z}_{t+1},\Z_{t+1})$, and $\sum_{t=1}^T\max_{\X\in\Sn}\left(\breg_{\X}(\X,\X_{t+1})-\breg_{\X}(\X,\widehat{\X}_{t+1})\right)$ resulting from the approximated sequences are all bounded by some constant, then the approximated mirror-prox method, Algorithm \ref{alg:ApproxMP}, converges with rate $\mathcal{O}(1/t)$, similarly to the exact mirror-prox method.

\section{Projected Extragradient Method with Low-Rank Projections}
\label{sec:EucledianCase}
In this section we formally state and prove our first main result: the projected extragradient method \cite{extragradientK} for the saddle-point Problem \eqref{problem1}, which is the Euclidean instantiation of mirror-prox, when initialized in the proximity of a saddle-point which satisfies GSC (Assumption \ref{ass:strictcompSaddlePoint}), converges with its original ergodic $O(1/t)$ rate,  while requiring only two low-rank SVD computations per iteration. The algorithm is given as Algorithm \ref{alg:EG}.

\begin{algorithm}[H]
	\caption{Projected extragradient method for saddle-point problems (see also \cite{extragradientK,Nemirovski})}\label{alg:EG}
	\begin{algorithmic}
		\State \textbf{Input:} sequence of step-sizes $\{\eta_t\}_{t\geq 1}$ 
		\State \textbf{Initialization:} $(\X_1,\y_1)\in\Sn\times\mathcal{K}$
		\For{$t = 1,2,...$} 
            \State $\Z_{t+1}=\Pi_{\Sn}[\X_t-\eta_t\nabla_{\X}f(\X_t,\y_t)]$   
			\State $\w_{t+1}=\Pi_{\mathcal{K}}[\y_t+\eta_t\nabla_{\y}f(\X_t,\y_t)]$
			\State $\X_{t+1}=\Pi_{\Sn}[\X_t-\eta_t\nabla_{\X}f(\Z_{t+1},\w_{t+1})]$   
			\State $\y_{t+1}=\Pi_{\mathcal{K}}[\y_t+\eta_t\nabla_{\y}f(\Z_{t+1},\w_{t+1})]$
        \EndFor
	\end{algorithmic}
\end{algorithm}


\begin{theorem}[main theorem] \label{theroem:allPutTogether}
Fix an optimal solution $(\X^*,\y^*)\in{\Sn\times\mathcal{K}}$ to Problem \eqref{problem1}. Let $\tilde{r}$ denote the multiplicity of $\lambda_{n}(\nabla_{\X}f(\X^*,\y^*))$ and for any $r\ge \tilde{r}$ define $\delta(r)=\lambda_{n-r}(\nabla_{\X}f(\X^*,\y^*)-\lambda_{n}(\nabla_{\X}f(\X^*,\y^*)$. Let $\lbrace(\X_t,\y_t)\rbrace_{t\ge1}$ and $\lbrace(\Z_{t},\w_t)\rbrace_{t\ge2}$ be the sequences of iterates generated by Algorithm \ref{alg:EG} with a fixed step-size \[ \eta=\min\Bigg\lbrace\frac{1}{2\sqrt{\beta_X^2+\beta_{yX}^2}},\frac{1}{2\sqrt{\beta_y^2+\beta_{Xy}^2}},\frac{1}{\beta_{X}+\beta_{Xy}},\frac{1}{\beta_y+\beta_{yX}}\Bigg\rbrace, \] 
where $\beta_X,\beta_y,\beta_{Xy},\beta_{yX}\ge0$ are as defined in \eqref{betas}.
Assume the initialization $(\X_1,\y_1)$ satisfies $\Vert(\X_1,\y_1)-(\X^*,\y^*)\Vert\le R_0(r)$, where
\begin{align*}
R_0(r):= \frac{\eta}{(1\hspace{-0.5mm}+\hspace{-0.5mm}\sqrt{2})\left(1\hspace{-0.5mm}+\hspace{-0.5mm}(2\hspace{-0.5mm}+\hspace{-0.5mm}\sqrt{2})\eta\max\lbrace \beta_{X},\beta_{Xy}\rbrace\right)}\max\Bigg\{\frac{\sqrt{\tilde{r}}\delta(r-\tilde{r}+1)}{2},\frac{\delta(r)}{(1+1/\sqrt{\tilde{r}})}\Bigg\}.
\end{align*}
Then, for all $t\ge1$, the projections $\Pi_{\Sn}[\X_t-\eta\nabla_{\X}f(\X_t,\y_t)]$ and $\Pi_{\Sn}[\X_t-\eta\nabla_{\X}f(\Z_{t+1},\w_{t+1})]$ can be replaced with their rank-r truncated counterparts (see \eqref{truncatedProjection}) without changing the sequences $\lbrace(\X_t,\y_t)\rbrace_{t\ge1}$ and $\lbrace(\Z_{t},\w_t)\rbrace_{t\ge2}$, and for any $T\ge0$ it holds that
\begin{align*}
& \max_{\y\in\mathcal{K}} f\left(\frac{1}{T}\sum_{t=1}^T \Z_{t+1},\y\right) -  \min_{\X\in\Sn} f\left(\X,\frac{1}{T}\sum_{t=1}^T\w_{t+1}\right)
\\ & \le \frac{D^2\max{\left\lbrace\sqrt{\beta_X^2+\beta_{yX}^2},\sqrt{\beta_y^2+\beta_{Xy}^2},\frac{1}{2}(\beta_{X}+\beta_{Xy}),\frac{1}{2}(\beta_y+\beta_{yX})\right\rbrace}}{T},
\end{align*} 
where $D:=\sup_{(\X,\y),(\Z,\w)\in{\Sn\times\mathcal{K}}}\Vert(\X,\y)-(\Z,\w)\Vert$.
\end{theorem}

\begin{remark}
Note that Theorem \ref{theroem:allPutTogether} implies that if standard strict complementarity holds for Problem \eqref{problem1}, that is Assumption \ref{ass:strictcompSaddlePoint} holds with $r=r^*=\rank(\X^*)$ and some $\delta>0$, then only rank-$r^*$ SVDs are required  so that Algorithm \ref{alg:EG} converges with the guaranteed  convergence rate of $O(1/t)$, when initialized with a ``warm-start''. Furthermore, by using SVDs of rank $r>r^*$, with moderately higher values of $r$, we can increase the radius of the ball in which Algorithm \ref{alg:EG} needs to be initialized quite significantly.
\end{remark}

Translating the convergence result we obtained for saddle-points back to the original nonsmooth problem, Problem \eqref{nonSmoothProblem}, we obtain the following convergence result. 

\begin{corollary} \label{cor:nonsmoothExtragradient} 
Fix an optimal solution $\X^*\in\Sn$ to Problem \eqref{nonSmoothProblem} and assume Assumption \ref{ass:struct} holds. Let $\G^*\in\partial g(\X^*)$ which satisfies that $\langle\X-\X^*,\G^*\rangle\ge0$  for all $\X\in\Sn$. Let $\tilde{r}$ denote the multiplicity of $\lambda_{n}(\G^*)$ and for any $r\ge \tilde{r}$ define $\delta(r):=\lambda_{n-r}(\G^*)-\lambda_{n}(\G^*)$.
Define $f$ as in Problem \eqref{problem1} and let $\lbrace(\X_t,\y_t)\rbrace_{t\ge1}$ and $\lbrace(\Z_{t},\w_t)\rbrace_{t\ge2}$ be the sequences of iterates generated by Algorithm \ref{alg:EG} where $\eta$ and $R_0(r)$ are as defined in Theorem \ref{theroem:allPutTogether}.
Then, for all $t\ge1$ the projections $\Pi_{\Sn}[\X_t-\eta\nabla_{\X}f(\X_t,\y_t)]$ and $\Pi_{\Sn}[\X_t-\eta\nabla_{\X}f(\Z_{t+1},\w_{t+1})]$ can be replaced with rank-r truncated projections \eqref{truncatedProjection} without changing the sequences $\lbrace(\X_t,\y_t)\rbrace_{t\ge1}$ and $\lbrace(\Z_{t},\w_t)\rbrace_{t\ge2}$, and for any $T\ge0$ it holds that
\begin{align*}
& g\left(\frac{1}{T}\sum_{t=1}^T \Z_{t+1}\right)-g(\X^*) 
\\ & \le \frac{D^2\max{\left\lbrace\sqrt{\beta_X^2+\beta_{yX}^2},\sqrt{\beta_y^2+\beta_{Xy}^2},\frac{1}{2}(\beta_{X}+\beta_{Xy}),\frac{1}{2}(\beta_y+\beta_{yX})\right\rbrace}}{T},
\end{align*}
where $D:=\sup_{(\X,\y),(\Z,\w)\in{\Sn\times\mathcal{K}}}\Vert(\X,\y)-(\Z,\w)\Vert$.
\end{corollary}

\begin{proof}
Since Assumption \ref{ass:struct} holds, invoking Lemma \ref{lemma:connectionSubgradientNonSmoothAndSaddlePoint} we obtain that there exists a point $\y^*\in\argmax_{\y\in\mathcal{K}}f(\X^*,\y)$ such that $(\X^*,\y^*)$ is a saddle-point of Problem \eqref{problem1}, and $\nabla_{\X}f(\X^*,\y^*)=\G^*$.
Therefore, the assumptions of Theorem \ref{theroem:allPutTogether} hold and so by Theorem \ref{theroem:allPutTogether} we get that
\begin{align} \label{ineq:boundFromTheorem1}
& \max_{\y\in\mathcal{K}} f\left(\frac{1}{T}\sum_{t=1}^T \Z_{t+1},\y\right) -  \min_{\X\in\Sn} f\left(\X,\frac{1}{T}\sum_{t=1}^T\w_{t+1}\right) \nonumber
\\ & \le \frac{D^2\max{\left\lbrace\sqrt{\beta_X^2+\beta_{yX}^2},\sqrt{\beta_y^2+\beta_{Xy}^2},\frac{1}{2}(\beta_{X}+\beta_{Xy}),\frac{1}{2}(\beta_y+\beta_{yX})\right\rbrace}}{T}.
\end{align} 
From the definition of $g$ it holds that
\begin{align} \label{ineq:replaceMaxTerm}
g\left(\frac{1}{T}\sum_{t=1}^T \Z_{t+1}\right)=\max_{\y\in\mathcal{K}}f\left(\frac{1}{T}\sum_{t=1}^T \Z_{t+1},\y\right)
\end{align}
and
\begin{align} \label{ineq:replaceMinTerm}
\min_{\X\in\Sn}f\left(\X,\frac{1}{T}\sum_{t=1}^T\w_{t+1}\right) \le f\left(\X^*,\frac{1}{T}\sum_{t=1}^T\w_{t+1}\right) \le \max_{\y\in\mathcal{K}}f\left(\X^*,\y\right)=g(\X^*).
\end{align}
Plugging \eqref{ineq:replaceMaxTerm} and \eqref{ineq:replaceMinTerm} into the LHS of \eqref{ineq:boundFromTheorem1} we obtain the required result. 
\end{proof}

To prove Theorem \ref{theroem:allPutTogether} we first introduce two technical lemmas. We begin by establishing that the iterates of Algorithm \ref{alg:EG} always remain inside a ball of a certain radius around an optimal solution. The following lemma is stated in terms of the full Lipschitz parameter $\beta\ge0$, which satisfies that for all $\X,\widetilde{\X}\in\Sn$ and $\y,\widetilde{\y}\in\mK$ it holds that $\Vert(\nabla_{\X}f(\X,\y),-\nabla_{\y}f(\X,\y))-(\nabla_{\X}f(\tilde{\X},\tilde{\y}),-\nabla_{\y}f(\tilde{\X},\tilde{\y}))\Vert \le\beta\Vert(\X,\Y)-(\tilde{\X},\tilde{\y})\Vert$.
The relationship between $\beta_X$, $\beta_y$, $\beta_{Xy}$, and $\beta_{yX}$ as defined in \eqref{betas} and $\beta$ is given in Appendix \ref{sec:appendix:fullLipschitz}. 
The proof of the lemma is given in Appendix \ref{sec:AppendixLemma9}.

\begin{lemma} \label{lemma:convergenceOfXYZWsequences}
Let $\lbrace(\X_t,\y_t)\rbrace_{t\ge1}$ and $\lbrace(\Z_{t},\w_t)\rbrace_{t\ge2}$ be the sequences generated by Algorithm \ref{alg:EG} with a step-size $\eta_t\le\frac{1}{\beta}$ where $\beta= \sqrt{2}\max\left\lbrace\sqrt{\beta_{X}^2+\beta_{yX}^2},\sqrt{\beta_{y}^2+\beta_{Xy}^2}\right\rbrace$. Then for all $t\ge1$ it holds that 
\begin{align*}
\Vert(\X_{t+1},\y_{t+1})-(\X^*,\y^*)\Vert & \le \Vert(\X_{t},\y_t)-(\X^*,\y^*)\Vert,
\\ \Vert(\Z_{t+1},\w_{t+1})-(\X^*,\y^*)\Vert & \le \Bigg(1+\frac{1}{\sqrt{1-\eta_{t}^2\beta^2}}\Bigg)\Vert(\X_t,\y_t)-(\X^*,\y^*)\Vert.
\end{align*}
\end{lemma}

We now turn to prove the central step in the analysis necessary for proving Theorem \ref{theroem:allPutTogether}. We prove that when close enough to a low-rank saddle-point of Problem \eqref{problem1}, under an assumption of an eigen-gap in the gradient of the saddle-point  (Assumption \ref{ass:strictcompSaddlePoint}), both projections onto the spectrahedron that are necessary in each iteration of Algorithm \ref{alg:EG}, result in low-rank matrices.

\begin{lemma}
\label{lemma:radiusOfLowRankProjections}
Let $(\X^*,\y^*)$ be an optimal solution to Problem \eqref{problem1}. Let $\tilde{r}$ denote the multiplicity of $\lambda_{n}(\nabla_{\X}f(\X^*,\y^*))$ and for any $r\ge \tilde{r}$ denote $\delta(r):=\lambda_{n-r}(\nabla_{\X}f(\X^*,\y^*))-\lambda_{n}(\nabla_{\X}f(\X^*,\y^*))$.
Then, for any $\eta\ge0$ and $(\X,\y)\in\mathcal{S}_n\times\mathcal{K}$, if
\begin{align} \label{ineq:EuclideanRadius}
& \Vert(\X,\y)-(\X^*,\y^*)\Vert \nonumber \\ & \le \frac{\eta}{1+\sqrt{2}\eta\max\lbrace \beta_{X},\beta_{Xy}\rbrace\left(1+\frac{1}{\sqrt{1-\eta^2\beta^2}}\right)}\max\left\lbrace\frac{\sqrt{\tilde{r}}\delta(r-\tilde{r}+1)}{2},\frac{\delta(r)}{(1+1/\sqrt{\tilde{r}})}\right\rbrace
\end{align}
where $\beta= \sqrt{2}\max\left\lbrace\sqrt{\beta_{X}^2+\beta_{yX}^2},\sqrt{\beta_{y}^2+\beta_{Xy}^2}\right\rbrace$, then $\rank\left( \Pi_{\mathcal{S}_n}[\X-\eta \nabla_{\X}f(\X,\y)]\right)\le r$ and $\rank\left( \Pi_{\mathcal{S}_n}[\X-\eta \nabla_{\X}f(\Z_+,\w_+)]\right)\le r$ where $\Z_+=\Pi_{\Sn}[\X-\eta\nabla_{\X}f(\X,\y)]$ and $\w_+=\Pi_{\mathcal{K}}[\y-\eta\nabla_{\y}f(\X,\y)]$. 
\end{lemma}
At a high-level, the proof of the lemma relies on leveraging the special spectral structure of the Euclidean projection onto the spectrahedron, as described in Eq. \eqref{eq:euclidProj}, to argue that if a certain simple condition on the eigenvalues of a matrix to project holds, then the projection has bounded rank. Due to optimality conditions, this condition holds with additional ``slack'' for a matrix obtained by performing a gradient step from any saddle-point satisfying (generalized) strict complementarity (Assumption \ref{ass:strictcompSaddlePoint}). Thus, by performing a perturbation analysis for the eigenvalues, we can argue that for any matrix obtained from a gradient step from a matrix close enough to such a saddle point, this condition also holds, and hence its projection also has  bounded rank. 

\begin{proof}
Denote $\P^*=\X^*-\eta\nabla_{\X}f(\X^*,\y^*)$. By Lemma \ref{lemma:connectionSubgradientNonSmoothAndSaddlePoint}, $\nabla_{\X}f(\X^*,\y^*)$ is a subgradient of the corresponding nonsmooth objective $g(\X)=\max_{\y\in\mathcal{K}}f(\X,\y)$ at the point $\X^*$. Moreover, this subgradient also satisfies the first-order optimality condition. Hence, invoking Lemma \ref{lemma:Xopt-subgradEigen} with this subgradient  we have that
\begin{align} \label{eq:eigsOfPstar}
 \forall i\le \rank(\X^*) & :\ \lambda_i(\P^*) = \lambda_i(\X^*)-\eta\lambda_n(\nabla_{\X}f(\X^*,\y^*)); \nonumber \\
\forall i>\rank(\X^*) & :\ \lambda_i(\P^*) = -\eta\lambda_{n-i+1}(\nabla_{\X}f(\X^*,\y^*)). 
\end{align}

Therefore, using \eqref{eq:eigsOfPstar} and the fact that $\lambda_{n-i+1}(\nabla{}f_{\X}(\X^*,\y^*)) = \lambda_{n}(\nabla{}f_{\X}(\X^*,\y^*))$ for all $i\leq \tilde{r}$ we have,
\begin{align} \label{ineq:sumOfXeigsInProof}
\sum_{i=1}^{\tilde{r}} {\lambda_i(\P^*)} & = \sum_{i=1}^{\tilde{r}} {\lambda_i(\X^*-\eta\nabla_{\X}f(\X^*,\y^*))}
= \sum_{i=1}^{\tilde{r}} \lambda_i(\X^*) -\eta\sum_{i=1}^{\tilde{r}} \lambda_{n-i+1}(\nabla_{\X}f(\X^*,\y^*)) \nonumber
\\ & = \sum_{i=1}^{\rank(\X^*)}\lambda_i(\X^*) -\eta \sum_{i=1}^{\tilde{r}}\lambda_n(\nabla_{\X}f(\X^*,\y^*)) = 1 -\eta \tilde{r}\lambda_n(\nabla_{\X}f(\X^*,\y^*)).
\end{align}

Let $\P\in\mathbb{S}^n$.  From the structure of the Euclidean projection onto the spectrahedron (see Eq. \eqref{eq:euclidProj}), it follows that
a sufficient condition so that $\rank\left( \Pi_{\mathcal{S}_n}[\P]\right)\le r$ is that $\sum_{i=1}^{r} {\lambda_i(\P)}-r\lambda_{r+1}(\P) \ge 1$. We will bound the LHS of this inequality.

First, it holds that
\begin{align} \label{ineq:sumTildeRbound}
\sum_{i=1}^{\tilde{r}} {\lambda_i(\P)} & \underset{(a)}{\ge} \sum_{i=1}^{\tilde{r}} {\lambda_i(\P^*)} - \sum_{i=1}^{\tilde{r}} {\lambda_i(\P^*-\P)} 
\ge \sum_{i=1}^{\tilde{r}} {\lambda_i(\P^*)} - \sqrt{\tilde{r}\sum_{i=1}^{\tilde{r}} {\lambda_i^2(\P-\P^*)}} \nonumber
\\ & \ge \sum_{i=1}^{\tilde{r}} {\lambda_i(\P^*)} - \sqrt{\tilde{r}\sum_{i=1}^{n} {\lambda_i^2(\P-\P^*)}} 
\ge \sum_{i=1}^{\tilde{r}} {\lambda_i(\P^*)} - \sqrt{\tilde{r}}\Vert\P-\P^*\Vert_F \nonumber
\\ & \underset{(b)}{\ge}  1 -\eta \tilde{r}\lambda_n(\nabla_{\X}f(\X^*,\y^*))- \sqrt{\tilde{r}}\Vert\P-\P^*\Vert_F,
\end{align}
where (a) holds from Ky Fan?s inequality for eigenvalues and (b) holds from \eqref{ineq:sumOfXeigsInProof}.

Now, for any $r\ge \tilde{r}$ using Weyl's inequality and \eqref{eq:eigsOfPstar}
\begin{align} \label{ineq:Rplus1bound1}
\lambda_{r+1}(\P) & \le \lambda_{r+1}(\P^*)+\lambda_{1}(\P-\P^*) \le \lambda_{r+1}(\P^*)+\Vert\P-\P^*\Vert_F \nonumber
\\ & = -\eta\lambda_{n-r}(\nabla_{\X}f(\X^*,\y^*))+\Vert\P-\P^*\Vert_F.
\end{align}

Thus, combining \eqref{ineq:sumTildeRbound} and \eqref{ineq:Rplus1bound1} we obtain
\begin{align} \label{ineq:opt1_Pbound}
\sum_{i=1}^{r} {\lambda_i(\P)}-r\lambda_{r+1}(\P)  \ge \sum_{i=1}^{\tilde{r}} {\lambda_i(\P)}-\tilde{r}\lambda_{r+1}(\P)
\ge 1 +\eta \tilde{r}\delta(r) - (\tilde{r}+\sqrt{\tilde{r}})\Vert\P-\P^*\Vert_F.
\end{align} 

Alternatively, if $r\ge2\tilde{r}-1$ then using the general Weyl inequality and \eqref{eq:eigsOfPstar} we obtain
\begin{align} \label{ineq:Rplus1bound2}
\lambda_{r+1}(\P) & \le \lambda_{r-\tilde{r}+2}(\P^*)+\lambda_{\tilde{r}}(\P-\P^*) = \lambda_{r-\tilde{r}+2}(\P^*)+\sqrt{\lambda^2_{\tilde{r}}(\P-\P^*)} \nonumber
\\ & \le \lambda_{r-\tilde{r}+2}(\P^*)+\frac{1}{\sqrt{\tilde{r}}}\Vert\P-\P^*\Vert_F \nonumber
\\ & = -\eta\lambda_{n-r+\tilde{r}-1}(\nabla_{\X}f(\X^*,\y^*))+\frac{1}{\sqrt{\tilde{r}}}\Vert\P-\P^*\Vert_F.
\end{align} 

Thus, combining \eqref{ineq:sumTildeRbound} and \eqref{ineq:Rplus1bound2} we obtain
\begin{align} \label{ineq:opt2_Pbound}
& \sum_{i=1}^{r} {\lambda_i(\P)}-r\lambda_{r+1}(\P)  \ge \sum_{i=1}^{\tilde{r}} {\lambda_i(\P)}-\tilde{r}\lambda_{r+1}(\P)
\ge 1 +\eta \tilde{r}\delta(r-\tilde{r}+1) - 2\sqrt{\tilde{r}}\Vert\P-\P^*\Vert_F.
\end{align} 

Now we are left with bounding $\Vert\P-\P^*\Vert_F$. Note that by the smoothness of $f$, for any $(\X,\y)\in\Sn\times\mathcal{K}$ it holds that
\begin{align} \label{ineq:betaSmoothnessInProof}
& \Vert\nabla_{\X}f(\X,\y)-\nabla_{\X}f(\X^*,\y^*)\Vert_F \nonumber \\ & \le \Vert\nabla_{\X}f(\X,\y)-\nabla_{\X}f(\X^*,\y)\Vert_F+\Vert\nabla_{\X}f(\X^*,\y)-\nabla_{\X}f(\X^*,\y^*)\Vert_F \nonumber
\\ & \le \beta_{X}\Vert\X-\X^*\Vert_F+\beta_{Xy}\Vert\y-\y^*\Vert_2.
\end{align}

Taking $\P=\X-\eta\nabla_{\X}f(\X,\y)$ we get
\begin{align*}
\Vert\P-\P^*\Vert_F & = \Vert\X-\eta\nabla_{\X}f(\X,\y)-\X^*+\eta\nabla_{\X}f(\X^*,\y^*)\Vert_F
\\ & \le \Vert\X-\X^*\Vert_F+\eta\Vert\nabla_{\X}f(\X,\y)-\nabla_{\X}f(\X^*,\y^*)\Vert_F
\\ & \le \Vert(\X,\y)-(\X^*,\y^*)\Vert_F+\eta\Vert\nabla_{\X}f(\X,\y)-\nabla_{\X}f(\X^*,\y^*)\Vert_F
\\ & \le \Vert(\X,\y)-(\X^*,\y^*)\Vert_F+\eta\beta_{X}\Vert\X-\X^*\Vert_F+\eta\beta_{Xy}\Vert\y-\y^*\Vert_2,
\end{align*}
where the last inequality holds from \eqref{ineq:betaSmoothnessInProof}.

For any $a,b\ge0$ it holds that 
\begin{align*}
a\Vert\X-\X^*\Vert_F+b\Vert\y-\y^*\Vert_2 & \le \max\lbrace a,b\rbrace\left(\Vert\X-\X^*\Vert_F+\Vert\y-\y^*\Vert_2\right)
\\ & \le \sqrt{2}\max\lbrace a,b\rbrace \Vert(\X,\y)-(\X^*,\y^*)\Vert.
\end{align*}

Thus, by taking $a=\eta\beta_{X}$ and $b=\eta\beta_{Xy}$ we obtain
\begin{align} \label{ineq:p1NormBound}
\Vert\P-\P^*\Vert_F & \le \left(1+\sqrt{2}\eta\max\lbrace\beta_{X},\beta_{Xy}\rbrace\right)\Vert(\X,\y)-(\X^*,\y^*)\Vert.
\end{align}

Therefore, plugging \eqref{ineq:p1NormBound} into \eqref{ineq:opt1_Pbound} we obtain that the condition $\sum_{i=1}^{r} {\lambda_i(\P)}-r\lambda_{r+1}(\P)\ge1$ holds if
\begin{align*} 
\Vert(\X,\y)-(\X^*,\y^*)\Vert \le \frac{\eta\delta(r)}{(1+1/\sqrt{\tilde{r}})\left(1+\sqrt{2}\eta\max\lbrace\beta_{X},\beta_{Xy}\rbrace\right)}.
\end{align*}

Alternatively, plugging \eqref{ineq:p1NormBound} into \eqref{ineq:opt2_Pbound} we obtain that if $r\ge2\tilde{r}-1$ then the condition $\sum_{i=1}^{r} {\lambda_i(\P)}-r\lambda_{r+1}(\P)\ge1$ holds if
\begin{align*}
\Vert(\X,\y)-(\X^*,\y^*)\Vert \le \frac{\eta\sqrt{\tilde{r}}\delta(r-\tilde{r}+1)}{2\left(1+\sqrt{2}\eta\max\lbrace\beta_{X},\beta_{Xy}\rbrace\right)}.
\end{align*}

Note that $\delta(r-\tilde{r}+1)>0$ only if $r\ge2\tilde{r}-1$. Therefore, we can combine the last two inequalities to conclude that for any $r\ge\tilde{r}$ if
\begin{align} \label{ineq:XseriesRadius}
\Vert(\X,\y)-(\X^*,\y^*)\Vert \le \frac{\eta}{1+\sqrt{2}\eta\max\lbrace\beta_{X},\beta_{Xy}\rbrace}\max\Bigg\{\frac{\sqrt{\tilde{r}}\delta(r-\tilde{r}+1)}{2},\frac{\delta(r)}{(1+1/\sqrt{\tilde{r}})}\Bigg\}
\end{align}
then $\rank(\Pi_{\Sn}[\X-\nabla_{\X}f(\X,\y)])\le r$.

Similarly, taking $\P=\X-\eta\nabla_{\X}f(\Z_+,\w_+)$ we get
\begin{align*} 
\Vert\P-\P^*\Vert_F & = \Vert\X-\eta\nabla_{\X}f(\Z_+,\w_+)-\X^*+\eta\nabla_{\X}f(\X^*,\y^*)\Vert_F 
\\ & \le \Vert\X-\X^*\Vert_F+\eta\Vert\nabla_{\X}f(\Z_+,\w_+)-\nabla_{\X}f(\X^*,\y^*)\Vert_F
\\ & \le \Vert\X-\X^*\Vert_F+\eta\beta_{X}\Vert\Z_+-\X^*\Vert_F+\eta\beta_{Xy}\Vert\w_+-\y^*\Vert_2,
\end{align*}
where the last inequality holds from \eqref{ineq:betaSmoothnessInProof}.

For any $a,b,c\ge0$ it holds that 
\begin{align*}
& a\Vert\X-\X^*\Vert_F+b\Vert\Z_+-\X^*\Vert_F+c\Vert\w_+-\y^*\Vert_2 
\\ & \le a\Vert\X-\X^*\Vert_F+\max\lbrace b,c\rbrace\left(\Vert\Z_+-\X^*\Vert_F+\Vert\w_+-\y^*\Vert_2\right)
\\ & \le a\Vert\X-\X^*\Vert_F+\sqrt{2}\max\lbrace b,c\rbrace \Vert(\Z_+,\w_+)-(\X^*,\y^*)\Vert
\\ & \le a\Vert(\X,\y)-(\X^*,\y^*)\Vert_F+\sqrt{2}\max\lbrace b,c\rbrace \Vert(\Z_+,\w_+)-(\X^*,\y^*)\Vert
\\ & \le \left(a+\sqrt{2}\max\lbrace b,c\rbrace\left(1+\frac{1}{\sqrt{1-\eta^2\beta^2}}\right)\right)\Vert(\X,\y)-(\X^*,\y^*)\Vert,
\end{align*}
where the last inequality holds from Lemma \ref{lemma:convergenceOfXYZWsequences}.

Thus, by taking $a=1$, $b=\eta\beta_{X}$, and $c=\eta\beta_{Xy}$ we obtain
\begin{align} \label{ineq:p2NormBound}
\Vert\P-\P^*\Vert_F & \le \left(1+\sqrt{2}\eta\max\lbrace \beta_{X},\beta_{Xy}\rbrace\left(1+\frac{1}{\sqrt{1-\eta^2\beta^2}}\right)\right)\Vert(\X,\y)-(\X^*,\y^*)\Vert.
\end{align}

Therefore, plugging \eqref{ineq:p2NormBound} into \eqref{ineq:opt1_Pbound} we obtain that the condition $\sum_{i=1}^{r} {\lambda_i(\P)}-r\lambda_{r+1}(\P)\ge1$ holds if
\begin{align*}
\Vert(\X,\y)-(\X^*,\y^*)\Vert \le \frac{\eta\delta(r)}{(1+1/\sqrt{r^*})\left(1+\sqrt{2}\eta\max\lbrace \beta_{X},\beta_{Xy}\rbrace\left(1+\frac{1}{\sqrt{1-\eta^2\beta^2}}\right)\right)}.
\end{align*}

Alternatively, plugging \eqref{ineq:p2NormBound} into \eqref{ineq:opt2_Pbound} we obtain that if $r\ge2\tilde{r}-1$ then the condition $\sum_{i=1}^{r} {\lambda_i(\P)}-r\lambda_{r+1}(\P)\ge1$ holds if
\begin{align*}
\Vert(\X,\y)-(\X^*,\y^*)\Vert \le \frac{\eta\sqrt{\tilde{r}}\delta(r-\tilde{r}+1)}{1+\sqrt{2}\eta\max\lbrace \beta_{X},\beta_{Xy}\rbrace\left(1+\frac{1}{\sqrt{1-\eta^2\beta^2}}\right)}.
\end{align*}

Since $\delta(r-\tilde{r}+1)>0$ only if $r\ge2\tilde{r}-1$, by combining the last two inequalities we conclude that for any $r\ge\tilde{r}$ if
\begin{align} \label{ineq:ZseriesRadius}
& \Vert(\X,\y)-(\X^*,\y^*)\Vert \nonumber
\\ & \le \frac{\eta}{1+\sqrt{2}\eta\max\lbrace \beta_{X},\beta_{Xy}\rbrace\left(1+\frac{1}{\sqrt{1-\eta^2\beta^2}}\right)}\max\left\lbrace\frac{\sqrt{\tilde{r}}\delta(r-\tilde{r}+1)}{2},\frac{\delta(r)}{(1+1/\sqrt{\tilde{r}})}\right\rbrace
\end{align}
then $\rank(\Pi_{\Sn}[\X-\nabla_{\X}f(\Z_+,\w_+)])\le r$.

Taking the minimum between the RHS of \eqref{ineq:XseriesRadius} the RHS of\eqref{ineq:ZseriesRadius} gives us the bound on the radius as in \eqref{ineq:EuclideanRadius}.

\end{proof}

Now we can prove Theorem \ref{theroem:allPutTogether}.
\begin{proof}
We will prove by induction that for all $t\ge 1$ it holds that $\Vert(\X_t,\y_t)-(\X^*,\y^*)\Vert_F \le R_0(r)$ and $\Vert(\Z_{t},\w_{t})-(\X^*,\y^*)\Vert_F \le \left(1+\sqrt{2}\right)R_0(r)$, thus implying through Lemma \ref{lemma:radiusOfLowRankProjections} that all projections $\Pi_{\Sn}[\X_t-\eta\nabla_{\X}f(\X_t,\y_t)]$ and $\Pi_{\Sn}[\X_t-\eta\nabla_{\X}f(\Z_{t+1},\w_{t+1})]$ can be replaced with their rank-r truncated counterparts given in \eqref{truncatedProjection}, without any change to the result. 

The initialization $\Vert(\X_1,\y_1)-(\X^*,\y^*)\Vert\le R_0(r)$ holds trivially.
Now, by Lemma \ref{lemma:convergenceOfXYZWsequences}, using recursion, we have that for all $t\ge1$,
\begin{align*}
\Vert(\X_{t+1},\y_{t+1})-(\X^*,\y^*)\Vert & \le \Vert(\X_{t},\y_t)-(\X^*,\y^*)\Vert 
\\ & \le \cdots \le \Vert(\X_{1},\y_1)-(\X^*,\y^*)\Vert\le R_0(r),
\end{align*}
and for $\beta= \sqrt{2}\max\left\lbrace\sqrt{\beta_{X}^2+\beta_{yX}^2},\sqrt{\beta_{y}^2+\beta_{Xy}^2}\right\rbrace$ we have that,
\begin{align*}
\Vert(\Z_{t+1},\w_{t+1})-(\X^*,\y^*)\Vert & \le \Bigg(1+\frac{1}{\sqrt{1-\eta_{t}^2\beta^2}}\Bigg)\Vert(\X_t,\y_t)-(\X^*,\y^*)\Vert
\\ & \le \Bigg(1+\frac{1}{\sqrt{1-\eta_{t}^2\beta^2}}\Bigg)\Vert(\X_1,\y_1)-(\X^*,\y^*)\Vert
\\ & \le (1+\sqrt{2})\Vert(\X_1,\y_1)-(\X^*,\y^*)\Vert\le (1+\sqrt{2})R_0(r).
\end{align*}

Therefore, under the assumptions of the theorem, Algorithm \ref{alg:EG} can be run using only rank-r truncated projections, while maintaining the original extragradient convergence rate as stated in Lemma \ref{lemma:convergenceApproxMethod} with exact updates (that is, $\X_{t+1}=\widehat{\X}_{t+1}$, $\Z_{t+1}=\widehat{\Z}_{t+1}$, and $\theta=0$).
\end{proof}

\begin{remark}
A downside of considering the saddle-point formulation \eqref{problem1} when attempting to solve Problem \eqref{nonSmoothProblem} that arises from Theorem \ref{theroem:allPutTogether}, is that not only do we need a ``warm-start'' initialization for the original primal matrix variable $\X$, in the saddle-point formulation we need a ``warm-start''  for the primal-dual pair $(\X,\y)$. Nevertheless, as we demonstrate extensively in Section \ref{sec:expr}, it seems that very simple initialization schemes work very well in practice.
\end{remark}

\subsection{Efficiently-computable certificates for correctness of low-rank Euclidean projections}\label{sec:certificate}

Since Theorem \ref{theroem:allPutTogether} only applies in some neighborhood of an optimal solution, it is of interest to have a procedure for verifying if the rank-$r$ truncated projection of a given point indeed equals the exact Euclidean projection. In particular, from a practical point of view, it does not matter whether the conditions of Theorem \ref{theroem:allPutTogether} hold. In practice, as long as the truncated projection equals the exact projection, we are guaranteed that Algorithm \ref{alg:EG} converges correctly with rate $O(1/t)$, without needing to verify any other condition.
Luckily, the expression in \eqref{eq:euclidProj} which characterizes the structure of the Euclidean projection onto the spectrahedron, yields exactly such a verification procedure. As already noted in  \cite{garberLowRankSmooth}, for any $\X\in\mbS^n$, we have $\widehat{\Pi}_{\Sn}^r[\X]=\Pi_{\Sn}[\X]$ if and only if the condition
\begin{align*} 
\sum_{i=1}^r\lambda_i(\X)\ge 1+r\cdot\lambda_{r+1}(\X)
\end{align*}
holds. 
Note that verifying this condition  simply requires increasing the rank of SVD computation by one, i.e., computing a rank-$(r+1)$ SVD of the matrix to project rather than a rank-$r$ SVD.

\section{Mirror-Prox with Low-Rank Matrix Exponentiated Gradient Updates}\label{sec:MEG}
In this section we formally state and prove our second main result: A low-rank approximated mirror-prox method with MEG updates, when initialized in the proximity of a saddle-point which satisfies standard strict complementarity (Assumption \ref{ass:strictcompSaddlePoint} with $r=r^*$), converges with the original $\mathcal{O}(1/t)$ mirror-prox rate, while requiring only two low-rank SVD computations per iteration. 


The standard mirror-prox method with primal matrix exponentiated gradient steps (and generic dual steps) is given in Algorithm \ref{alg:EGMP}.
\begin{algorithm}[H]
	\caption{Mirror-prox with MEG updates for saddle-point problems}\label{alg:EGMP}
	\begin{algorithmic}
	\State \textbf{Input:} $\{\eta_t\}_{t\ge1}$ - sequence of (positive) step-sizes
		\State \textbf{Initialization:} $(\X_1,\y_1)\in\Sn\cap\mathbb{S}^n_{++}\times\mathcal{K}\cap\textrm{dom}(\omega_{\y})$
		\For{$t = 1,2,...$} 
		    \State $\P_t = \exp(\log(\X_t)-\eta_t\nabla_{\X}f(\X_t,\y_t))$
		    \State $\Z_{t+1}=\frac{\P_t}{\trace(\P_t)}$   
			\State $\w_{t+1}=\argmin_{\y\in\mathcal{K}}\lbrace\langle-\eta_t\nabla_{\y}f(\X_t,\y_t),\y\rangle+\breg_{\y}(\y,\y_t)\rbrace$
			\State $\B_t = \exp(\log(\X_t)-\eta_t\nabla_{\X}f(\Z_{t+1},\w_{t+1}))$  
		    \State $\X_{t+1}=\frac{\B_t}{\trace(\B_t)}$     
			\State $\y_{t+1}=\argmin_{\y\in\mathcal{K}}\lbrace\langle-\eta_t\nabla_{\y}f(\Z_{t+1},\w_{t+1}),\y\rangle+\breg_{\y}(\y,\y_t)\rbrace$
        \EndFor
	\end{algorithmic}
\end{algorithm}

The low-rank mirror-prox method with MEG steps is given in Algorithm \ref{alg:LRMP}. As discussed in the Introduction, and following our recent work  \cite{garber2020efficient}, our low-rank variant follows from replacing the exact MEG steps in Algorithm \ref{alg:EGMP}, which are also given in  \eqref{exponentiatedGradientUpdate},  with the corresponding approximations given in \eqref{lowrankexponentiatedGradientUpdate}, which only require a rank-$r$ SVD, and thresholds all lower $n-r$ eigenvalues to the same (small) value.
\begin{algorithm}[H]
	\caption{Low-rank mirror-prox with approximated MEG updates for saddle-point problems}\label{alg:LRMP}
	\begin{algorithmic}
	\State \textbf{Input:} $\{\eta_t\}_{t\ge1}$ - sequence of (positive) step-sizes, $\lbrace\varepsilon_t\rbrace_{t\ge 0}\subset[0,1]$ - sequence of approximation parameters, $r\in\{1,2,\dots,n-1\}$ - SVD rank parameter, initialization matrix -  $\X_0\in\mS_n$ such that $\rank(\X_0) = r$
		\State \textbf{Initialization:} $(\X_1,\y_1)=((1-\varepsilon_0)\X_0+\frac{\varepsilon_0}{n}\I,\y_1)$ where $\y_1\in\mathcal{K}$
		\For{$t = 1,2,...$} 
		    \State $\P_t = \exp(\log(\X_t)-\eta_t\nabla_{\X}f(\X_t,\y_t))$ \textrm{\{not to be explicitly computed\}}
		    \State $\P^r_t = \V_{\P}^r\bLambda_{\P}^r{\V_{\P}^r}^{\top}$ --- rank-$r$ approximation of $\P_t$
		    \State $\Z_{t+1}=(1-\varepsilon_t)\frac{\P^r_t}{\trace(\P^r_t)}+\frac{\varepsilon_t}{n-r}(\I-\V_{\P}^r{\V_{\P}^r}^{\top})$   \textrm{\{denote $\widehat{\Z}_{t+1}=\frac{\P_t}{\trace(\P_t)}$\}}
			\State $\w_{t+1}=\argmin_{\y\in\mathcal{K}}\lbrace\langle-\eta_t\nabla_{\y}f(\X_t,\y_t),\y\rangle+\breg_{\y}(\y,\y_t)\rbrace$
			\State $\B_t = \exp(\log(\X_t)-\eta_t\nabla_{\X}f(\Z_{t+1},\w_{t+1}))$  \textrm{\{not to be explicitly computed\}}
		    \State $\B^r_t = \V_{\B}^r\bLambda_{\B}^r{\V_{\B}^r}^{\top}$ --- rank-$r$ approximation of $\B_t$
		    \State $\X_{t+1}=(1-\varepsilon_t)\frac{\B^r_t}{\trace(\B^r_t)}+\frac{\varepsilon_t}{n-r}(\I-\V_{\B}^r{\V_{\B}^r}^{\top})$     \textrm{\{denote $\widehat{\X}_{t+1}=\frac{\B_t}{\trace(\B_t)}$\}}
			\State $\y_{t+1}=\argmin_{\y\in\mathcal{K}}\lbrace\langle-\eta_t\nabla_{\y}f(\Z_{t+1},\w_{t+1}),\y\rangle+\breg_{\y}(\y,\y_t)\rbrace$
        \EndFor
	\end{algorithmic}
\end{algorithm}

The updates of both primal sequences $\{\X_t\}_{t\geq 1}, \{\Z_t\}_{t\geq 2}$  
involve computing the top-$r$ components in the eigen-decomposition of the matrix $\exp(\log(\X_t)-\eta_t\nabla_t)$ for some $\nabla_t\in\mathbb{S}^n$. Importantly, this could be carried out from one iteration to next while storing only rank-$r$ matrices and computing only rank-$r$ orthogonal decompositions: to compute the rank-$r$ SVD of $\exp(\log(\X_t)-\eta_t\nabla_t)$, we compute the top-$r$ eigen-decomposition of $\log(\X_t)-\eta_t\nabla_t$, and then exponentiate the eigenvalues. To compute the top-$r$ eigen-decomposition of $\log(\X_t)-\eta_t\nabla_t$, fast iterative methods for thin SVD (see for example \cite{Musco15}) could be applied. 
By our update, the eigen-decomposition of $\X_t$ has the form of $\X_t=(1-\varepsilon_{t-1})\V^r\bLambda^r{\V^r}^{\top}+\frac{\varepsilon_{t-1}}{n-r}\left(\I-\V^r{\V^r}^{\top}\right)$ for some $\V^r\in\reals^{n\times r}$ and diagonal matrix $\bLambda^r\in\reals^{r\times r}$, and thus $\log(\X_t)$ admits the following low-rank factorization:
$\log(\X_t)=\V^r\log((1-\varepsilon_{t-1})\bLambda^r){\V^r}^{\top}+\log\left(\frac{\varepsilon_{t-1}}{n-r}\right)\left(\I-\V^r{\V^r}^{\top}\right)$, which allows for fast implementations when computing the top-$r$ components in the eigen-decomposition of $\log(\X_t)-\eta_t\nabla_t$, as discussed above. Importantly, an efficient implementation will indeed store the current primal iterates $\X_t,\Z_t$ in factored form as suggested above, without computing them explicitly. In this way only $n\times r$ matrices (the matrices $\V^r_{\P}$  and $\V^r_{\B}$)  needs to be stored in memory (aside from the gradients of course). 
Please see more detailed discussions in \cite{garber2020efficient}.

Note that the update of the dual variables can be w.r.t. any bregman distance, under the assumption that solving the corresponding minimization problems over the set $\mK$, as required in Algorithms \ref{alg:EGMP} and \ref{alg:LRMP}, can be done efficiently.

We now state our second main theorem which concerns the convergence of Algorithm \ref{alg:LRMP}.

\begin{theorem}\label{thm:allPutTogetherVonNeumann}
Fix a saddle point $(\X^*,\y^*)$ to Problem \eqref{problem1} for which Assumption \ref{ass:strictcompSaddlePoint} holds with parameters $r^*$ and $\delta>0$ where $r^*=\rank(\X^*)$.
Let $\{(\X_t,\y_t)\}_{t\geq 1}$ and $\{(\Z_t,\w_t)\}_{t\geq 2}$ be the sequences of iterates generated by Algorithm \ref{alg:LRMP} with SVD rank parameter $r=r^*$ and with a fixed step-size $\eta_t=\eta:=\min\left\lbrace\frac{1}{\beta_{X}+\beta_{Xy}},\frac{1}{\beta_{y}+\beta_{yX}},\frac{1}{2\sqrt{2}\sqrt{\beta_{X}^2+\beta_{yX}^2}},\frac{1}{2\sqrt{\beta_{y}^2+\beta_{Xy}^2}}\right\rbrace$, where $\beta_X,\beta_y,\beta_{Xy},\beta_{yX}\ge0$ are as defined in \eqref{betas}. Suppose that for all $t\geq 0$: 
\begin{align*}
& \varepsilon_t=
\\ & \min\left\lbrace\hspace{-0.5mm} 1,\min\left\lbrace\hspace{-0.5mm}\frac{\tilde{\varepsilon}_0^2}{16\max\{G^2\eta^2,1\}},\left(\hspace{-0.5mm}\frac{\delta}{16\sqrt{2}\max\lbrace\beta_{X}\hspace{-0.5mm},\hspace{-0.5mm}\beta_{Xy}\rbrace\left(\sqrt{5}\hspace{-0.5mm}+\hspace{-0.5mm}\sqrt{\eta G}\right)}\hspace{-0.5mm}\right)^4\hspace{-0.5mm}\right\rbrace\frac{1}{(t+1+c)^3}\hspace{-0.5mm}\right\rbrace
\end{align*}
for some $\tilde{\varepsilon}_0\leq R_0^2$,
where
\begin{align*}
R_0 := \frac{1}{4}\left(7\max\lbrace\beta_{X},\beta_{Xy}\rbrace+\left(1+\frac{2\sqrt{2r^*}}{\lambda_{r^*}(\X^*)}\right)G\right)^{-1}\delta,
\end{align*}
$c\ge \frac{12}{\eta\delta}$, and $G\geq\sup_{(\X,\y)\in\Sn\times\mathcal{K}}\Vert \nabla_{\X} f(\X,\y)\Vert_2$. Finally, assume the initialization matrix $\X_0$ satisfies $\rank(\X_0) = r^*$, assume $\X_1$ satisfies $\breg_{\X}(\X^*,\X_1)\le\frac{1}{2}R_0^2$ (see conditions on the parameters $\X_0,\X^*,\tilde{\varepsilon}_0,\frac{1}{\sqrt{2}}R_0$ in Lemma \ref{lem:warmstart}),
and assume $\y_1$ satisfies $\breg_{\y}(\y^*,\y_1)\le\frac{1}{2}R_0^2$.
Then, for any $T\geq 1$ it holds that
\begin{align*}
& \max_{\y\in\mathcal{K}} f\left(\frac{1}{T}\sum_{t=1}^T \Z_{t+1},\y\right) -  \min_{\X\in\Sn} f\left(\X,\frac{1}{T}\sum_{t=1}^T\w_{t+1}\right)
\\ & \le \frac{D^2}{\eta T} + \frac{\eta^{-1}+2(\beta_{X}+\beta_{yX})}{16\max\{G^2\eta^2,1\}}\frac{R_0^4}{T} +\frac{R_0^2}{\max\lbrace\eta,1\rbrace T},
\end{align*}
where $D^2:=\sup_{(\X,\y)\in{\Sn\times\mathcal{K}}}\left(\breg_{\X}(\X,\X_1)+\breg_{\y}(\y,\y_1)\right)$.
\end{theorem}

Translating the convergence result we obtained for saddle-points in Theorem \ref{thm:allPutTogetherVonNeumann} back to the original nonsmooth problem, Problem \eqref{nonSmoothProblem}, we obtain the following convergence result. The proof is identical to the proof of Corollary \ref{cor:nonsmoothExtragradient}.

\begin{corollary} \label{cor:nonsmoothVNE}
Fix an optimal solution $\X^*\in\Sn$ to Problem \eqref{nonSmoothProblem} and assume Assumption \ref{ass:struct} holds. Let $\G^*\in\partial g(\X^*)$ which satisfies that $\langle\X-\X^*,\G^*\rangle\ge0$  for all $\X\in\Sn$. Define $\delta:=\lambda_{n-r^*}(\G^*)-\lambda_{n}(\G^*)>0$. Define $f$ as in Problem \eqref{problem1} and let and let $\lbrace(\X_t,\y_t)\rbrace_{t\ge1}$ and $\lbrace(\Z_{t},\w_t)\rbrace_{t\ge2}$ be the sequences of iterates generated by Algorithm \ref{alg:LRMP} where $\eta$, $R_0$, $\lbrace\varepsilon_t\rbrace_{t=1}^T$, $\X_0$, $\X_1$, and $\y_1$ are as defined in Theorem \ref{thm:allPutTogetherVonNeumann}.
Then, for all $T\ge1$ it holds that
\begin{align*}
g\left(\frac{1}{T}\sum_{t=1}^T \Z_{t+1}\right)-g(\X^*) 
& \le \frac{D^2}{\eta T} + \frac{\eta^{-1}+2(\beta_{X}+\beta_{yX})}{16\max\{G^2\eta^2,1\}}\frac{R_0^4}{T} +\frac{R_0^2}{\max\lbrace\eta,1\rbrace T},
\end{align*} 
where $D^2:=\hspace{-0.5mm}\sup\limits_{(\X,\y)\in{\Sn\times\mathcal{K}}}\hspace{-0.5mm}\left(\breg_{\X}(\X,\X_1)+\breg_{\y}(\y,\y_1)\right)$ and $G\geq\hspace{-0.5mm}\sup\limits_{(\X,\y)\in\Sn\times\mathcal{K}}\hspace{-0.5mm}\Vert \nabla_{\X} f(\X,\y)\Vert_2$.
\end{corollary}

\begin{remark}
Theorem \ref{thm:allPutTogetherVonNeumann} requires that Assumption \ref{ass:strictcompSaddlePoint} holds with parameter $r=\rank(\X^*)$, i.e., that the saddle point problem satisfies strict complementarity, and consequently, Corollary \ref{cor:nonsmoothVNE} requires that the nonsmooth Problem \eqref{nonSmoothProblem} satisfies strict complementarity (Assumption \ref{ass:struct}). This is opposed to the corresponding Theorem \ref{theroem:allPutTogether} and Corollary \ref{cor:nonsmoothExtragradient} for the Euclidean extragradient method, which allow to relax these strict complementarity assumptions and only requires generalized strict complementarity. It is currently unclear to us if the conditions of Theorem  \ref{thm:allPutTogetherVonNeumann} could be relaxed to only require generalized strict complementarity with some $r\geq\rank(\X^*)$. 
\end{remark}

The following lemma states a condition on the initialization of the $\X$ variable w.r.t. the parameters $\X_0$ and $\varepsilon_0$ in Algorithm \ref{alg:LRMP}, so that the initial point $\X_1$ is close enough to the low-rank optimal solution, as is required in Theorem \ref{thm:allPutTogetherVonNeumann}. This lemma is identical to Lemma 4 in \cite{garber2020efficient}.

\begin{lemma}[Warm-start initialization]\label{lem:warmstart}
Let $(\X^*,\y^*)\in\mX^*\times\mY^*$ be an optimal solution to Problem \eqref{problem1}. Let $\rank(\X^*):=r^*$, and let $\X\in\mS_n$ be such that $\rank(\X) = \rank(\X^*)$. Let $\varepsilon\in(0,3/4]$ and fix some $R>0$. Suppose
\begin{align}\label{eq:init:equalRank}
\trace(\X^*\log\X^* - \X^*\V_{\X}\log(\bLambda_{\X})\V_{\X}^{\top}) +  \frac{2\lambda_1(\X^*)}{\lambda_{r^*}(\X^*)^2}\log\left({\frac{n}{\varepsilon}}\right)\Vert{\X^*-\X}\Vert_F^2  + 4\varepsilon \leq R^2.
\end{align}
holds, where  $\V_{\X}\bLambda_{\X}\V_{\X}^{\top}$ denotes the compact-form eigen-decomposition of $\X$ (i.e., $\bLambda_{\X}\succ\mathbf{0}$), and similarly $\V_{\X^*}\bLambda_{\X^*}\V_{\X^*}^{\top}$ denotes the compact-form eigen-decomposition of $\X^*$. Then, the matrix $\W = (1-\varepsilon)\X + \frac{\varepsilon}{n }\I$ satisfies $\breg_{\X}(\X^*,\W) \leq R^2$. 
\end{lemma}

In order to prove Theorem \ref{thm:allPutTogetherVonNeumann}, we first need to prove the following lemma which is the central piece of the analysis. The lemma, which is a certain analog of Lemma \ref{lemma:radiusOfLowRankProjections} from the Euclidean case, establishes that, when close enough to a low-rank saddle-point of Problem \eqref{problem1} and under Assumption \ref{ass:strictcompSaddlePoint}, the errors originating from approximating the matrix exponentiated gradient steps are sufficiently bounded. 

\begin{lemma} \label{lemma:series_error}
Let $(\X^*,\y^*)\in\mX^*\times\mY^*$ be an optimal solution for which Assumption \ref{ass:strictcompSaddlePoint} holds with parameters $r^*$ and $\delta>0$. Let $\rank(\X^*):=r^*$ and let $n\not= r\ge r^*$ be the SVD rank parameter in Algorithm \ref{alg:LRMP}. Fix some iteration $t\ge1$ of Algorithm \ref{alg:LRMP}. Let $\X_t\in\Sn\cap\mathbb{S}^n_{++}$ such that $\lambda_{r^*+1}(\X_t)\le\frac{\varepsilon_{t-1}}{n-r}$ and denote $G=\sup_{(\X,\y)\in{\Sn\times\mathcal{K}}}\Vert\nabla_{\X}f(\X,\y)\Vert_2$. If
\begin{align*}
& \sqrt{\breg_{\X}(\X^*,\X_t) +\breg_{\y}(\y^*,\y_t)} 
\\ & \le \frac{1}{\sqrt{2}}\left(2\max\lbrace\beta_{X},\beta_{Xy}\rbrace\sqrt{\frac{1}{\gamma_t}+8}+\left(1+\frac{2\sqrt{2r^*}}{\lambda_{r^*}(\X^*)}\right)G\right)^{-1}
\\ & \ \ \ \left(\delta -\frac{1}{\eta_t}\log\left(\frac{\varepsilon_{t-1}}{\varepsilon_t}\right) +\frac{2\varepsilon_t}{\eta_t} -4\max\lbrace\beta_{X},\beta_{Xy}\rbrace\left(\sqrt{\frac{1}{\gamma_t}+8}\sqrt{\varepsilon_t}+\sqrt{\frac{\eta_tG}{\gamma_t}}\sqrt[4]{\varepsilon_t}\right)\right),
\end{align*}
where $\gamma_t = \min\left\lbrace 1-4\eta_t^2(\beta_{X}^2+\beta_{yX}^2),1-2\eta_t^2(\beta_{y}^2+\beta_{Xy}^2)\right\rbrace$,
then for $\widehat{\Z}_{t+1}$ and $\widehat{\X}_{t+1}$ as defined in Algorithm \ref{alg:LRMP}, it holds for any $\X\in\Sn$ that 
\begin{align} \label{ineq:toBound}
\max\left\lbrace\breg_{\X}(\X,\X_{t+1})-\breg_{\X}(\X,\widehat{\X}_{t+1}),\breg_{\X}(\widehat{\Z}_{t+1},\Z_{t+1})\right\rbrace \le 2\varepsilon_t.
\end{align}
\end{lemma}
The proof of the lemma leverages the specific spectral structure of the bregman distance corresponding to the von Neumann entropy over the spectrahedron and applies perturbation bounds for the eigenvalues in order to relate the eigenvalues of the approximated points to those of the saddle-point which satisfies Assumption  \ref{ass:strictcompSaddlePoint}, in order to bound the error terms due to the low-rank approximations.
\begin{proof}
Our goal is to bound the LHS of \eqref{ineq:toBound} under the conditions stated in Lemma \ref{lemma:series_error}.
For the first term in the LHS of \eqref{ineq:toBound}, by invoking the von Neumann inequality, for all $\X\in\Sn$ it holds that
\begin{align} \label{ineq:firstTermBound}
& \breg_{\X}(\X,\X_{t+1})-\breg_{\X}(\X,\widehat{\X}_{t+1}) \nonumber
\\ & = \trace(\X\log(\X)-\X\log(\X_{t+1}))-\trace(\X\log(\X)-\X\log(\widehat{\X}_{t+1})) \nonumber
\\ & = \trace(\X(\log(\widehat{\X}_{t+1})-\log(\X_{t+1}))) \nonumber
\\ & \le \sum_{i=1}^n \lambda_i(\X)\lambda_i(\log(\widehat{\X}_{t+1})-\log(\X_{t+1})) \nonumber
\\ & \le \lambda_1(\log(\widehat{\X}_{t+1})-\log(\X_{t+1})).
\end{align}

Similarly, invoking the von Neumann inequality and using the fact that $\widehat{\Z}_{t+1}\in\Sn$, for the second term in the LHS of \eqref{ineq:toBound} it holds that 
\begin{align} \label{ineq:secondTermBound}
\breg_{\X}(\widehat{\Z}_{t+1},\Z_{t+1}) 
& = \trace(\widehat{\Z}_{t+1}\log(\widehat{\Z}_{t+1})-\widehat{\Z}_{t+1}\log(\Z_{t+1})) \nonumber
\\ & \le \sum_{i=1}^n \lambda_i(\widehat{\Z}_{t+1})\lambda_i(\log(\widehat{\Z}_{t+1})-\log(\Z_{t+1})) \nonumber
\\ & \le \lambda_1(\log(\widehat{\Z}_{t+1})-\log(\Z_{t+1})).
\end{align}

Thus, combining \eqref{ineq:firstTermBound} and \eqref{ineq:secondTermBound} our goal is to bound
\begin{align}  \label{ineq:newTermToBound}
\max\left\lbrace\lambda_1(\log(\widehat{\X}_{t+1})-\log(\X_{t+1})),\lambda_1(\log(\widehat{\Z}_{t+1})-\log(\Z_{t+1}))\right\rbrace,
\end{align}
by bounding the top eigenvalues of the matrices $\log(\widehat{\X}_{t+1})-\log(\X_{t+1})$ and $ \log(\widehat{\Z}_{t+1})-\log(\Z_{t+1})$.

Let $\P_t,\B_t,\P^{r}_t,\B^{r}_t$ be as defined in Algorithm \ref{alg:LRMP}. Let $\P_t = \V_{\P}\bLambda_{\P}\V_{\P}^{\top}$ and $\B_t = \V_{\B}\bLambda_{\B}\V_{\B}^{\top}$ denote the eigen-decompositions of $\P_t$ and $\B_t$ and let $\P^{r}_t = \V^{r}_{\P}\bLambda^{r}_{\P}{\V^{r}_{\P}}^{\top}$ and $\B^{r}_t = \V^{r}_{\B}\bLambda^{r}_{\B}{\V^{r}_{\B}}^{\top}$ denote their rank-$r$ decompositions respectively. $\Z_{t+1}$ and $\widehat{\Z}_{t+1}$ have the same eigenvectors and also $\X_{t+1}$ and $\widehat{\X}_{t+1}$ have the same eigenvectors. 
Therefore, plugging these into the terms of \eqref{ineq:newTermToBound}, we have that
\begin{align} 
\log(\widehat{\Z}_{t+1})-\log(\Z_{t+1})&=\V_{\P}\left[\log\left(\frac{\bLambda_{\P}}{\trace(\P_t)}\right)-\log\left(\begin{array}{cc}(1-\varepsilon_t)\frac{\bLambda_{\P}^{r}}{\trace(\P^{r}_t)} & 0 \\ 0 & \frac{\varepsilon_t}{n-r}\I \end{array}\right)\right]\V_{\P}^{\top}, \label{eq:ztEigs}
\\ \log(\widehat{\X}_{t+1})-\log(\X_{t+1})&=\V_{\B}\left[\log\left(\frac{\bLambda_{\B}}{\trace(\B_t)}\right)-\log\left(\begin{array}{cc}(1-\varepsilon_t)\frac{\bLambda_{\B}^{r}}{\trace(\B^{r}_t)} & 0 \\ 0 & \frac{\varepsilon_t}{n-r}\I \end{array}\right)\right]\V_{\B}^{\top}. \label{eq:xtEigs}
\end{align}
Denote the matrices of eigenvalues in \eqref{eq:ztEigs} and \eqref{eq:xtEigs} as
\begin{align*}
\D^{\P}_t&:=\log\left(\frac{\bLambda_{\P}}{\trace(\P_t)}\right)-\log\left(\begin{array}{cc}(1-\varepsilon_t)\frac{\bLambda_{\P}^{r}}{\trace(\P^{r}_t)} & 0 \\ 0 & \frac{\varepsilon_t}{n-r}\I \end{array}\right), 
\\ \D^{\B}_t&:=\log\left(\frac{\bLambda_{\B}}{\trace(\B_t)}\right)-\log\left(\begin{array}{cc}(1-\varepsilon_t)\frac{\bLambda_{\B}^{r}}{\trace(\B^{r}_t)} & 0 \\ 0 & \frac{\varepsilon_t}{n-r}\I \end{array}\right). 
\end{align*}
The eigenvalues in the diagonals of $\D^{\P}_t$ and $\D^{\B}_t$ satisfy that
\begin{align}
\forall j\in[r] &:\ \lambda_j(\D^{\P}_t) = \log\left(\frac{\frac{\lambda_j(\P_t)}{\trace(\P_t)}}{(1-\varepsilon_t)\frac{\lambda_j(\P_t)}{\trace(\P_t^{r})}}\right),\ \ \lambda_j(\D^{\B}_t) = \log\left(\frac{\frac{\lambda_j(\B_t)}{\trace(\B_t)}}{(1-\varepsilon_t)\frac{\lambda_j(\B_t)}{\trace(\B_t^{r})}}\right)  ; \label{BigEig} 
\\ \forall j>r & :\ \lambda_j(\D^{\P}_t) = \log\left(\frac{\frac{\lambda_j(\P_t)}{\trace(\P_t)}}{\frac{\varepsilon_t}{n-r}}\right),\ \ \lambda_j(\D^{\B}_t) = \log\left(\frac{\frac{\lambda_j(\B_t)}{\trace(\B_t)}}{\frac{\varepsilon_t}{n-r}}\right). \label{LittleEig}
\end{align}
It is important to note that $\lambda_1(\D^{\P}_t),\ldots,\lambda_n(\D^{\P}_t)$ and $\lambda_1(\D^{\B}_t),\ldots,\lambda_n(\D^{\B}_t)$ are not necessarily ordered in a non-increasing order. Thus, we must bound the top eigenvalues of both \eqref{BigEig} and \eqref{LittleEig} to obtain a bound on \eqref{ineq:newTermToBound}.

\paragraph*{Bounding the first r eigenvalues of $\D^{\P}_t$ and $\D^{\B}_t$}
From \eqref{BigEig}, for all $j\le r$ it holds that
\begin{align} \label{ineq:bigEigsBoundP} 
\lambda_j(\D^{\P}_t) & = \log\left(\frac{\lambda_j(\P_t)\sum_{i=1}^{r}\lambda_i(\P_t)}{(1-\varepsilon_t)\lambda_j(\P_t)\sum_{i=1}^n\lambda_i(\P_t)}\right) \nonumber
\\ & = \log\left(\frac{\sum_{i=1}^{r}\lambda_i(\P_t)}{\sum_{i=1}^n\lambda_i(\P_t)}\right)+ \log\left(\frac{1}{1-\varepsilon_t}\right)
\le \log\left(\frac{1}{1-\varepsilon_t}\right) \underset{(a)}{\le} 2\varepsilon_t,
\end{align}
where (a) holds for any $\varepsilon_t\le\frac{3}{4}$. Using the same arguments, for all $j\le r$ and any any $\varepsilon_t\le\frac{3}{4}$ also 
\begin{align*}
\lambda_j(\D^{\B}_t) & \le 2\varepsilon_t.
\end{align*}

\paragraph*{Bounding the rest of the eigenvalues of $\D^{\P}_t$ and $\D^{\B}_t$} 
For $j\ge r+1$, note that using the definition of the eigenvalues of $\D^{\P}_t$ and $\D^{\B}_t$ in \eqref{LittleEig}, it holds that 
\begin{align} 
\lambda_j(\D^{\P}_t) & = \log\left(\frac{(n-r)\lambda_j(\P_t)}{\varepsilon_t \sum_{i=1}^n \lambda_i(\P_t)}\right)
\le \log\left(\frac{(n-r)\lambda_{r+1}(\P_t)}{\varepsilon_t \sum_{i=1}^n \lambda_i(\P_t)}\right) = \lambda_{r+1}(\D^{\P}_t), \label{ineq:toUseCertificateConvergenceP}
\\ \lambda_j(\D^{\B}_t) & = \log\left(\frac{(n-r)\lambda_j(\B_t)}{\varepsilon_t \sum_{i=1}^n \lambda_i(\B_t)}\right)
\le \log\left(\frac{(n-r)\lambda_{r+1}(\B_t)}{\varepsilon_t \sum_{i=1}^n \lambda_i(\B_t)}\right) = \lambda_{r+1}(\D^{\B}_t). \label{ineq:toUseCertificateConvergenceB}
\end{align}

It holds that
\begin{align} 
\log\left(\frac{(n-r)\lambda_{r+1}(\P_t)}{\varepsilon_t \sum_{i=1}^n \lambda_i(\P_t)}\right) & = \log\left(\frac{n-r}{\varepsilon_t}\right) + \log(\lambda_{r+1}(\P_t))-\log\left( \sum_{i=1}^n \lambda_i(\P_t)\right), \label{ineq:LittleEigBoundP}
\\ \log\left(\frac{(n-r)\lambda_{r+1}(\B_t)}{\varepsilon_t \sum_{i=1}^n \lambda_i(\B_t)}\right) & = \log\left(\frac{n-r}{\varepsilon_t}\right) + \log(\lambda_{r+1}(\B_t))-\log\left( \sum_{i=1}^n \lambda_i(\B_t)\right). \label{ineq:LittleEigBoundB}
\end{align}
We will now separately bound the last two terms in the RHS of \eqref{ineq:LittleEigBoundP} and the RHS of \eqref{ineq:LittleEigBoundB}.

\paragraph*{Bounds for the middle terms in the RHS of \eqref{ineq:LittleEigBoundP} and \eqref{ineq:LittleEigBoundB}}
\begin{align}  \label{ineq:InProofLambdaRPlus1P}
& \log(\lambda_{r+1}(\P_t)) \le \log(\lambda_{r^*+1}(\P_t)) = \lambda_{r^*+1}(\log(\P_t)) \nonumber
\\ & \underset{(a)}{\le} \lambda_{r^*+1}(\log(\X_t)-\eta\nabla_{\X} f(\X^*,\y^*)) + \eta_t\lambda_1(\nabla_{\X} f(\X^*,\y^*)-\nabla_{\X}f(\X_{t},\y_{t})) \nonumber
\\ & \le \sum_{i=1}^{r^*+1}\lambda_{i}(\log(\X_t)-\eta_t\nabla_{\X}f(\X^*,\y^*))-\sum_{i=1}^{r^*}\lambda_{i}(\log(\X_t)-\eta_t\nabla_{\X}f(\X^*,\y^*))\nonumber \\ & \ \ \  +\eta_t\Vert\nabla_{\X} f(\X^*,\y^*)-\nabla_{\X}f(\X_{t},\y_{t})\Vert_2 \nonumber
\\ & \underset{(b)}{\le} \sum_{i=1}^{r^*+1}\lambda_{i}(\log(\X_t))-\eta\sum_{i=1}^{r^*+1}\lambda_{n-i+1}(\nabla_{\X}f(\X^*,\y^*))-\sum_{i=1}^{r^*}\lambda_{i}(\log(\X)-\eta_t\nabla_{\X}f(\X^*,\y^*)) \nonumber \\ & \ \ \ +\eta_t\Vert\nabla_{\X} f(\X^*,\y^*)-\nabla_{\X}f(\X_{t},\y_{t})\Vert_2 \nonumber
\\ & \underset{(c)}{=} \sum_{i=1}^{r^*}\lambda_{i}(\log(\X_t))+\log\left(\lambda_{r^*+1}(\X_t)\right)-\eta r^*\lambda_n(\nabla_{\X}f(\X^*,\y^*))-\eta_t\lambda_{n-r^*}(\nabla_{\X}f(\X^*,\y^*))\nonumber \\ & \ \ \ -\sum_{i=1}^{r^*}\lambda_{i}(\log(\X_t)-\eta_t\nabla_{\X}f(\X^*,\y^*))+\eta_t\Vert\nabla_{\X} f(\X^*,\y^*)-\nabla_{\X}f(\X_{t},\y_{t})\Vert_2,
\end{align}
where (a) follows from Weyl's inequality, (b) follows from Ky Fan's inequality for eigenvalues and (c) follows from Lemma \ref{lemma:Xopt-subgradEigen}.

Using the same arguments and replacing $\nabla_{\X}f(\X_{t},\y_{t})$ with $\nabla_{\X}f(\Z_{t+1},\w_{t+1})$ we obtain 
\begin{align}  \label{ineq:InProofLambdaRPlus1B}
& \log(\lambda_{r+1}(\B_t)) \le \log(\lambda_{r^*+1}(\B_t)) \nonumber
\\ & \le \sum_{i=1}^{r^*}\lambda_{i}(\log(\X_t))+\log\left(\lambda_{r^*+1}(\X_t)\right)-\eta r^*\lambda_n(\nabla_{\X}f(\X^*,\y^*))-\eta_t\lambda_{n-r^*}(\nabla_{\X}f(\X^*,\y^*))\nonumber \\ & \ \ \ -\sum_{i=1}^{r^*}\lambda_{i}(\log(\X_t)-\eta_t\nabla_{\X}f(\X^*,\y^*))+\eta_t\Vert\nabla_{\X} f(\X^*,\y^*)-\nabla_{\X}f(\Z_{t+1},\w_{t+1})\Vert_2.
\end{align}

Denote $\log(\X_t)=\W\bLambda_{\X}\W^{\top}$ and $\W_{r^*}$ to be the matrix with the $r^*$ first columns of $\W$.
\begin{align} \label{ineq:inProofSumReigs}
-\sum_{i=1}^{r^*}\lambda_{i}(\log(\X_t)-\eta_t\nabla_{\X}f(\X^*,\y^*)) & \le -\W_{r^*}\W_{r^*}^{\top}\bullet(\log(\X_t)-\eta_t\nabla_{\X}f(\X^*,\y^*)) \nonumber
\\ & = -\sum_{i=1}^{r^*}\lambda_i(\log(\X_t))+\eta_t\W_{r^*}\W_{r^*}^{\top}\bullet\nabla_{\X}f(\X^*,\y^*).
\end{align}

Let $\X^*=\V^*\bLambda^*{\V^*}^{\top}$ denote the eigen-decomposition of $\X^*$. By Lemma \ref{lemma:connectionSubgradientNonSmoothAndSaddlePoint}, $\nabla_{\X}f(\X^*,\y^*)$ is a subgradient of the corresponding nonsmooth objective $g(\X)=\max_{\y\in\mathcal{K}}f(\X,\y)$ at $\X^*$ which also satisfies the first-order optimality condition. Therefore, invoking Lemma \ref{lemma:Xopt-subgradEigen} with this sugradient we get that $\V^*{\V^*}^{\top}\bullet\nabla_{\X}f(\X^*,\y^*)=r^*\lambda_{n}(\nabla_{\X}f(\X^*,\y^*))$. Therefore,
\begin{align} \label{ineq:inProofInnerProd}
& \W_{r^*}\W_{r^*}^{\top}\bullet\nabla_{\X}f(\X^*,\y^*) \nonumber
\\ & = \V^*{\V^*}^{\top}\bullet\nabla_{\X}f(\X^*,\y^*)+\W_{r^*}\W_{r^*}^{\top}\bullet\nabla_{\X}f(\X^*,\y^*)-\V^*{\V^*}^{\top}\bullet\nabla_{\X}f(\X^*,\y^*) \nonumber
\\ & \le \V^*{\V^*}^{\top}\bullet\nabla_{\X}f(\X^*,\y^*)+\Vert\W_{r^*}\W_{r^*}^{\top}-\V^*{\V^*}^{\top}\Vert_{*}\Vert\nabla_{\X}f(\X^*,\y^*)\Vert_{2} \nonumber
\\ & = r^*\lambda_{n}(\nabla_{\X}f(\X^*,\y^*))+\Vert\W_{r^*}\W_{r^*}^{\top}-\V^*{\V^*}^{\top}\Vert_{*}\Vert\nabla_{\X}f(\X^*,\y^*)\Vert_{2} \nonumber
\\ & \le r^*\lambda_{n}(\nabla_{\X}f(\X^*,\y^*))+\sqrt{2r^*}\Vert\W_{r^*}\W_{r^*}^{\top}-\V^*{\V^*}^{\top}\Vert_{F}\Vert\nabla_{\X}f(\X^*,\y^*)\Vert_{2} \nonumber
\\ & \underset{(a)}{\le} r^*\lambda_{n}(\nabla_{\X}f(\X^*,\y^*))+\frac{2\sqrt{2r^*}\Vert\X_t-\X^*\Vert_{F}}{\lambda_{r^*}(\X^*)}\Vert\nabla_{\X}f(\X^*,\y^*)\Vert_{2} \nonumber
\\ & \le r^*\lambda_{n}(\nabla_{\X}f(\X^*,\y^*))+\frac{2\sqrt{2r^*}\Vert\X_t-\X^*\Vert_{*}}{\lambda_{r^*}(\X^*)}\Vert\nabla_{\X}f(\X^*,\y^*)\Vert_{2}.
\end{align}
Here, (a) follows using the Davis-Kahan $\sin(\theta)$ theorem (see for instance Theorem 2 in \cite{Davis-Kahan}).

From the smoothness of $f$ it holds that
\begin{align} \label{ineq:smoothnessXt}
& \Vert\nabla_{\X} f(\X^*,\y^*)-\nabla_{\X}f(\X_{t},\y_{t})\Vert_2 \nonumber
\\ & = \Vert\nabla_{\X} f(\X^*,\y^*)-\nabla_{\X}f(\X_{t},\y^*)\Vert_2 + \Vert\nabla_{\X} f(\X_{t},\y^*)-\nabla_{\X}f(\X_{t},\y_{t})\Vert_2 \nonumber
\\ & \le \beta_{X}\Vert\X_{t}-\X^*\Vert_* + \beta_{Xy}\Vert\y_{t}-\y^*\Vert_{\mathcal{Y}},
\end{align}
and similarly,
\begin{align} \label{ineq:smoothnessZt}
& \Vert\nabla_{\X} f(\X^*,\y^*)-\nabla_{\X}f(\Z_{t+1},\w_{t+1})\Vert_2 \nonumber
\\ & = \Vert\nabla_{\X} f(\X^*,\y^*)-\nabla_{\X}f(\Z_{t+1},\y^*)\Vert_2 + \Vert\nabla_{\X} f(\Z_{t+1},\y^*)-\nabla_{\X}f(\Z_{t+1},\w_{t+1})\Vert_2 \nonumber
\\ & \le \beta_{X}\Vert\Z_{t+1}-\X^*\Vert_* + \beta_{Xy}\Vert\w_{t+1}-\y^*\Vert_{\mathcal{Y}},
\end{align}

Plugging the bounds \eqref{ineq:inProofSumReigs}, \eqref{ineq:inProofInnerProd}, and \eqref{ineq:smoothnessXt} into \eqref{ineq:InProofLambdaRPlus1P} and using the assumption that $\lambda_{r^*+1}(\X_t)\le\varepsilon_{t-1}/(n-r)$, we obtain
\begin{align} \label{ineq:InProofSecondTermP}
\log(\lambda_{r+1}(\P_t)) & \le 
\log\left(\frac{\varepsilon_{t-1}}{n-r}\right)-\eta \lambda_{n-r^*}(\nabla_{\X}f(\X^*,\y^*))+ \eta_t\beta_{Xy}\Vert\y_{t}-\y^*\Vert_{\mathcal{Y}} \nonumber
\\ & \ \ \ +\eta_t\beta_{X}\Vert\X_{t}-\X^*\Vert_* +\eta\frac{2\sqrt{2r^*}\Vert\X_t-\X^*\Vert_{*}}{\lambda_{r^*}(\X^*)}\Vert\nabla_{\X}f(\X^*,\y^*)\Vert_{2}. 
\end{align}
Similarly, plugging the bounds \eqref{ineq:inProofSumReigs}, \eqref{ineq:inProofInnerProd}, and \eqref{ineq:smoothnessZt} into \eqref{ineq:InProofLambdaRPlus1B}, we obtain
\begin{align} \label{ineq:InProofSecondTermB}
\log(\lambda_{r+1}(\B_t)) & \le \log\left(\frac{\varepsilon_{t-1}}{n-r}\right)-\eta \lambda_{n-r^*}(\nabla_{\X}f(\X^*,\y^*))+ \eta_t\beta_{Xy}\Vert\w_{t+1}-\y^*\Vert_{\mathcal{Y}} \nonumber
\\ & \ \ \ +\eta_t\beta_{X}\Vert\Z_{t+1}-\X^*\Vert_* +\eta\frac{2\sqrt{2r^*}\Vert\X_t-\X^*\Vert_{*}}{\lambda_{r^*}(\X^*)}\Vert\nabla_{\X}f(\X^*,\y^*)\Vert_{2}. 
\end{align}

\paragraph*{Bounds for the last terms in the RHS of \eqref{ineq:LittleEigBoundP} and \eqref{ineq:LittleEigBoundB}}
For the last term of \eqref{ineq:LittleEigBoundP}, 
\begin{align} \label{ineq:InProofThirdTermP} 
& -\log\left(\sum_{i=1}^n\lambda_i(\P_t)\right) \\ & = -\log\Big(\sum_{i=1}^n\lambda_i\left(\exp\left(\log(\X_t)-\eta_t\nabla_{\X}f(\X_{t},\y_{t})\right)\right)\Big) \nonumber
\\ & \underset{(a)}{\le} -\sum_{i=1}^n\lambda_i(\X_t)\cdot\log\left(\frac{\lambda_i\left(\exp\left(\log(\X_t)-\eta_t\nabla_{\X}f(\X_{t},\y_{t})\right)\right)}{\lambda_i(\X_t)}\right) \nonumber
\\ & = -\sum_{i=1}^n\lambda_i(\X_t)\cdot\left[\log\left(\lambda_i\left(\exp\left(\log(\X_t)-\eta_t\nabla_{\X}f(\X_{t},\y_{t})\right)\right)\right)-\log(\lambda_i(\X_t)\right] \nonumber
\\ & = -\sum_{i=1}^n\lambda_i(\X_t)\cdot\left[\lambda_i\left(\log(\X_t)-\eta_t\nabla_{\X}f(\X_{t},\y_{t})\right)-\log(\lambda_i(\X_t)\right] \nonumber
\\ & \underset{(b)}{\le} -\langle\X_t,\log(\X_t)-\eta_t\nabla_{\X}f(\X_{t},\y_{t})\rangle +\sum_{i=1}^n\lambda_i(\X_t)\lambda_i(\log(\X_t)) \nonumber
\\ & \underset{(c)}{=} -\langle\X_t,\log(\X_t)-\eta_t\nabla_{\X}f(\X_{t},\y_{t})\rangle +\sum_{i=1}^n\lambda_i(\X_t\log(\X_t)) \nonumber
\\ & = -\trace(\X_t\log(\X_t))+\trace(\X_t\log(\X_t))+\eta_t\langle\X_t,\nabla_{\X}f(\X_{t},\y_{t})\rangle \nonumber
\\ & = \eta_t\langle\X^*,\nabla_{\X}f(\X^*,\y^*)\rangle + \eta_t\langle\X_t-\X^*,\nabla_{\X}f(\X^*,\y^*)\rangle \nonumber \\ & \ \ \ +\eta_t\langle\X_t,\nabla_{\X}f(\X_{t},\y_{t})-\nabla_{\X}f(\X^*,\y^*)\rangle \nonumber
\\ & \underset{(d)}{=} \eta_t\lambda_{n}(\nabla_{\X}f(\X^*,\y^*)) + \eta_t\langle\X_t-\X^*,\nabla_{\X}f(\X^*,\y^*)\rangle \nonumber \\ & \ \ \ +\eta_t\langle\X_t,\nabla_{\X}f(\X_{t},\y_{t})-\nabla_{\X}f(\X^*,\y^*)\rangle \nonumber
\\ & \le \eta_t\lambda_{n}(\nabla_{\X}f(\X^*,\y^*)) + \eta_t\Vert\X_t-\X^*\Vert_*\Vert\nabla_{\X}f(\X^*,\y^*)\Vert_2 \nonumber \\ & \ \ \ +\eta_t\Vert\X_t\Vert_*\Vert\nabla_{\X}f(\X_{t},\y_{t})-\nabla_{\X}f(\X^*,\y^*)\Vert_2 \nonumber
\\ & \underset{(e)}{\le} \eta_t\lambda_{n}(\nabla_{\X}f(\X^*,\y^*)) + \eta_t\Vert\X_t-\X^*\Vert_*\Vert\nabla_{\X}f(\X^*,\y^*)\Vert_2 +\eta_t\beta_{X}\Vert\X_{t}-\X^*\Vert_* \nonumber \\ & \ \ \ + \eta_t\beta_{Xy}\Vert\y_{t}-\y^*\Vert_{\mathcal{Y}},
\end{align}
where (a) follows since $-\log(\cdot)$ is convex, (b) follows from the von Neumann inequality, (c) follows since $\X_t$ and $\log(\X_t)$ have the same eigenvectors, (d) follows from Lemma \ref{lemma:Xopt-subgradEigen}, and (e) follows from \eqref{ineq:smoothnessXt} and since $\Vert\X_t\Vert_*=\sum_{i=1}^n\lambda_i(\X_t)=\trace(\X_t)=1$.

For the last term in the RHS of \eqref{ineq:LittleEigBoundB}, replacing $\nabla_{\X}f(\X_{t},\y_{t})$ with $\nabla_{\X}f(\Z_{t+1},\w_{t+1})$ and replacing \eqref{ineq:smoothnessXt} with \eqref{ineq:smoothnessZt}, using the same arguments as in \eqref{ineq:InProofThirdTermP} we obtain  
\begin{align} \label{ineq:InProofThirdTermB} 
-\log\left(\sum_{i=1}^n\lambda_i(\B_t)\right) & 
\le \eta_t\lambda_{n}(\nabla_{\X}f(\X^*,\y^*)) + \eta_t\Vert\X_t-\X^*\Vert_*\Vert\nabla_{\X}f(\X^*,\y^*)\Vert_2 \nonumber
\\ & \ \ \ +\eta_t\beta_{X}\Vert\Z_{t+1}-\X^*\Vert_* + \eta_t\beta_{Xy}\Vert\w_{t+1}-\y^*\Vert_{\mathcal{Y}}.
\end{align}

\paragraph{Combining all parts of the bounds for the eigenvalues of $\D^{\P}_t$ and $\D^{\B}_t$}
Plugging \eqref{ineq:InProofSecondTermP} and \eqref{ineq:InProofThirdTermP} into the RHS of \eqref{ineq:LittleEigBoundP}, we get
\begin{align*}
& \log\left(\frac{(n-r)\lambda_{r+1}(\P_t)}{\varepsilon_t \sum_{i=1}^n \lambda_i(\P_t)}\right)
\\ & \le -\eta_t\delta +\log\left(\frac{\varepsilon_{t-1}}{\varepsilon_t}\right)+ \left(1+\frac{2\sqrt{2r^*}}{\lambda_{r^*}(\X^*)}\right)\Vert\nabla_{\X}f(\X^*,\y^*)\Vert_{2}\Vert\X_t-\X^*\Vert_*
\\ & \ \ \ +2\eta_t\beta_{X}\Vert\X_{t}-\X^*\Vert_*+ 2\eta_t\beta_{Xy}\Vert\y_{t}-\y^*\Vert_{\mathcal{Y}}
\\ & \le -\eta_t\delta +\log\left(\frac{\varepsilon_{t-1}}{\varepsilon_t}\right)+ \sqrt{2}\eta_t\left(1+\frac{2\sqrt{2r^*}}{\lambda_{r^*}(\X^*)}\right)G\sqrt{\breg_{\X}(\X^*,\X_t)} 
\\ & \ \ \  + 2\sqrt{2}\eta_t\beta_{X}\sqrt{\breg_{\X}(\X^*,\X_{t})} + 2\sqrt{2}\eta_t\beta_{Xy}\sqrt{\breg_{\y}(\y^*,\y_{t})}
\\ & \le -\eta_t\delta +\log\left(\frac{\varepsilon_{t-1}}{\varepsilon_t}\right)
\\ & \ \ \ + \sqrt{2}\eta_t\max\left\lbrace\left(1+\frac{2\sqrt{2r^*}}{\lambda_{r^*}(\X^*)}\right)G+2\beta_{X},2\beta_{Xy}\right\rbrace\sqrt{\breg_{\X}(\X^*,\X_t)+\breg_{\y}(\y^*,\y_{t})} ,
\end{align*}
where the second inequality follows from \eqref{ineq:strongConvexityOfBregmanDistance} and the definition of $G$.

Therefore, if $(\X_t,\y_t)$ is a point such that
\begin{align} \label{radiusBoundP}
& \sqrt{\breg_{\X}(\X^*,\X_t)+\breg_{\y}(\y^*,\y_{t})}\nonumber
\\ & \le \frac{1}{\sqrt{2}}\left(\max\left\lbrace\left(1+\frac{2\sqrt{2r^*}}{\lambda_{r^*}(\X^*)}\right)G+2\beta_{X},2\beta_{Xy}\right\rbrace\right)^{-1}\left(\delta -\frac{1}{\eta_t}\log\left(\frac{\varepsilon_{t-1}}{\varepsilon_t}\right)+\frac{2\varepsilon_t}{\eta_t}\right)
\end{align}
then
\begin{align*}
\log\left(\frac{(n-r)\lambda_{r+1}(\P_t)}{\varepsilon_t \sum_{i=1}^n \lambda_i(\P_t)}\right) \le 2\varepsilon_t,
\end{align*}
which implies together with \eqref{ineq:bigEigsBoundP} that $\breg_{\X}(\widehat{\Z}_{t+1},\Z_{t+1})\le 2\varepsilon_t$.

Plugging \eqref{ineq:InProofSecondTermB} and \eqref{ineq:InProofThirdTermB} into the RHS of \eqref{ineq:LittleEigBoundB}, we get
\begin{align*}
& \log\left(\frac{(n-r)\lambda_{r+1}(\B_t)}{\varepsilon_t \sum_{i=1}^n \lambda_i(\B_t)}\right)
\\ & \le -\eta_t\delta +\log\left(\frac{\varepsilon_{t-1}}{\varepsilon_t}\right)+ \left(1+\frac{2\sqrt{2r^*}}{\lambda_{r^*}(\X^*)}\right)\Vert\nabla_{\X}f(\X^*,\y^*)\Vert_{2}\Vert\X_t-\X^*\Vert_*
\\ & \ \ \ +2\eta_t\beta_{X}\Vert\Z_{t+1}-\X^*\Vert_*+ 2\eta_t\beta_{Xy}\Vert\w_{t+1}-\y^*\Vert_{\mathcal{Y}}
\\ & \underset{(a)}{\le} -\eta_t\delta +\log\left(\frac{\varepsilon_{t-1}}{\varepsilon_t}\right)+ \sqrt{2}\eta_t\left(1+\frac{2\sqrt{2r^*}}{\lambda_{r^*}(\X^*)}\right)G\sqrt{\breg_{\X}(\X^*,\X_t)} 
\\ & \ \ \  +2\eta_t\beta_{X}\Vert\Z_{t+1}-\X^*\Vert_*+ 2\eta_t\beta_{Xy}\Vert\w_{t+1}-\y^*\Vert_{\mathcal{Y}}
\\ & \underset{(b)}{\le} -\eta_t\delta +\log\left(\frac{\varepsilon_{t-1}}{\varepsilon_t}\right)+ \sqrt{2}\eta_t\left(1+\frac{2\sqrt{2r^*}}{\lambda_{r^*}(\X^*)}\right)G\sqrt{\breg_{\X}(\X^*,\X_t)} 
\\ & \ \ \  +2\sqrt{2}\eta_t\max\lbrace\beta_{X},\beta_{Xy}\rbrace\sqrt{\frac{1}{\gamma_t}+8}\sqrt{\breg_{\X}(\X^*,\X_t) +\breg_{\y}(\y^*,\y_t)}
\\ & \ \ \ +2\sqrt{2}\eta_t\max\lbrace\beta_{X},\beta_{Xy}\rbrace\sqrt{\left(\frac{1}{\gamma_t}+8\right)\breg_{\X}(\widehat{\Z}_{t+1},\Z_{t+1})+ \frac{\sqrt{2}\eta_tG}{\gamma_t}\sqrt{\breg_{\X}(\widehat{\Z}_{t+1},\Z_{t+1})}}
\\ & \le -\eta_t\delta +\log\left(\frac{\varepsilon_{t-1}}{\varepsilon_t}\right)
\\ & \ \ \  +\sqrt{2}\eta_t\hspace{-0.5mm}\left(\hspace{-0.5mm}2\max\lbrace\beta_{X},\beta_{Xy}\rbrace\sqrt{\hspace{-0.5mm}\frac{1}{\gamma_t}\hspace{-0.5mm}+\hspace{-0.5mm}8}\hspace{-0.5mm}+\hspace{-0.5mm}\left(1\hspace{-0.5mm}+\hspace{-0.5mm}\frac{2\sqrt{2r^*}}{\lambda_{r^*}(\X^*)}\hspace{-0.5mm}\right)\hspace{-0.5mm}G\right)\sqrt{\breg_{\X}(\X^*,\X_t) \hspace{-0.5mm}+\hspace{-0.5mm}\breg_{\y}(\y^*,\y_t)}
\\ & \ \ \ +2\sqrt{2}\eta_t\max\lbrace\beta_{X},\beta_{Xy}\rbrace\hspace{-1mm}\left(\hspace{-1mm}\sqrt{\hspace{-0.5mm}\frac{1}{\gamma_t}\hspace{-0.5mm}+\hspace{-0.5mm}8}\sqrt{\breg_{\X}(\widehat{\Z}_{t+1},\Z_{t+1})}\hspace{-0.5mm}+\hspace{-0.5mm}\sqrt{\hspace{-0.5mm}\frac{\sqrt{2}\eta_tG}{\gamma_t}}\sqrt[4]{\breg_{\X}(\widehat{\Z}_{t+1},\Z_{t+1})}\hspace{-0.5mm}\right),
\end{align*}
where the (a) follows from \eqref{ineq:strongConvexityOfBregmanDistance} and the definition of $G$ and (b) follows from Lemma \ref{lemma:decreasingRadius}.

Therefore, if $(\X_t,\y_t)$ is a point such that 
\begin{align} \label{radiusBoundB}
& \sqrt{\breg_{\X}(\X^*,\X_t) +\breg_{\y}(\y^*,\y_t)} \nonumber
\\ & \le \frac{1}{\sqrt{2}}\left(2\max\lbrace\beta_{X},\beta_{Xy}\rbrace\sqrt{\frac{1}{\gamma_t}+8}+\left(1+\frac{2\sqrt{2r^*}}{\lambda_{r^*}(\X^*)}\right)G\right)^{-1} \nonumber
\\ & \ \ \ \Bigg(\delta -\frac{1}{\eta_t}\log\left(\frac{\varepsilon_{t-1}}{\varepsilon_t}\right) +\frac{2\varepsilon_t}{\eta_t} \nonumber
\\ & \ \ \ -2\sqrt{2}\max\lbrace\beta_{X},\beta_{Xy}\rbrace\Bigg(\hspace{-1mm}\sqrt{\hspace{-0.5mm}\frac{1}{\gamma_t}\hspace{-0.5mm}+\hspace{-0.5mm}8}\sqrt{\breg_{\X}(\widehat{\Z}_{t+1},\Z_{t+1})}+\sqrt{\hspace{-0.5mm}\frac{\sqrt{2}\eta_tG}{\gamma_t}}\sqrt[4]{\breg_{\X}(\widehat{\Z}_{t+1},\Z_{t+1})}\Bigg)\hspace{-0.5mm}\Bigg),
\end{align}
then
\begin{align*}
\log\left(\frac{(n-r)\lambda_{r+1}(\B_t)}{\varepsilon_t \sum_{i=1}^n \lambda_i(\B_t)}\right)\le 2\varepsilon_t.
\end{align*}

Thus, if $(\X_t,\y_t)$ satisfies both \eqref{radiusBoundP} and \eqref{radiusBoundB}, then taking the minimum between the RHS of \eqref{radiusBoundP} and the RHS of \eqref{radiusBoundB} and bounding $\breg_{\X}(\widehat{\Z}_{t+1},\Z_{t+1})\le 2\varepsilon_t$ gives us the result in the lemma.

\end{proof}

Now we can prove Theorem \ref{thm:allPutTogetherVonNeumann}.
\begin{proof}
We first observe that by our choice of the sequence $\{\varepsilon_t\}_{t\geq 0}$ we have that the sequence $\left\{-\frac{1}{\eta}\log\left(\frac{\varepsilon_{t-1}}{\varepsilon_{t}}\right)\right\}_{t\geq 1}$ is monotone non-decreasing and thus, for all $t\geq 1$:
\begin{align*}
 -\frac{1}{\eta}\log\left(\frac{\varepsilon_{t-1}}{\varepsilon_{t}}\right) &\ge -\frac{1}{\eta}\log\left(\frac{\varepsilon_{0}}{\varepsilon_{1}}\right)=-\frac{3}{\eta}\log\left(\frac{c+2}{c+1}\right)
\ge -\frac{3}{\eta}\log\left(1 + \frac{1}{c}\right) \geq -\frac{3}{\eta c} \geq -\frac{\delta}{4},
\end{align*}
where the last inequality follows from plugging our choice for $c$.

In addition, by our choice $\eta_t=\eta$ it holds that
\begin{align*}
\gamma_t=\gamma := \min\left\lbrace 1-4\eta^2\left(\beta_{X}^2+\beta_{yX}^2\right),1-2\eta^2\left(\beta_{y}^2+\beta_{Xy}^2\right)\right\rbrace \ge \frac{1}{2},
\end{align*}
and by our choice of $\varepsilon_t$ it holds that $\varepsilon_t
\le \left(\frac{\delta}{16\sqrt{2}\max\lbrace\beta_{X},\beta_{Xy}\rbrace\left(\sqrt{5}+\sqrt{\eta G}\right)}\right)^4\frac{1}{(t+1+c)^3}\le \left(\frac{\delta}{16\sqrt{2}\max\lbrace\beta_{X},\beta_{Xy}\rbrace\left(\sqrt{5}+\sqrt{\eta G}\right)}\right)^4$ and $\varepsilon_t\le1$. Therefore, it holds that
\begin{align*}
& -4\max\lbrace\beta_{X},\beta_{Xy}\rbrace\left(\sqrt{\frac{1}{\gamma}+8}\sqrt{\varepsilon_t}+\sqrt{\frac{\eta G}{\gamma}}\sqrt[4]{\varepsilon_t}\right) 
\\ & \ge -4\sqrt{2}\max\lbrace\beta_{X},\beta_{Xy}\rbrace\left(\sqrt{5}\sqrt{\varepsilon_t}+\sqrt{\eta G}\sqrt[4]{\varepsilon_t}\right)
\\ & \ge -4\sqrt{2}\max\lbrace\beta_{X},\beta_{Xy}\rbrace\left(\sqrt{5}+\sqrt{\eta G}\right)\sqrt[4]{\varepsilon_t} \ge -\frac{\delta}{4}.
\end{align*}

Thus, we have that for all $t\geq 1$:
\begin{align} \label{eq:thm:smoothConv:1}
& \frac{1}{\sqrt{2}}\left(2\max\lbrace\beta_{X},\beta_{Xy}\rbrace\sqrt{\frac{1}{\gamma}+8}+\left(1+\frac{2\sqrt{2r^*}}{\lambda_{r^*}(\X^*)}\right)G\right)^{-1} \nonumber
\\ & \ \ \ \left(\delta -\frac{1}{\eta}\log\left(\frac{\varepsilon_{t-1}}{\varepsilon_t}\right) +\frac{2\varepsilon_t}{\eta} -4\max\lbrace\beta_{X},\beta_{Xy}\rbrace\left(\sqrt{\frac{1}{\gamma}+8}\sqrt{\varepsilon_t}+\sqrt{\frac{\eta G}{\gamma}}\sqrt[4]{\varepsilon_t}\right)\right) \nonumber
\\ & \ge \frac{1}{2\sqrt{2}}\left(7\max\lbrace\beta_{X},\beta_{Xy}\rbrace+\left(1+\frac{2\sqrt{2r^*}}{\lambda_{r^*}(\X^*)}\right)G\right)^{-1}\delta.
\end{align}

Thus, in order to invoke Lemma \ref{lemma:series_error} for all $t\geq 1$ it suffices to prove that for all $t\geq 1$: i. $\lambda_{r^*+1}(\X_t) \leq \varepsilon_{t-1}/(n-r)$ and ii. $\sqrt{\breg_{\X}(\X^*,\X_{t})+\breg_{\y}(\y^*,\y_{t})} \leq \textrm{RHS of } \eqref{eq:thm:smoothConv:1}$. 

The requirement $\lambda_{r^*+1}(\X_t) \leq \varepsilon_{t-1}/(n-r)$ holds trivially by the design of the algorithm and since the SVD parameter satisfies $r=r^*$. 

We now prove by induction that indeed for all $t\geq 1$, $\sqrt{\breg_{\X}(\X^*,\X_{t})+\breg_{\y}(\y^*,\y_{t})} \leq \textrm{RHS of } \eqref{eq:thm:smoothConv:1}$. The base case $t=1$ clearly holds due to the choice of initialization. Now, if the assumption holds for all $i\in\{1,...,t\}$, then invoking Lemma \ref{lemma:series_error} for all $i\in\{1,...,t\}$ it holds that $\max\lbrace\breg_{\X}(\X^*,\X_{i+1})-\breg_{\X}(\X^*,\widehat{\X}_{i+1}),\breg_{\X}(\widehat{\Z}_{i+1},\Z_{i+1})\rbrace\le 2\varepsilon_i$ for all $i\in\{1,...,t\}$, where $\widehat{\X}_{t+1}$ and $\widehat{\Z}_{t+1}$ are the exact updates as defined in Algorithm \ref{alg:LRMP}. Thus, invoking Lemma \ref{lemma:decreasingRadius} with non-exact updates (i.e., $\theta=1$) for all $i=1,...,t$ we have that
\begin{align} \label{ineq:recursionOfRadious}
& \breg_{\X}(\X^*,\X_{t+1})+\breg_{\y}(\y^*,\y_{t+1}) \nonumber \\ & \le \breg_{\X}(\X^*,\X_{t})+\breg_{\y}(\y^*,\y_{t})+4\varepsilon_t+ 2\eta G\sqrt{\varepsilon_t} \nonumber
\\ &  \le \ldots \le \breg_{\X}(\X^*,\X_{1})+\breg_{\y}(\y^*,\y_{1})+4\sum_{i=1}^{t}\varepsilon_i+2\eta G\sum_{i=1}^{t}\sqrt{\varepsilon_i} \nonumber
\\ &  \leq \breg_{\X}(\X^*,\X_{1})+\breg_{\y}(\y^*,\y_{1})+\frac{\tilde{\varepsilon}_0^2}{4\max\{G^2\eta^2,1\}}\sum_{i=1}^{\infty}\frac{1}{(i+1)^3} +\frac{\tilde{\varepsilon}_0}{2}\sum_{i=1}^{\infty}\frac{1}{(i+1)^{3/2}} \nonumber
\\ & \le \breg_{\X}(\X^*,\X_{1})+\breg_{\y}(\y^*,\y_{1})+0.1\tilde{\varepsilon}_0^2 +0.9\tilde{\varepsilon}_0 \nonumber
\\ & \underset{(a)}{\le} \breg_{\X}(\X^*,\X_{1})+\breg_{\y}(\y^*,\y_{1})+R_0^2 \underset{(b)}{\le} 2R_0^2,
\end{align}
where (a) follows since
\begin{align*}
R_0 & = \frac{1}{4}\left(7\max\lbrace\beta_{X},\beta_{Xy}\rbrace+\left(1+\frac{2\sqrt{2r^*}}{\lambda_{r^*}(\X^*)}\right)G\right)^{-1}\delta
\le \frac{\delta}{4G}
\\ & \le \frac{\lambda_{n-r^*}(\nabla_{\X}f(\X^*,\y^*))-\lambda_{n}(\nabla_{\X}f(\X^*,\y^*))}{4\cdot\sup_{(\X,\y)\in\Sn\times\mathcal{K}}\Vert \nabla_{\X} f(\X,\y)\Vert_2} \le 1,
\end{align*}
and (b) follows from our initialization assumption and Lemma \ref{lem:warmstart}.

Thus, we obtain that 
\begin{align*}
\sqrt{\breg_{\X}(\X^*,\X_{t+1})+\breg_{\y}(\y^*,\y_{t+1})} & \le \sqrt{2}R_0  
\\ & = \frac{1}{2\sqrt{2}}\left(7\max\lbrace\beta_{X},\beta_{Xy}\rbrace+\left(1+\frac{2\sqrt{2r^*}}{\lambda_{r^*}(\X^*)}\right)G\right)^{-1}\delta,
\end{align*}
and the induction holds.

Invoking Lemma \ref{lemma:series_error} for all $t\geq 1$ guarantees that for all $t\ge 1$, $\breg_{\X}(\widehat{\Z}_{t+1},\Z_{t+1})\le 2\varepsilon_t$ and for all $t\ge 1$ and $\X\in\Sn$, $\breg_{\X}(\X,\X_{t+1})-\breg_{\X}(\X,\widehat{\X}_{t+1})\le 2\varepsilon_t$. Therefore, it holds from Lemma \ref{lemma:convergenceApproxMethod} with non-exact updates that for all $T\geq 1$,
\begin{align*}
& \max_{\y\in\mathcal{K}} f\left(\frac{1}{T}\sum_{t=1}^T \Z_{t+1},\y\right) -  \min_{\X\in\Sn} f\left(\X,\frac{1}{T}\sum_{t=1}^T\w_{t+1}\right)
\\ & \le \frac{D^2}{\eta T} + \frac{2\eta^{-1}+4(\beta_{X}+\beta_{yX})}{T}\sum_{t=1}^T \varepsilon_t +\frac{2G}{T}\sum_{t=1}^T\sqrt{\varepsilon_t}.
\end{align*}

In order to bound the RHS of this inequality, we note that the following inequalities hold:
\begin{align*}
& \sum_{t=1}^{T}\sqrt{\varepsilon_t} \leq \sum_{t=1}^{T}\frac{R_0^2}{4\max\{G\eta,1\}}\frac{1}{(t+1+c)^{3/2}} 
 < \frac{R_0^2}{4\max\{G\eta,1\}}\int_0^{\infty}\frac{dt}{(t+1)^{3/2}} \\ & = \frac{R_0^2}{2\max\{G\eta,1\}}.\\
& \sum_{t=1}^{T}\varepsilon_t \leq \sum_{t=1}^{T}\frac{R_0^4}{16\max\{G^2\eta^2,1\}}\frac{1}{(t+1+c)^{3}} < \frac{R_0^4}{16\max\{G^2\eta^2,1\}}\int_0^{\infty}\frac{dt}{(t+1)^{3}} 
\\ & = \frac{R_0^4}{32\max\{G^2\eta^2,1\}}.
\end{align*}

Plugging-in these  bounds we obtain the rate in the theorem.
\end{proof}

\subsection{Efficiently-computable certificates for convergence of Algorithm  \ref{alg:LRMP}}\label{sec:certificateMP}
Similarly to the Euclidean case (see Section \ref{sec:certificate}), Theorem \ref{thm:allPutTogetherVonNeumann} only holds when the iterates are within some bounded distance from the optimal solution. However, from a practical point of view, it might be impossible to determine whether an initialization choice indeed satisfies this condition. 
Let us recall that the convergence result of Theorem \ref{thm:allPutTogetherVonNeumann} is based on proving that the error terms in the convergence rate in Lemma  \ref{lemma:convergenceApproxMethod} (which establishes the convergence of a generic mirror-prox method with approximated updates) are appropriately bounded.
Thus, in the special case of the bregman distance induced by the von Neummann entropy (as in the primal updates of  Algorithm \ref{alg:LRMP}), it is of importance to establish an easily computable certificate that can ensure that the error terms associated with the approximated sequences  in Lemma \ref{lemma:convergenceApproxMethod} are indeed sufficiently bounded to guarantee the  convergence of  Algorithm \ref{alg:LRMP}. As already suggested in \cite{garber2020efficient}, such certificates arise from the bounds in \eqref{ineq:toUseCertificateConvergenceP} and \eqref{ineq:toUseCertificateConvergenceB}. For any iteration $t$ of the algorithm, if it can be verified that both
\begin{align*} 
\log\left(\frac{(n-r)\lambda_{r+1}(\P_t)}{\varepsilon_t \sum_{i=1}^n \lambda_i(\P_t)}\right) \le 2\varepsilon_t \quad \textrm{and} \quad   \log\left(\frac{(n-r)\lambda_{r+1}(\B_t)}{\varepsilon_t \sum_{i=1}^n \lambda_i(\B_t)}\right) \le 2\varepsilon_t,
\end{align*}
hold, then it immediately follows that 
$$\max\left\lbrace\max_{\X\in\mS_n}\breg_{\X}(\X,\X_{t+1})-\breg_{\X}(\X,\widehat{\X}_{t+1}),\breg_{\X}(\widehat{\Z}_{t+1},\Z_{t+1})\right\rbrace \le 2\varepsilon_t$$ also holds, as required. 

To avoid computing the trace of $\P_t$ and $\B_t$, as that would require calculating $\P_t$ and $\B_t$ which we are striving to avoid, the weaker conditions 
\begin{align*} 
\log\left(\frac{(n-r)\lambda_{r+1}(\P_t)}{\varepsilon_t \sum_{i=1}^k \lambda_i(\P_t)}\right) \le 2\varepsilon_t \quad \textrm{and} \quad   \log\left(\frac{(n-r)\lambda_{r+1}(\B_t)}{\varepsilon_t \sum_{i=1}^k \lambda_i(\B_t)}\right) \le 2\varepsilon_t,
\end{align*}
can be computed for some $k<<n$. Specifically, choosing $k=r+1$ requires merely increasing the necessary rank of the SVD computations by $1$ and as we demonstrate in Section \ref{sec:expr}, seems to work well in practice.

\section{Empirical Evidence}\label{sec:expr}

The goal of this section is to bring empirical evidence in support of our theoretical approach.
We consider various tasks that take the form of minimizing a composite objective, i.e., the sum of a smooth convex function and a nonsmooth convex function, where the nonsmoothness comes from either an $\ell_1$-norm or $\ell_2$-norm regularizer / penalty term, over a $\tau$-scaled spectrahedron. In all cases the nonsmooth objective can be written as a saddle-point with function $f(\X,\y)$ which is linear in $\y$ and in particular satisfies Assumption \ref{ass:struct}. 

For all tasks considered we generate random instances, and examine the sequences of iterates generated by Algorithm \ref{alg:EG} or Algorithm \ref{alg:LRMP} $\lbrace(\X_t,\y_t)\rbrace_{t\ge1}$, $\lbrace(\Z_{t},\w_t)\rbrace_{t\ge2}$, when initialized with simple initialization procedures. Out of both sequences generated, we choose our candidate for the optimal solution to be the iterate for which the dual-gap, which is a certificate for optimality, is smallest. See Appendix \ref{appendixDualGapCalculation}. We denote this solution by $(\X^*,\y^*)$.

We consider several different types of nonsmooth low-rank matrix recovery problems where the goal is to recover a ground-truth low-rank matrix $\M_0\in\mathbb{S}^n$ from some noisy observation of it $\M=\M_0+\N$, where $\N\in\mathbb{S}^{n}$ is a noise matrix. We measure the signal-to-noise ratio (SNR) as $\Vert\M_0\Vert_F^2\big/\Vert\N\Vert_F^2$. 
In all experiments we measure the relative initialization error by $\left\Vert\frac{\trace(\M_0)}{\tau}\X_1-\M_0\right\Vert_F^2\Big{/}\left\Vert\M_0\right\Vert_F^2$, and similarly we measure the relative recovery error by $\left\Vert\frac{\trace(\M_0)}{\tau}\X^*-\M_0\right\Vert_F^2\Big{/}\left\Vert\M_0\right\Vert_F^2$. Note that in some of the experiments we  take $\tau<\trace(\M_0)$ to prevent the method from overfitting the noise. In addition, we measure the (standard) strict complementarity parameter which corresponds to the eigen-gap $\lambda_{n-r}(\nabla_{\X}f(\X^*,\y^*))-\lambda_{n}(\nabla_{\X}f(\X^*,\y^*))$, $r = \rank(\M_0)$.

\subsection{Evidence for strict complementarity and performance of the Euclidean low-rank extragradient method}
We begin by providing extensive empirical support for both the plausibility of the strict complementarity assumption, and for the correct convergence (see explanation in the sequel) of our low-rank extragradient method, from simple ``warm-start'' initializations. Here we focus on the Euclidean setting since the low-rank extragradient method is simpler and more straightforward to implement than the low-rank mirror-prox with MEG updates --- Algorithm \ref{alg:LRMP}, which requires additional parameters and tuning. In particular, as discussed in Section \ref{sec:certificate} (see also reminder in the sequel), in the Euclidean setting it is straightforward to verify if the low-rank updates to the primal matrix variables match those of the exact extragradient method, and hence we can easily verify that our low-rank extragradient method indeed produces exactly the same sequences of iterates as its exact counterpart. 

The tasks considered in this section include 1. sparse PCA, 2. robust PCA, 3. low-rank and sparse recovery, 4. phase synchronization, and 5. linearly-constrained low-rank estimation, under variety of parameters.

In all experiments we use SVDs of rank $r=\rank(\M_0)$ to compute the projections in Algorithm \ref{alg:EG} according to the truncated projection given in \eqref{truncatedProjection}. To certify the correctness of these low-rank projections (that is, that they equal the exact Euclidean projection) we confirm that the inequality
\begin{align} \label{convergenceCondition:inExp}
\sum_{i=1}^r\lambda_i(\Q_j)\ge \tau+r\cdot\lambda_{r+1}(\Q_j)
\end{align}
always holds for $\Q_1=\X_t-\eta\nabla_{\X}f(\X_t,\Y_t)$ and $\Q_2=\X_t-\eta\nabla_{\X}f(\Z_{t+1},\W_{t+1})$ (see also Section \ref{sec:certificate}). Indeed, we can now already state our main observation from the experiments:

\smallskip
\framebox{\parbox{\dimexpr\linewidth-8\fboxsep-2\fboxrule}{\itshape%
In all tasks considered and for all random instances generated, throughout all iterations of Algorithm \ref{alg:EG}, when initialized with a simple ``warm-start'' strategy and when computing only rank-$r$ truncated projections, $r=\rank(\M_0)$, the truncated projections of $\Q_1=\X_t-\eta\nabla_{\X}f(\X_t,\Y_t)$ and $\Q_2=\X_t-\eta\nabla_{\X}f(\Z_{t+1},\W_{t+1})$ equal their exact full-rank counterparts. That is, Algorithm \ref{alg:EG}, using only rank-$r$ SVDs, computed exactly the same sequences of iterates it would have computed if using full-rank SVDs.
}}
\smallskip

Aside from the above observation, in the sequel we demonstrate that all models considered indeed satisfy that: 1. the returned solution is of the same rank as the ground-truth matrix and satisfies the strict complementarity condition with non-negligible parameter  (measured by the eigengap $\lambda_{n-r}(\nabla_{\X}f(\X^*,\y^*))-\lambda_{n}(\nabla_{\X}f(\X^*,\y^*))$), 2. the recovery error of the returned solution indeed improves significantly over the error of the initialization point.

\subsubsection{Sparse PCA.} \label{sec:sparsePCA}

We consider the sparse PCA problem in a well known convex formulation taken form \cite{d2007direct} and its equivalent saddle-point formulation:
\begin{align*}
\min_{\substack{\trace(\X) =1,\\ \X\succeq 0}} \langle{\X,-\M}\rangle + \lambda\Vert{\X}\Vert_1 =
\min_{\substack{\trace(\X) =1,\\ \X\succeq 0}} \max_{\Vert\Y\Vert_{\infty}\le1}\lbrace\langle\X,-\M\rangle+\lambda\langle\X,\Y\rangle\rbrace,
\end{align*}
where $\M = \z\z^{\top} + \frac{c}{2}(\N+\N^{\top})$ is a noisy observation of a rank-one matrix $\z\z^{\top}$, with $\z$ being a sparse unit vector. Each entry $\z_i$ is chosen to be $0$ with probability $0.9$ and $U\lbrace1,\ldots,10\rbrace$ with probability $0.1$, and then we normalize $\z$ to be of unit norm.

We test the results obtained when adding different magnitudes of Gaussian or uniform noise. 
We set the signal-to-noise ratio (SNR) to be a constant. Thus, we set the noise level to $c=\frac{2}{\textrm{SNR}\cdot\Vert\N+\N^{\top}\Vert_F}$ for our choice of $\textrm{SNR}$.

We initialize the $\X$ variable with the rank-one approximation of $\M$. That is, we take $\X_1=\u_1\u_1^{\top}$, where $\u_1$ is the top eigenvector of $\M$. For the $\Y$ variable we initialize it with $\Y_1=\sign(\X_1)$ which is a subgradient of $\Vert\X_1\Vert_1$.

We set the step-size to $\eta=1/(2\lambda)$ and we set the number of iterations to $T=1000$ and for any set of parameters we average the measurements over $10$ i.i.d. runs.

\begin{table}[h]
\caption{Numerical results for the tensor completion problem. Each result is the average of 10 i.i.d. runs.}\label{table:sparsePCAuniformBigNoise}
\begin{tabular*}{\textwidth}{@{\extracolsep\fill}lcccc} 
\toprule
 dimension (n) & $100$ & $200$ & $400$ & $600$ \\ \midrule
  \multicolumn{5}{c}{ $\downarrow$ $\N\sim U[0,1]$, $\textrm{SNR}=1$ $\downarrow$}\\ 
  \midrule
  $\lambda$ & $0.008$ & $0.004$ & $0.002$ & $0.0013$ \\ 
  initialization error & $0.5997$ & $0.6009$ & $0.5990$ & $0.6002$ \\ 
    recovery error & $0.0054$ & $0.0040$ & $0.0035$ & $0.0043$ \\ 
    dual gap & $4.1\times{10}^{-5}$ & $7.9\times{10}^{-5}$ & $4.9\times{10}^{-5}$ & $3.4\times{10}^{-6}$ \\
  $\lambda_{n-1}(\nabla_{\X}f(\X^*,\y^*))-\lambda_{n}(\nabla_{\X}f(\X^*,\y^*))$ & $0.8840$ & $0.8898$ & $0.8938$ & $0.8777$ \\ 
  \midrule 
  \multicolumn{5}{c}{ $\downarrow$ $\N\sim U[0,1]$, $\textrm{SNR}=0.05$ $\downarrow$}\\
  \midrule
  $\lambda$ & $0.04$ & $0.02$ & $0.01$ & $0.0067$ \\ 
  initialization error & $1.7456$ & $1.7494$ & $1.7566$ & $1.7625$ \\ 
  recovery error& $0.0425$ & $0.0244$ & $0.0149$ & $0.0100$ \\ 
    dual gap & $2.0\times{10}^{-9}$ & $5.8\times{10}^{-6}$ & $4.5\times{10}^{-4}$ & $0.0018$ \\
 $\lambda_{n-1}(\nabla_{\X}f(\X^*,\y^*))-\lambda_{n}(\nabla_{\X}f(\X^*,\y^*))$ & $0.7092$ & $0.7854$ & $0.8340$ & $0.8622$ \\ 
 \midrule   
  \multicolumn{5}{c}{ $\downarrow$ $\N\sim\mathcal{N}(0.5,\I_n)$, $\textrm{SNR}=1$ $\downarrow$}\\ 
  \midrule
 $\lambda$ & $0.006$ & $0.003$ & $0.0015$ & $0.001$ \\ 
  initialization error & $0.1584$ & $0.1464$ & $0.1443$ & $0.1411$ \\ 
  recovery error & $0.0059$ & $0.0033$ & $0.0019$ & $0.0015$ \\ 
    dual gap & $8.6\times{10}^{-4}$ & $0.0031$ & $0.0053$ & $0.0060$ \\ 
   $\lambda_{n-1}(\nabla_{\X}f(\X^*,\y^*))-\lambda_{n}(\nabla_{\X}f(\X^*,\y^*))$ & $0.8406$ & $0.8869$ & $0.9178$ & $0.9331$ \\ 
   \midrule 
  \multicolumn{5}{c}{ $\downarrow$ $\N\sim\mathcal{N}(0.5,\I_n)$, $\textrm{SNR}=0.05$ $\downarrow$}\\ \midrule
  $\lambda$ & $0.04$ & $0.02$ & $0.01$ & $0.005$ \\ 
  initialization error & $1.6701$ & $1.6620$ & $1.6542$ & $1.6610$ \\ 
  recovery error & $0.0502$ & $0.0234$ & $0.0137$ & $0.0109$ \\ 
    dual gap & $1.9\times{10}^{-5}$ & $0.0041$ & $0.0534$ & $0.0409$ \\ 
  $\lambda_{n-1}(\nabla_{\X}f(\X^*,\y^*))-\lambda_{n}(\nabla_{\X}f(\X^*,\y^*))$ & $0.2200$ & $0.4076$ & $0.5460$ & $0.6788$ \\ 
  \botrule
   \end{tabular*}
\end{table}

\subsubsection{Low-rank and sparse matrix recovery.}

We consider the problem of recovering a simultaneously low-rank and sparse covariance matrix \cite{lowRankAndSparse}, which can be written as the following saddle-point optimization problem:
\begin{align*}
\min_{\substack{\trace(\X) =1,\\ \X\succeq 0}}\frac{1}{2}\Vert{\X-\M}\Vert_F^2 + \lambda\Vert{\X}\Vert_1 =
\min_{\substack{\trace(\X) =\tau,\\ \X\succeq 0}} \max_{\Vert\Y\Vert_{\infty}\le1}\frac{1}{2}\Vert\X-\M\Vert_F^2+\lambda\langle\X,\Y\rangle,
\end{align*} 
where $\M = {\Z_0}{\Z_0}^{\top} + \frac{c}{2}(\N+\N^{\top})$ is a noisy observation of some low-rank and sparse covariance matrix ${\Z_0}{\Z_0}^{\top}$. We choose $\Z_0\in\mathbb{R}^{n\times r}$ to be a sparse matrix where each entry ${\Z_0}_{i,j}$ is chosen to be $0$ with probability $0.9$ and $U\lbrace1,\ldots,10\rbrace$ with probability $0.1$, and then we normalize $\Z_0$ to be of unit Frobenius norm. We choose $\N\sim\mathcal{N}(0.5,\I_n)$.

We test the model with $\rank(\Z_0\Z_0^{\top})=1,5,10$. 
We set the signal-to-noise ratio (SNR) to be a constant and set the noise level to $c=\frac{2\Vert\Z_0\Z_0^{\top}\Vert_F}{\textrm{SNR}\cdot\Vert\N+\N^{\top}\Vert_F}$ for our choice of $\textrm{SNR}$. 

We initialize the $\X$ variable with the rank-r approximation of $\M$. That is, we take $\X_1=\U_r\diag\left(\Pi_{\Delta_{\tau,r}}[\diag(-\bLambda_r)]\right)\U_r^{\top}$, where $\U_r\bLambda_r\U_r^{\top}$ is the rank-r eigen-decomposition of $\M$ and $\Delta_{\tau,r}=\lbrace\z\in\reals^r~|~\z\ge0,\ \sum_{i=1}^r \z_i=\tau\rbrace$ is the simplex of radius $\tau$ in $\reals^r$. For the $\Y$ variable we initialize it with $\Y_1=\sign(\X_1)$ which is a subgradient of $\Vert\X_1\Vert_1$.  

We set the step-size to $\eta=1$, $\tau=0.7\cdot\trace(\Z_0\Z_0^{\top})$, and the number of iterations in each experiment to $T=2000$. For each value of $r$ and $n$ we average the measurements over $10$ i.i.d. runs.

\begin{table}[h]
\caption{Numerical results for the low-rank and sparse matrix recovery problem.}\label{table:lowRank&SparseRank10}
\begin{tabular*}{\textwidth}{@{\extracolsep\fill}lcccc}  \toprule 
  dimension (n) & $100$ & $200$ & $400$ & $600$ \\ 
  \midrule
  \multicolumn{5}{c}{$\downarrow$ $r=\rank(\Z_0\Z_0^{\top})=1$, $\textrm{SNR}=0.48$ $\downarrow$}\\ 
  \midrule
  $\lambda$ & $0.0012$ & $0.0035$ & $0.0016$ & $0.001$ \\ 
  initialization error & $0.4562$ & $0.4471$ & $0.4507$ & $0.4450$ \\ 
    recovery error & $0.0364$ & $0.0193$ & $0.0160$ & $0.0168$ \\ 
    dual gap & $0.0083$ & $0.0086$ & $0.0020$ & $4.2\times{10}^{-4}$ \\ 
 $\lambda_{n-r}(\nabla_{\X}f(\X^*,\y^*))-\lambda_{n}(\nabla_{\X}f(\X^*,\y^*))$ & $0.0628$ & $0.1439$ & $0.1258$ & $0.1069$ \\ 
 \midrule
\multicolumn{5}{c}{$\downarrow$ $r=\rank(\Z_0\Z_0^{\top})=5$, $\textrm{SNR}=2.4$ $\downarrow$}\\ 
\midrule
 $\lambda$ & $0.0012$ & $0.0006$ & $0.0003$ & $0.0002$ \\ 
  initialization error & $0.2132$ & $0.2103$ & $0.1983$ & $0.1907$ \\ 
    recovery error & $0.0641$ & $0.0478$ & $0.0349$ & $0.0274$ \\ 
    dual gap & $9.0\times{10}^{-4}$ & $4.3\times{10}^{-4}$ & $1.4\times{10}^{-4}$ & $7.3\times{10}^{-5}$ \\ 
$\lambda_{n-r}(\nabla_{\X}f(\X^*,\y^*))-\lambda_{n}(\nabla_{\X}f(\X^*,\y^*))$ & $0.0148$ & $0.0200$ & $0.0257$ & $0.0277$ \\ 
\midrule
\multicolumn{5}{c}{$\downarrow$ $r=\rank(\Z_0\Z_0^{\top})=10$, $\textrm{SNR}=4.8$ $\downarrow$}\\ 
\midrule
  $\lambda$ & $0.0007$ & $0.0004$ & $0.0002$ & $0.0001$ \\ 
  initialization error & $0.1855$ & $0.1661$ & $0.1527$ & $0.1473$ \\ 
    recovery error & $0.0702$ & $0.0403$ & $0.0268$ & $0.0356$ \\ 
    dual gap & $4.9\times{10}^{-4}$ & $6.6\times{10}^{-4}$ & $4.2\times{10}^{-4}$ & $3.4\times{10}^{-5}$ \\
 $\lambda_{n-r}(\nabla_{\X}f(\X^*,\y^*))-\lambda_{n}(\nabla_{\X}f(\X^*,\y^*))$ & $0.0072$ & $0.0142$ & $0.0187$ & $0.0160$ \\ 
 \botrule
   \end{tabular*}
\end{table}

\subsubsection{Robust PCA.}

We consider the robust PCA problem \cite{robustPCA1} in the following formulation:
\begin{align*}
\min_{\substack{\trace(\X) =\tau,\\ \X\succeq 0}}\Vert{\X-\M}\Vert_1 =
\min_{\substack{\trace(\X) =\tau,\\ \X\succeq 0}} \max_{\Vert\Y\Vert_{\infty}\le1}\langle\X-\M,\Y\rangle,
\end{align*}
where $\M = r\Z_0\Z_0^{\top} + \frac{1}{2}(\N+\N^{\top})$ is a sparsely-corrupted observation of some rank-r  matrix $\Z_0\Z_0^{\top}$. We choose $\Z_0\in\mathbb{R}^{n\times r}$ to be a random unit Frobenius norm matrix. For $\N\in\reals^{n\times n}$, we choose each entry to be $0$ with probability $1-1/\sqrt{n}$ and otherwise $1$ or $-1$ with equal probability.

We initialize the $\X$ variable with the projection $\X_1=\Pi_{\lbrace\trace(\X)=\tau,\ \X\succeq 0\rbrace}[\M]$, and the $\Y$ variable with $\Y_1=\sign(\X_1-\M)$.

We test the model with $\rank(\Z_0\Z_0^{\top})=1,5,10$. For $\rank(\Z_0\Z_0^{\top})=1$ we set the step-size to $\eta=n/10$ and for $\rank(\Z_0\Z_0^{\top})=5,10$ we set it to $\eta=1$. We set $\tau=0.95\cdot\trace(r\Z_0\Z_0^{\top})$. For each value of $r$ and $n$ we average the measurements over $10$ i.i.d. runs.

\begin{table}[h]
\caption{Numerical results for the robust PCA problem.} \label{table:robustPCArank10}
\begin{tabular*}{\textwidth}{@{\extracolsep\fill}lcccc} 
\toprule 
 dimension (n) & $100$ & $200$ & $400$ & $600$ \\ 
 \midrule
 \multicolumn{5}{c}{$\downarrow$ $r=\rank(\Z_0\Z_0^{\top})=1$, $T=3000$ $\downarrow$}\\ \midrule
SNR & $0.0021$ & $7.2\times{10}^{-4}$ & $2.5\times{10}^{-4}$ & $1.3\times{10}^{-4}$ \\ 
  initialization error & $1.3511$ & $1.3430$ & $1.2889$ & $1.2606$ \\ 
    recovery error & $0.0084$ & $0.0107$ & $0.0109$ & $0.0107$ \\ 
    dual gap & $0.0016$ & $0.0029$ & $0.0044$ & $0.0069$ \\ 
 $\lambda_{n-r}(\nabla_{\X}f(\X^*,\y^*))-\lambda_{n}(\nabla_{\X}f(\X^*,\y^*))$ & $15.5944$ & $41.2139$ & $85.8117$ & $140.5349$ \\ 
 \midrule
\multicolumn{5}{c}{$\downarrow$ $r=\rank(\Z_0\Z_0^{\top})=5$, $T=20,000$ $\downarrow$}\\ \midrule
 SNR & $0.0110$ & $0.0038$ & $0.0013$ & $6.9\times{10}^{-4}$ \\ 
  initialization error & $1.5501$ & $1.5527$ & $1.5221$ & $1.4833$ \\ 
    recovery error & $0.0092$ & $0.0092$ & $0.0087$ & $0.0075$  \\ 
    dual gap & $0.0084$ & $0.0390$ & $0.1866$ & $0.4721$ \\ 
 $\lambda_{n-r}(\nabla_{\X}f(\X^*,\y^*))-\lambda_{n}(\nabla_{\X}f(\X^*,\y^*))$ & $7.6734$ & $26.2132$ & $66.1113$ & $108.7215$  \\ 
 \midrule
\multicolumn{5}{c}{$\downarrow$ $r=\rank(\Z_0\Z_0^{\top})=10$, $T=30,000$ $\downarrow$}\\ \midrule
SNR & $0.0229$ & $0.0077$ & $0.0026$ & $0.0014$ \\ 
  initialization error & $1.5729$ & $1.6485$ & $1.6317$ & $1.5949$ \\
  recovery error & $0.0079$ & $0.0081$ & $0.0073$ & $0.0065$  \\ 
  dual gap & $0.0139$ & $0.0338$ & $0.1533$ & $0.3561$ \\ 
$\lambda_{n-r}(\nabla_{\X}f(\X^*,\y^*))-\lambda_{n}(\nabla_{\X}f(\X^*,\y^*))$ & $1.7945$ & $16.9890$ & $48.9799$ & $82.2727$  \\
 \botrule
   \end{tabular*}
\end{table}

\subsubsection{Phase synchronization.}

We consider the phase synchronization problem (see for instance \cite{phaseSyncronization1}) which can be written as:
\begin{align} \label{phaseSync}
\max_{\substack{\z\in\mathbb{C}^n ,\\ \vert z_j\vert=1\ \forall j\in[n]}}\z^*\M\z,
\end{align}
where $\M = \z_0\z_0^{*} + c\N$ is a noisy observation of some rank-one matrix $\z_0\z_0^{*}$ such that $\z_0\in\mathbb{C}^n$ and ${\z_0}_j=e^{i\theta_j}$ where $\theta_j\in[0,2\pi]$. We follow the statistical model in \cite{phaseSyncronization1} where the noise matrix $\N\in\mathbb{C}^{n\times{}n}$ is chosen such that every entry is
\begin{align*}
\N_{jk}=\Bigg\lbrace\begin{array}{lc}\mathcal{N}(0,1)+i\mathcal{N}(0,1) & j<k
\\ \overline{\N}_{kj} & j>k 
\\ 0 & j=k\end{array}.
\end{align*}
It is known that for a large $n$ and $c=\mathcal{O}\left(\sqrt{\frac{n}{\log{}n}}\right)$, with high probability the SDP relaxation of \eqref{phaseSync} is able to recover the original signal (see \cite{phaseSyncronization1}).

We solve a penalized version of the SDP relaxation of \eqref{phaseSync} which can be written as the following saddle-point optimization problem:
\begin{align*}
\min_{\substack{\trace(\X) =n,\\ \X\succeq 0}}\langle{\X,-\M}\rangle + \lambda\Vert{\textrm{diag}(\X)-\overrightarrow{\mathbf{1}}}\Vert_2 =
\min_{\substack{\trace(\X)=n,\\ \X\succeq 0}} \max_{\Vert\y\Vert_2\le1}\langle\X,-\M\rangle+\lambda\langle\diag(\X)-\overrightarrow{\mathbf{1}},\y\rangle,
\end{align*}
where $\overrightarrow{\mathbf{1}}$ is the all-ones vector.

While the phase synchronization problem is formulated over the complex numbers, extending our results to handle this model is straightforward.

We initialize the $\X$ variable with the rank-one approximation of $\M$. That is, we take $\X_1=n\u_1\u_1^{*}$, where $\u_1$ is the top eigenvector of $\M$. For the $\y$ variable we initialize it with $\y_1=(\diag(\X_1)-\overrightarrow{\mathbf{1}})/\Vert\diag(\X_1)-\overrightarrow{\mathbf{1}}\Vert_2$.

We set the noise level to $c=0.18\sqrt{n}$. We set the number of iterations in each experiment to $T=10,000$ and for each choice of $n$ we average the measurements over $10$ i.i.d. runs.


\begin{table}[h]
\caption{Numerical results for the phase synchronization problem.}\label{table:phaseSync}
\begin{tabular*}{\textwidth}{@{\extracolsep\fill}lcccc} 
\toprule 
 dimension (n) & $100$ & $200$ & $400$ & $600$ \\ \midrule
 SNR & $0.1553$ & $0.0775$ & $0.0387$ & $0.0258$ \\ 
$\lambda$ & $200$ & $600$ & $1600$ & $2800$ \\
$\eta$ & $1/400$ & $1/800$ & $1/1800$ & $1/1800$ \\
  initialization error & $0.1270$ & $0.1255$ & $0.1284$ & $0.1323$ \\ 
  recovery error & $0.0698$ & $0.0659$ & $0.0683$ & $0.0719$  \\ 
  dual gap & $7.8\times10^{-8}$ & $3.9\times10^{-5}$ & $0.1553$ & $0.5112$ \\
  $\lambda_{n-1}(\nabla_{\X}f(\X^*,\y^*))-\lambda_{n}(\nabla_{\X}f(\X^*,\y^*))$ & $39.8591$ & $78.9982$ & $150.3524$ & $217.06$  \\  
 $\Vert\diag(\X^*)-\overrightarrow{\mathbf{1}}\Vert_2$ & $3.2\times10^{-10}$ & $2.1\times10^{-8}$ & $5.1\times10^{-7}$ & $3.7\times10^{-7}$ \\ 
 \botrule
   \end{tabular*}
\end{table}

\subsubsection{Linearly constrained low-rank matrix estimation.}

We consider a linearly constrained low-rank matrix estimation problem, which we consider in the following penalized formulation:
\begin{align*}
\min_{\substack{\trace(\X)=1,\\ \X\succeq 0}} \langle\X,-\M\rangle+\lambda\Vert\mathcal{A}(\X)-\b\Vert_2 = \min_{\substack{\trace(\X)=1,\\ \X\succeq 0}} \max_{\Vert\y\Vert_2\le1}\langle\X,-\M\rangle+\lambda\langle\mathcal{A}(\X)-\b,\y\rangle,
\end{align*}

where $\M = \z_0\z_0^{\top} + \frac{c}{2}(\N+\N^{\top})$ is the noisy observation of some rank-one matrix $\z_0\z_0^{\top}$ such that $\Vert\z_0\Vert_2=1$ and the noise matrix is chosen $\N\sim\mathcal{N}(0,\I_n)$. We take $\mathcal{A}(\X)=(\langle\A_i,\X\rangle,\ldots,\langle\A_m,\X\rangle)^{\top}$ with matrices $\A_1,\ldots,\A_m\in\mathbb{S}^n$ of the form $\A_i=\v_i\v_i^{\top}$ such that $\v_i\sim\mathcal{N}(0,1)$. We take $\b\in\reals^m$ such that $b_i=\langle\A_i,\z_0\z_0^{\top}\rangle$.

We initialize the $\X$ variable with the rank-one approximation of $\M$. That is, we take $\X_1=\u_1\u_1^{\top}$, where $\u_1$ is the top eigenvector of $\M$. The $\y$ variable is initialized with $\y_1=(\mathcal{A}(\X_1)-\b)/\Vert\mathcal{A}(\X_1)-\b\Vert_2$.

We set the number of constraints to $m=n$, the penalty parameter to $\lambda=2$, and the step-size to $\eta=1/(2\lambda)$. We set the number of iterations in each experiment to $T=2000$ and for each value of $n$ we average the measurements over  $10$ i.i.d. runs.


\begin{table}[h]
\caption{Numerical results for the linearly constrained low-rank matrix estimation problem.} \label{table:linearlyConstrained}
\begin{tabular*}{\textwidth}{@{\extracolsep\fill}lcccc}  
\toprule
 dimension (n) & $100$ & $200$ & $400$ & $600$ \\ 
 \midrule
$\textrm{SNR}$ & $0.15$ & $0.075$ & $0.04$ & $0.027$ \\ 
  initialization error & $0.1219$ & $0.1324$ & $0.1242$ & $0.1228$ \\ 
  recovery error & $0.0437$ & $0.0617$ & $0.0685$ & $0.0735$  \\ 
  dual gap & $5.3\times10^{-11}$ & $5.0\times10^{-12}$ & $8.5\times10^{-12}$ & $2.3\times10^{-11}$ \\
  $\lambda_{n-1}(\nabla_{\X}f(\X^*,\y^*))-\lambda_{n}(\nabla_{\X}f(\X^*,\y^*))$ & $0.2941$ & $0.3409$ & $0.4690$ & $0.5069$  \\ 
$\Vert\mathcal{A}(\X^*)-\b\Vert_2$ & $0.0080$ & $0.0082$ & $0.0079$ & $0.0073$ \\ 
\botrule
   \end{tabular*}
\end{table}

\subsection{Comparison of low-rank mirror-prox methods}
In this section we empirically compare between the low-rank Euclidean extragradient method, denoted by Eue in Figure \ref{fig:MPvsEG}, and the low-rank mirror-prox which also updates the dual variables via standard Euclidean steps, but updates the primal variables via matrix exponentiated gradient steps which correspond to the bregman distance induced by the von Neumann entropy (i.e., Algorithm \ref{alg:LRMP}, where the dual variables are updated exactly as in the extragradient method --- Algorithm \ref{alg:EG}), denoted by VNE in Figure \ref{fig:MPvsEG}. Thus, the two methods differ only w.r.t. to the primal steps they compute. We compare both methods on the following low-rank and sparse least squares problem which can be written as:
\begin{align*}
\min_{\substack{\trace(\X) =\tau,\\ \X\succeq 0}}\frac{1}{2}\Vert\mathcal{A}(\X)-\b\Vert_{2}^2+\lambda\Vert\X\Vert_1 = \min_{\substack{\trace(\X) =\tau,\\ \X\succeq 0}}\max_{\Vert\Y\Vert_{\infty}\le1} \frac{1}{2}\Vert\mathcal{A}(\X)-\b\Vert_{2}^2+\lambda\langle\X,\Y\rangle.
\end{align*}

We choose this task for the comparison since the linear operator $\mA$ can create a significant difference between the primal smoothness parameters measured with respect to the Euclidean norm and with respect to the spectral norm. Also, the dual variable $\Y$ is maximized over the $\ell_{\infty}$-norm ball, for which the natural bregman distance is the Euclidean distance, which is used by both methods to update the dual variables. 

We set $\M = n\Z_0\Z_0^{\top}$ to be the ground-truth low-rank matrix, where $\Z_0\in\mathbb{R}^{n\times r}$ is a sparse matrix where each entry $\Z_0(i,j)$ is chosen to be $0$ with probability $0.9$ and a random standard Gaussian entry with probability $0.1$, and then normalizing $\Z_0$ to have unit Frobenius norm. We take $\mA(\X) = (\langle\A_1,\X\rangle,\ldots,\langle\A_m,\X\rangle)^{\top}$ with matrices $\A_1,\ldots,\A_m\in\mathbb{S}^n$ of the form $\A_i=\sum_{j=1}^k \v_i^{(j)}{\v_i^{(j)}}^{\top}$ such that $\v_i^{(j)}\sim\mathcal{N}(0,1)$. We set $\b_0\in\reals^m$ such that $\b_0(i)=\langle\A_i,\M\rangle$ and then we take $\b=\b_0+\n$ such that $\n=0.5\Vert\b_0\Vert_2\v$ for a random unit vector $\v\in\reals^{m}$. 

For both methods we initialize the $\X$ variable as
$\X_1 = \V_r\diag\left(\Pi_{\Delta_{\tau,r}}[\diag(-\bLambda_r)]\right)\V_r^{\top}$,
where $\Delta_{\tau,r}=\lbrace\z\in\reals^r~|~\z\ge0,\ \sum_{i=1}^r \z_i=\tau\rbrace$ is the simplex of radius $\tau$ in $\reals^r$, and $\Pi_{\Delta_{\tau,r}}[\cdot]$ denotes the Euclidean projection over it, $\V_r\bLambda_r\V_r$ is the rank-$r$ eigen-decomposition of $-\nabla_{\X}(\frac{1}{2}\Vert\mathcal{A}(\tau\U\U^{\top})-\b\Vert_{2}^2)=-\mathcal{A}^{\top}(\mathcal{A}(\tau\U\U^{\top})-\b)$, and $\U\in\reals^{n\times r}$ is produced by taking a random matrix with standard Gaussian entries and normalizing it to have a unit Frobenius norm. The $\Y$ variable is initialized such that $\Y_1=\sign(\X_1)$.

We set the number of measurements to $m=25n\cdot\rank(\M)$, the sequence $\lbrace\varepsilon_t\rbrace_{t=0}^T$ to $\varepsilon_t=\frac{1}{(t+c+1)^3}$ for some $c\in\reals$, and $\tau=0.5\cdot\trace(\M)$. For each choice of $\rank(\M)$ and $\rank(\A_i)$ we average the measurements over $10$ i.i.d. runs.

In Figure \ref{fig:MPvsEG} we compare the low-rank Euclidean extragradient method and the low-rank von Neumann entropy mirror-prox method. For each rank of $\M$ and each rank of the matrices $\A_i$, $i=1,\ldots,m$, we compare the convergence of both methods using several different magnitudes of step-sizes with respect to the function value and relative recovery error. We compute the basic step-size $\eta_0$ using the smoothness parameter 
$\beta_{X}=\left\Vert (\textrm{vec}(\A_1),\ldots,\textrm{vec}(\A_m))^{\top}\right\Vert^2_2$, where $\textrm{vec}(\cdot)$ denotes the vectorization of the input matrix, and so $(\textrm{vec}(\A_1),\ldots,\textrm{vec}(\A_m))$ is an $n^2\times m$ matrix. To check the correctness of the low-rank Euclidean projections, we confirm that the inequality $\sum_{i=1}^r\lambda_i(\Q_j)\ge \tau+r\cdot\lambda_{r+1}(\Q_j)$
always holds for $\Q_1=\X_t-\eta\nabla_{\X}f(\X_t,\Y_t)$ and $\Q_2=\X_t-\eta\nabla_{\X}f(\Z_{t+1},\W_{t+1})$ (see also Section \ref{sec:certificate}). Likewise, to certify the correctness of the convergence of the low-rank von Neumann entropy mirror-prox method, we check that the conditions 
\begin{align} \label{ineq:condMPexperiments}
\log\left(\frac{(n-r)\lambda_{r+1}(\P_t)}{\varepsilon_t \sum_{i=1}^n \lambda_i(\P_t)}\right) \le 2\varepsilon_t\quad  \textrm{and}\quad \log\left(\frac{(n-r)\lambda_{r+1}(\B_t)}{\varepsilon_t \sum_{i=1}^n \lambda_i(\B_t)}\right) \le 2\varepsilon_t
\end{align}
always hold for $\B_t$ and $\P_t$ as defined in Algorithm \ref{alg:LRMP} (see also Section \ref{sec:certificateMP}), which guarantee that the error from the approximated steps is negligible. The only case where the conditions in \eqref{ineq:condMPexperiments} were not met was for the case of $\rank(\M)=3$, $\rank(\A_i)=1$ and step-size of $25\eta_0$, but as can be seen in Figure \ref{fig:MPvsEG}, the low-rank von Neumann entropy mirror-prox method does not converge in this case.
In general, as can be seen in Figure \ref{fig:MPvsEG}, for most cases the low-rank von Neumann entropy mirror-prox method converges faster than the low-rank extragradient method in terms of both the function value and the relative recovery error.

In Figure \ref{fig:MPvsLowrank} we compare the exact mirror-prox method with von Neumann entropy-based primal updates and Euclidean dual updates (see Algorithm \ref{alg:EGMP}), denoted by VNE  in Figure \ref{fig:MPvsLowrank}, and its low-rank variant, as in Algorithm \ref{alg:LRMP}, denoted by low-rank VNE in Figure \ref{fig:MPvsLowrank}. Here too, the conditions in \eqref{ineq:condMPexperiments} hold throughout all iterations, and as can be seen, both variants converge similarly in terms of the function value. 
In addition, we measure the bregman distances $\breg_{\X}(\widetilde{\Z}_t,\Z_t)$ and $\breg_{\X}(\widetilde{\X}_t,\X_t)$ where $\lbrace\Z_t\rbrace_{t\ge2}$ and $\lbrace\X_t\rbrace_{t\ge1}$ are the iterates obtained by the low-rank variant and $\lbrace\widetilde{\Z}_t\rbrace_{t\ge2}$ and $\lbrace\widetilde{\X}_t\rbrace_{t\ge1}$ are the iterates obtained by the standard exact variant. It can be seen in Figure \ref{fig:MPvsLowrank} that the distances between the iterates of the two methods indeed decay quickly.


\begin{figure}
\centering
  \begin{subfigure}[t]{0.311\textwidth}
    \includegraphics[width=\textwidth]{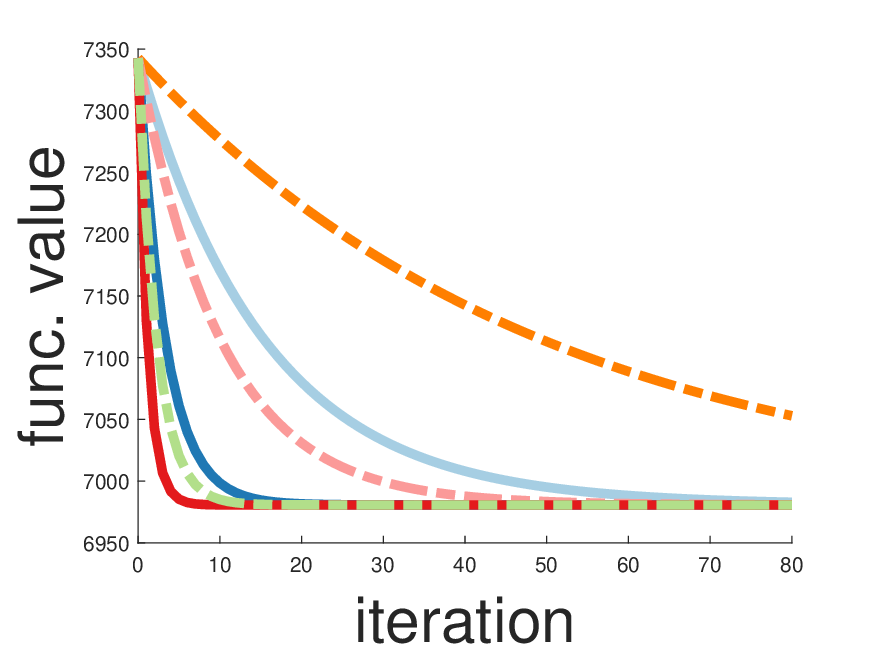}   
  \end{subfigure}\hfil
    \begin{subfigure}[t]{0.438\textwidth}
    \includegraphics[width=\textwidth]{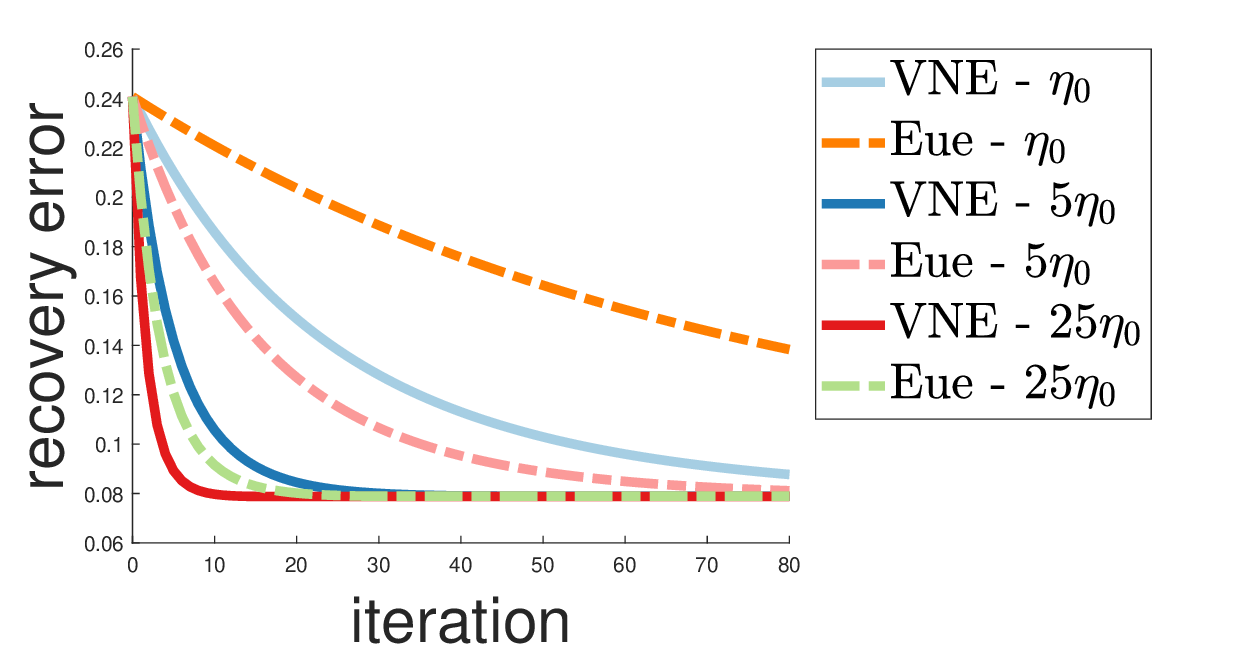} 
  \end{subfigure}
    \caption*{$\lambda=0.001$, $\rank(\M)=1$, $\rank(\A_i)=1$, $c=5000$}
  \medskip
    \begin{subfigure}[t]{0.311\textwidth}
    \includegraphics[width=\textwidth]{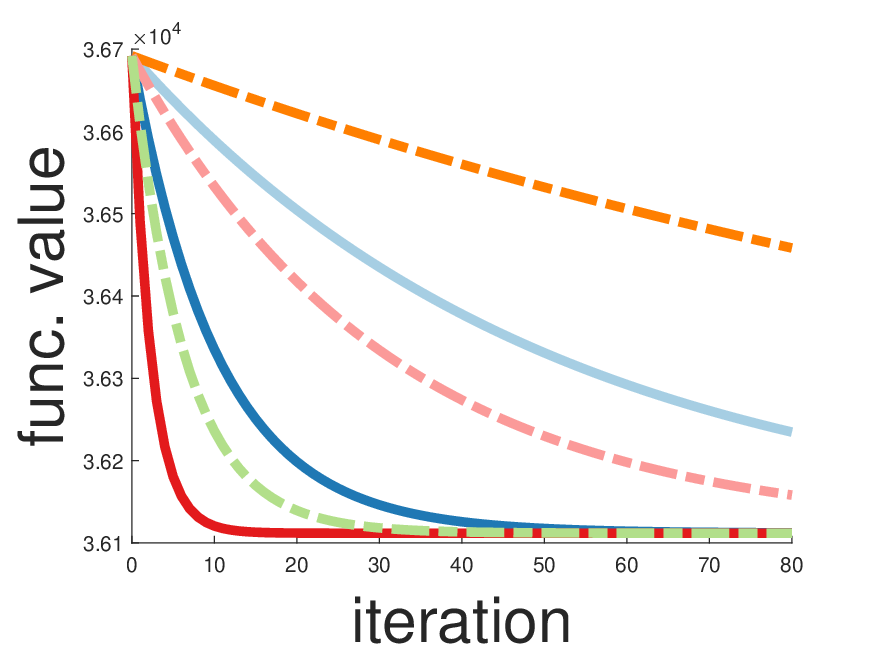}   
  \end{subfigure}\hfil
    \begin{subfigure}[t]{0.438\textwidth}
    \includegraphics[width=\textwidth]{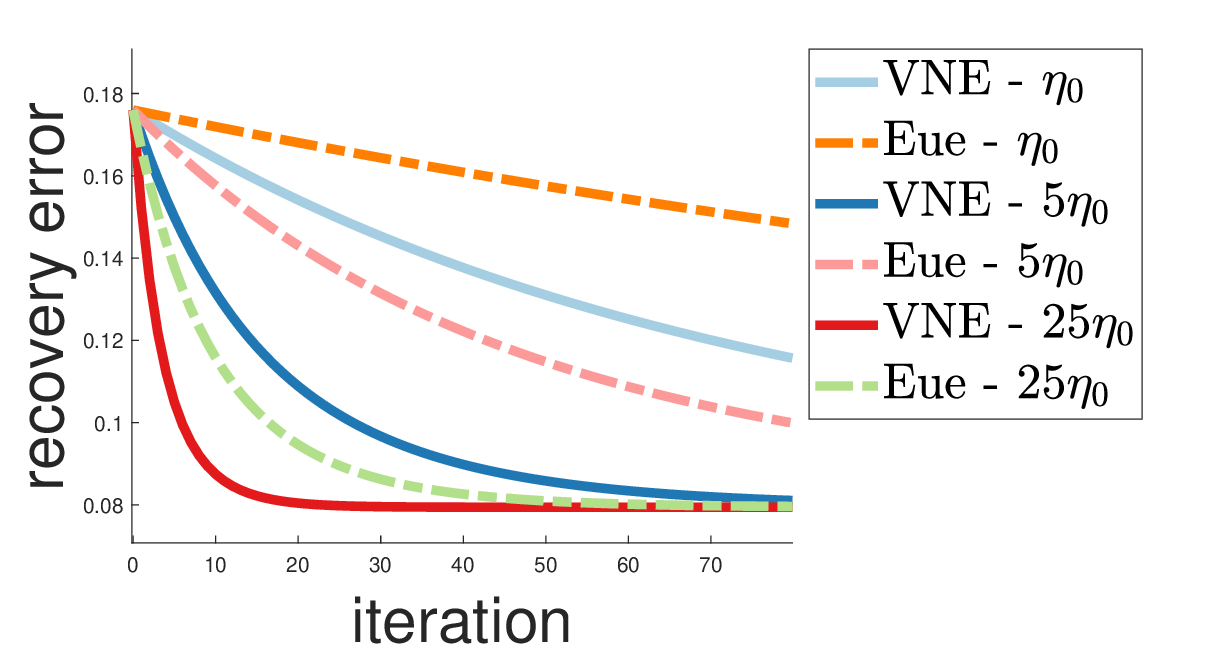} 
  \end{subfigure}
    \caption*{$\lambda=0.0001$, $\rank(\M)=1$, $\rank(\A_i)=3$, $c=10000$}
    \medskip
    \begin{subfigure}[t]{0.311\textwidth}
    \includegraphics[width=\textwidth]{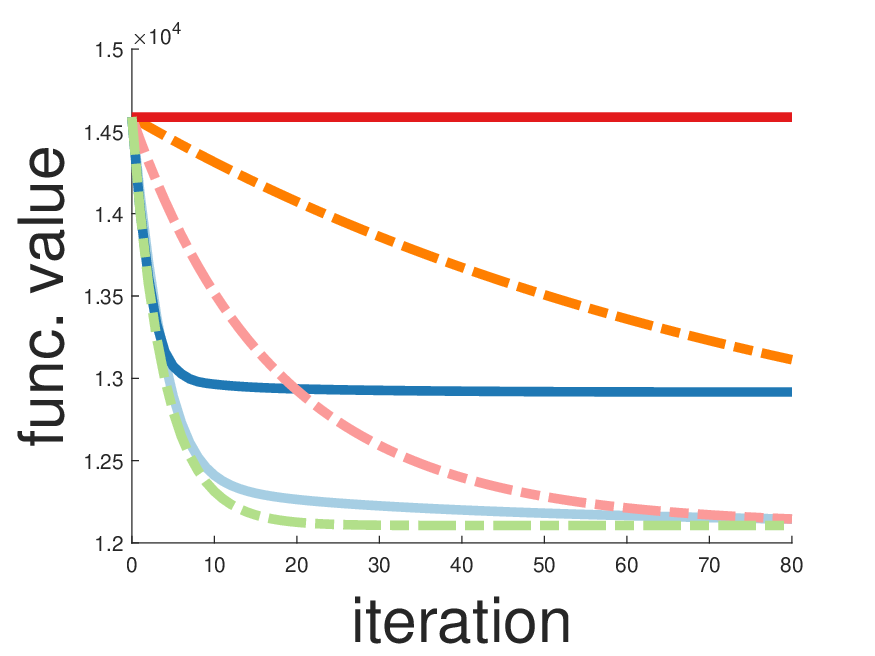}   
  \end{subfigure}\hfil
    \begin{subfigure}[t]{0.438\textwidth}
    \includegraphics[width=\textwidth]{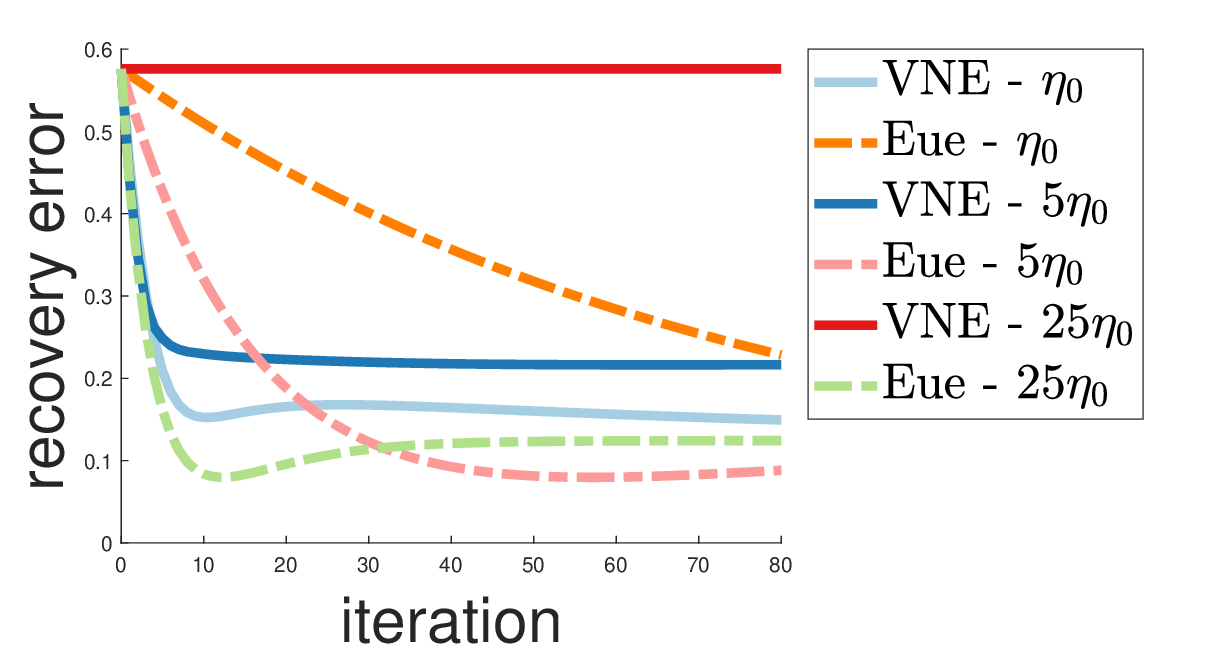} 
  \end{subfigure}
  \caption*{$\lambda=0.0001$, $\rank(\M)=3$, $\rank(\A_i)=1$, $c=10000$}
  \caption{Comparison between the Euclidean extragradient method (Eue) and the von Neumann entropy mirror-prox method (VNE)  with several different step-sizes. 
  } \label{fig:MPvsEG}
\end{figure}

\begin{figure}
\centering
  \begin{subfigure}[t]{0.33\textwidth}
    \includegraphics[width=\textwidth]{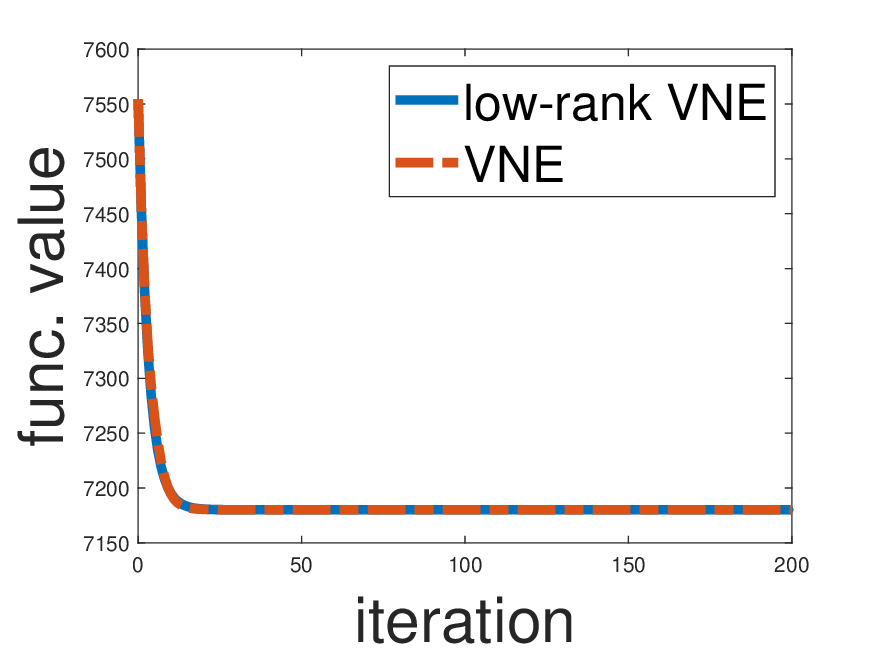}
  \end{subfigure}\hfil
  \begin{subfigure}[t]{0.33\textwidth}
    \includegraphics[width=\textwidth]{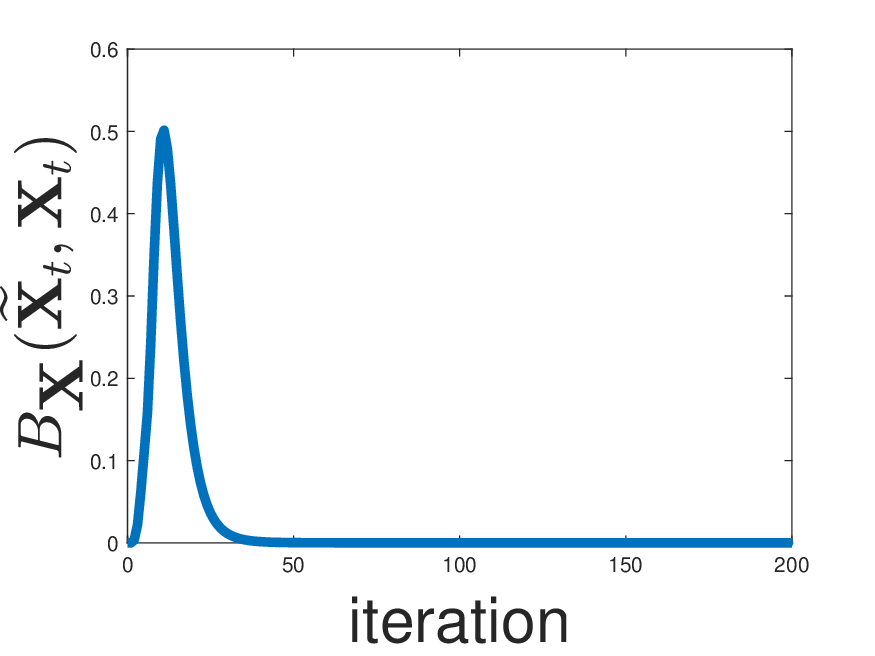} 
  \end{subfigure}\hfil
    \begin{subfigure}[t]{0.33\textwidth}
    \includegraphics[width=\textwidth]{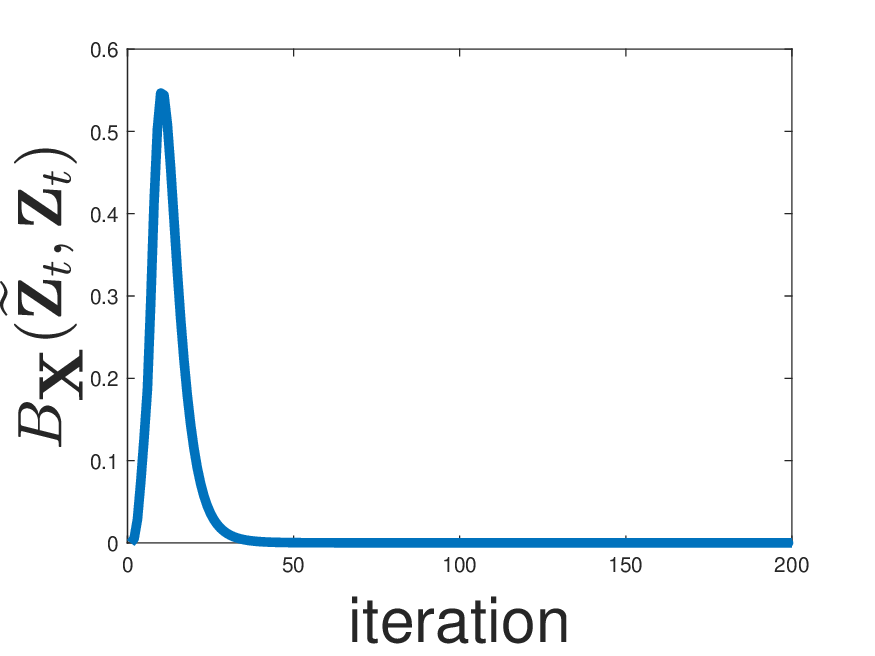}
  \end{subfigure}
    \caption*{$\eta=0.1$, $\lambda=0.001$, $\rank(\M)=1$, $\rank(\A_i)=1$, $c=5000$}
\medskip
  \begin{subfigure}[t]{0.33\textwidth}
    \includegraphics[width=\textwidth]{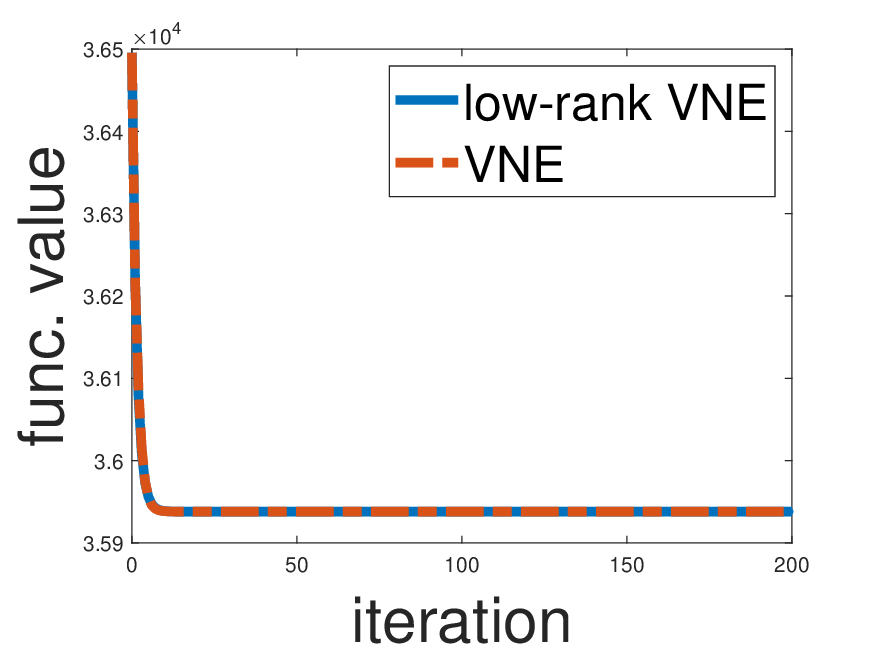}
  \end{subfigure}\hfil
  \begin{subfigure}[t]{0.33\textwidth}
    \includegraphics[width=\textwidth]{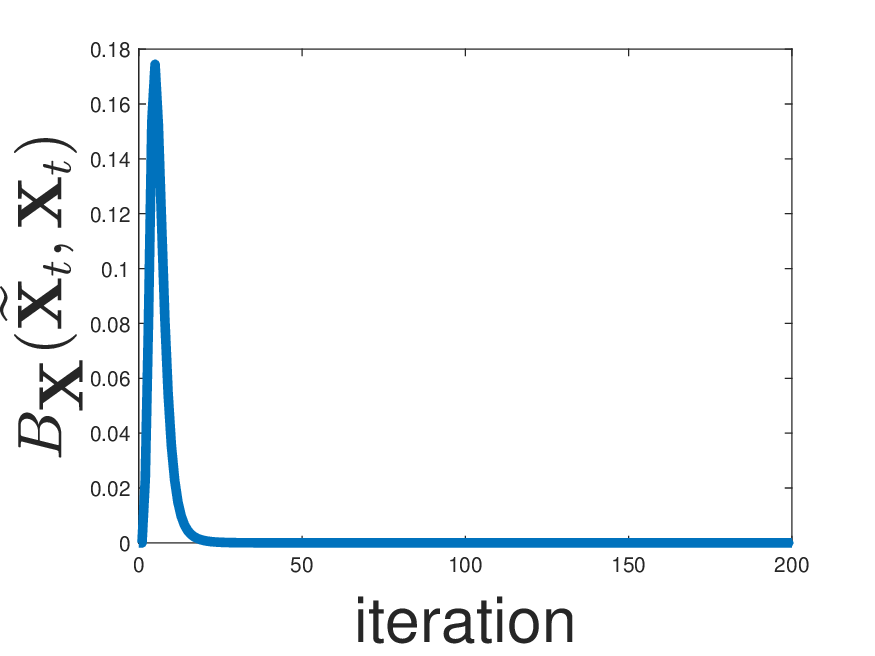} 
  \end{subfigure}\hfil
    \begin{subfigure}[t]{0.33\textwidth}
    \includegraphics[width=\textwidth]{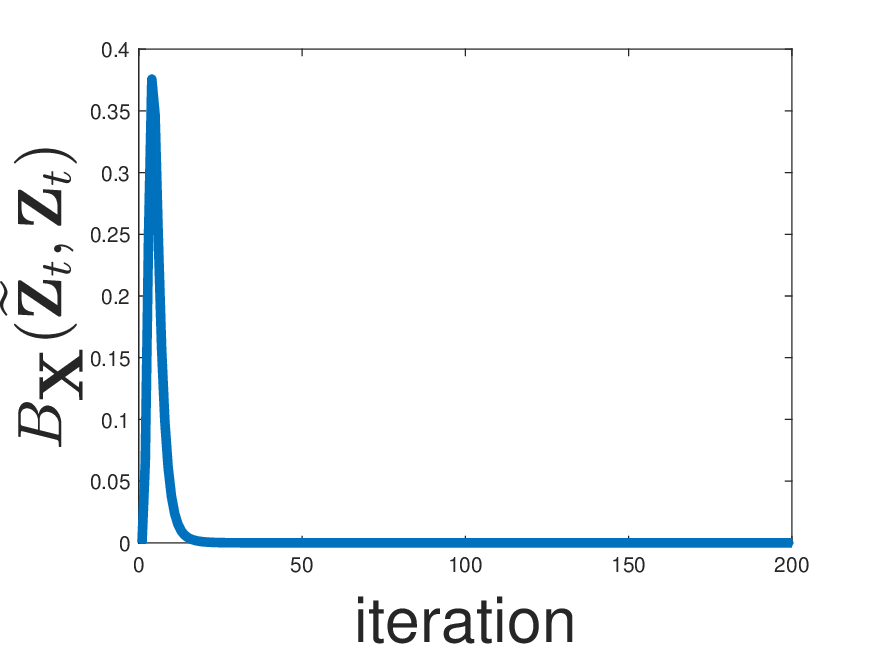} 
  \end{subfigure}
      \caption*{$\eta=0.1$, $\lambda=0.0001$, $\rank(\M)=1$, $\rank(\A_i)=3$, $c=10000$}
\medskip
  \begin{subfigure}[t]{0.33\textwidth}
    \includegraphics[width=\textwidth
    ]{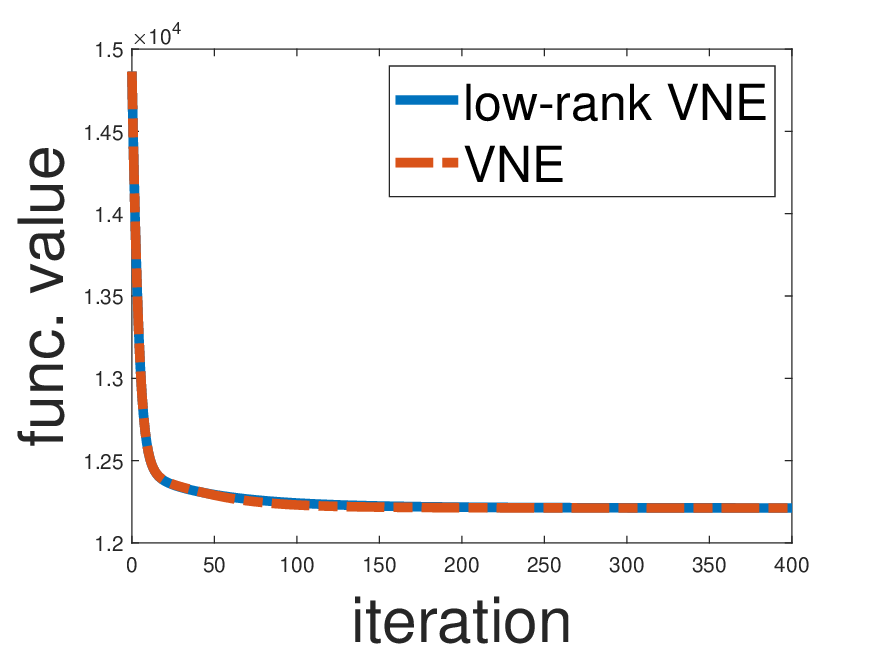}
  \end{subfigure}\hfil
  \begin{subfigure}[t]{0.33\textwidth}
    \includegraphics[width=\textwidth]
   {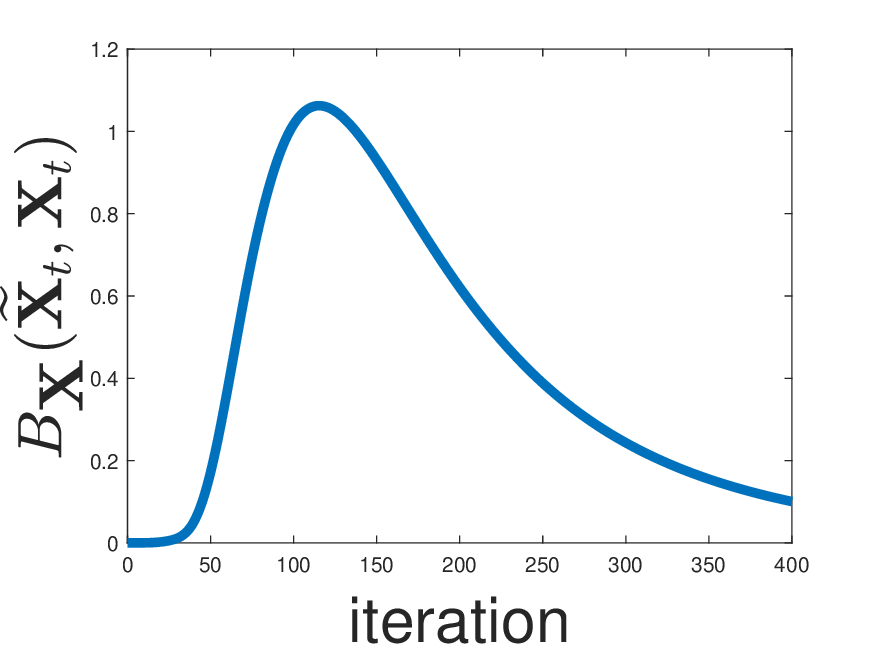} 
  \end{subfigure}\hfil
    \begin{subfigure}[t]{0.33\textwidth}
    \includegraphics[width=\textwidth]{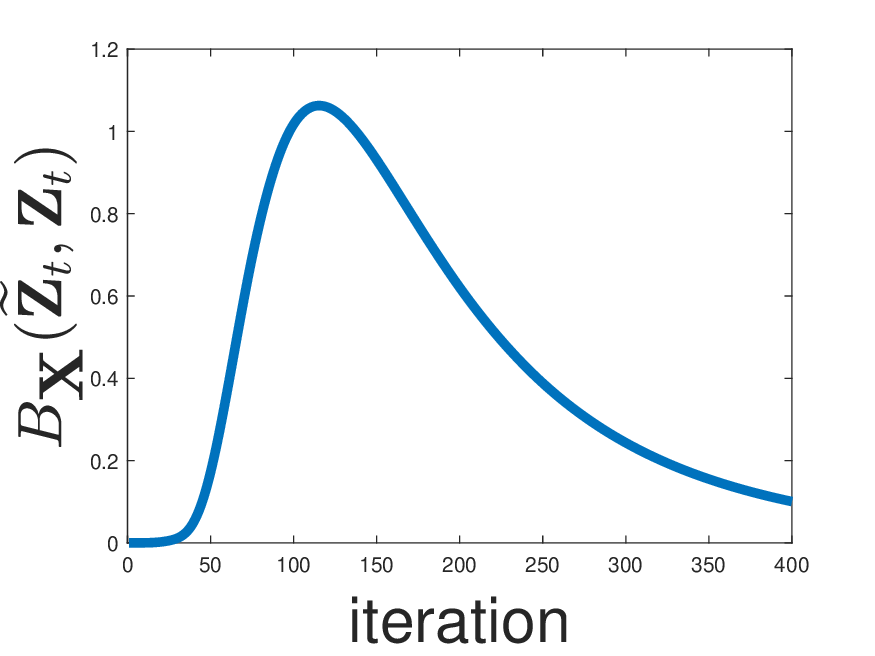} 
  \end{subfigure}
        \caption*{$\eta=0.006$, $\lambda=0.0001$, $\rank(\M)=3$, $\rank(\A_i)=1$, $c=10000$}
  \caption{Comparison between the standard von Neumann entropy mirror-prox (VNE) method and the low-rank von Neumann entropy mirror-prox (low-rank VNE).} \label{fig:MPvsLowrank}
\end{figure}

\subsection{Performance of low-rank subgradient descent on sparse PCA}

We now turn to demonstrate the performance of the subgradient descent method with low-rank projections (low-rank SD) on the sparse PCA problem and compare it with its exact counterpart which uses exact projections (SD), and our low-rank Euclidean extragradient method (low-rank EG). Recall that in Lemma \ref{lemma:negativeExampleNonSmooth} we have used the sparse PCA problem to establish that there exists instances which satisfy strict complementarity yet, at any proximity of the optimal solution, low-rank SD produces iterates of rank higher than that of the optimal solution. 
We consider two types of instances of the sparse PCA problem. The first  is simply the deterministic problem from Lemma \ref{lemma:negativeExampleNonSmooth} with parameters $n=400$ and $k=n/4$.
For the second type of instances, we used random instances as described in Section \ref{sec:sparsePCA}, with $n=400$, $\N\sim U[0,1]$, and $\textrm{SNR}=0.05$. 
For both types of instances we initialize all methods with a random rank-one matrix $\v_1\v_1^{\top}\in\Sn$, where $\v_1\in\reals^n$ is the unit-norm normalization of the vector $\u$ such that for all $i\in[n]$:
\begin{align} \label{eq:initialInExp}
\u_i=\Bigg\lbrace \begin{array}{ll} \z_i+\xi_i & i\in\support(\z) 
\\ 0 & i\not\in\support(\z)\end{array},
\end{align}
where $\z$ is the ground-truth sparse vector to be recovered and $\xi_i\sim\mathcal{N}(0,10^{-4})$ (in particular, with this initialization the condition $\langle\z,\v_1\rangle^2\ge1-1/(2k^2)$ stated in Lemma  \ref{lemma:negativeExampleNonSmooth} holds for all random draws of $\v_1$). For both types of instances we use rank-one projections for the low-rank SD and low-rank EG methods (we note that in all instances strict-complementarity indeed holds w.r.t. a rank-one optimal solution).
For all methods we tune the step-size for best performance. For the subgradient descent methods we choose the step-size $\eta=1/\sqrt{T}$ for the first type of instances, and $\eta=20/(n\sqrt{T})$ for the second type, where $T$ is the total number of iterations. For the extragradient method we use a step-size of $\eta=100\eta_0$ for both types of instances, where $\eta_0$ is the theoretical step-size (see Section \ref{sec:sparsePCA}) which we observe to be over pessimistic. We set the number of iterations to $T=400$ for the first type of instances and $T=1000$ for the second one. In both cases the results are the averages over $10$ i.i.d. runs.

In Figure \ref{fig:EGvsSD} we plot the values of the nonsmooth objective $g(\X)=\langle\X,-\M\rangle+\lambda\Vert\X\Vert_1$ for each method, where on each iteration we take the minimal value of $g$ achieved so far. 
In accordance with the theoretical convergence rates, we see that the low-rank extragradient method indeed converges considerably faster than the two subgradient descent methods on both types of instances. 
Also, it can be seen in Figure \ref{fig:12} that for the instance proposed in Lemma \ref{lemma:negativeExampleNonSmooth}, the low-rank projected subgradient method does not converge to the optimal value. 
Somewhat surprisingly, when considering random instances, we see in Figure \ref{fig:13} that the low-rank subgradient descent  converges faster than its counterpart with exact projections, despite the fact that upon investigation we observe that the low-rank projection is different than the exact one on almost every iteration. It may be an interesting future direction to theoretically study the convergence of low-rank subgradient descent in such circumstances. 



\begin{figure}
\centering
   \begin{subfigure}[t]{0.33\textwidth}
    \includegraphics[width=\textwidth]
    {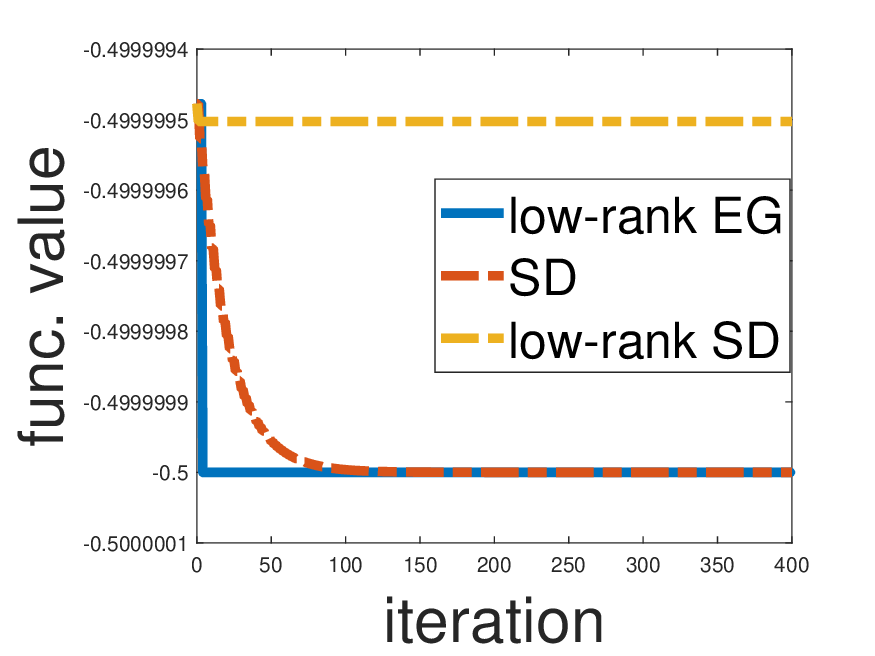}
    \caption{deterministic instance from Lemma \ref{lemma:negativeExampleNonSmooth}}
    \label{fig:12}
  \end{subfigure}
   \begin{subfigure}[t]{0.33\textwidth}
    \includegraphics[width=\textwidth]
    {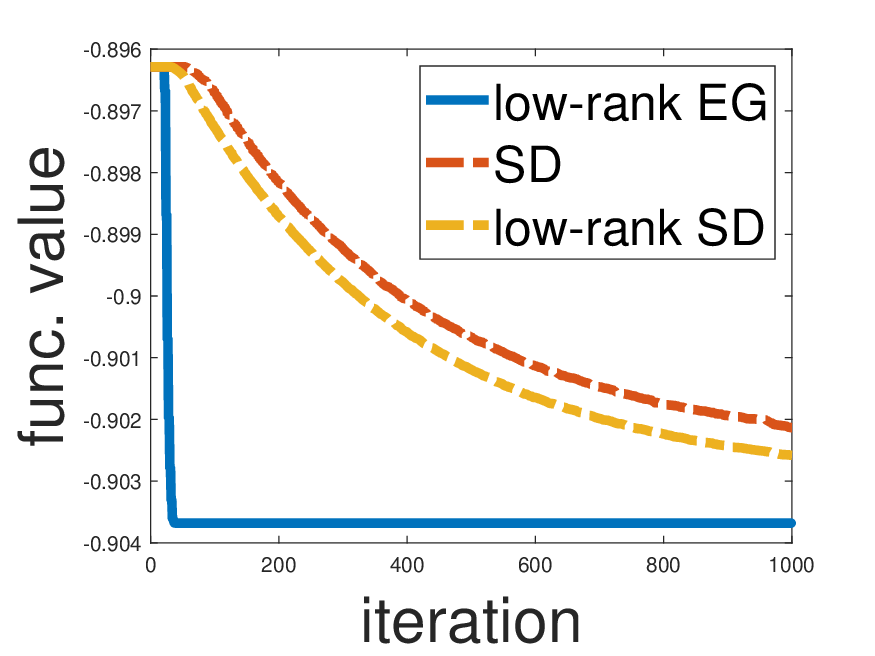}
    \caption{random instances}
    \label{fig:13}
  \end{subfigure}
     \caption{Comparison between the projected subgradient decent (SD), low-rank projected subgradient decent (low-rank SD), and low-rank projected extragradient (low-rank EG).} \label{fig:EGvsSD}
\end{figure}

\backmatter

\begin{appendices}

\section{Proofs Omitted from Section \ref{sec:strictComp}}
\label{sec:proofKKT}

\subsection{Proof of Lemma \ref{lemma:kkt}}
\label{sec:appendix:proofLemma2}
We first restate the lemma and then prove it.

\begin{lemma} 
Let $\X^*\in\Sn$ be a rank-$r^*$ optimal solution to Problem \eqref{nonSmoothProblem}. $\X^*$ satisfies the (standard) strict complementarity assumption with parameter $\delta>0$ if and only if there exists a subgradient $\G^*\in\partial g(\X^*)$ such that $\langle \X-\X^*,\G^*\rangle\ge0$ for all $\X\in\mathcal{S}_n$ and $\lambda_{n-r^*}(\G^*)-\lambda_{n}(\G^*)\ge\delta$.
\end{lemma}

\begin{proof}

By Slater's condition strong duality holds for Problem \eqref{nonSmoothProblem}. Therefore, the KKT conditions for Problem \eqref{nonSmoothProblem} hold for the optimal solution $\X^*$ and some optimal dual solution $(\Z^*,s^*)$. 
The Lagrangian of Problem \eqref{nonSmoothProblem} can be written as
\begin{align*}
\mathcal{L}(\X,\Z,s)=g(\X)+s(1-\trace(\X))-\langle\Z,\X\rangle.
\end{align*}
Thus, using the generalized KKT conditions for nonsmooth optimization problems (see Theorem 6.1.1 in \cite{generalizedKKT}), this implies that for the primal and dual optimal solutions
\begin{align*}
\mathbf{0}\in \partial g(\X^*)-\Z^*-s^*\I,
\\ \langle \X^*,\Z^*\rangle=0,
\\ \trace(\X^*)=1,
\\ \X^*,\Z^*\succeq0.
\end{align*}
The generalized first order optimality condition for unconstrained minimization implies that there exists some $\G^*\in \partial g(\X^*)$ for which $\mathbf{0}= \G^*-\Z^*-s^*\I$. It remains to be shown that $\langle \X-\X^*,\G^*\rangle\ge0$ for all $\X\in\mathcal{S}_n$.

The cone of positive semidefinite matrices is self-dual, that is $\Z^*\succeq0$ if and only if $\langle \X,\Z^*\rangle\ge0$ for all $\X\in\mathcal{S}_n$.
Therefore, $\Z^*\succeq0$ if and only if for all $\X\in\mathcal{S}_n$ it holds that
\begin{align*}
0\le \langle \X,\Z^*\rangle & = \langle \X,\Z^*\rangle-\langle \X^*,\Z^*\rangle+s^*\langle \X-\X^*,\I\rangle
\\ & = \langle \X-\X^*,\Z^*+s^*\I\rangle
 = \langle \X-\X^*,\G^*\rangle
\end{align*}
as desired. The first equality holds using the complementarity condition and the property that $\trace(\X)=\trace(\X^*)=1$.

Using the equality $\G^*=\Z^*+s^*\I$ it holds that 
\begin{align*}
\lambda_{n-r^*}(\Z^*) & = \lambda_{n-r^*}(\Z^*)+s^*\I-\lambda_{n}(\Z^*)-s^*\I = \lambda_{n-r^*}(\Z^*+s^*\I)-\lambda_{n}(\Z^*+s^*\I)
\\ & =\lambda_{n-r^*}(\G^*)-\lambda_{n}(\G^*).
\end{align*}
Thus, $\X^*$ satisfies the strict complementarity assumption with parameter $\delta>0$, i.e., $\lambda_{n-r^*}(\Z^*)\ge\delta$, if and only if $\lambda_{n-r^*}(\G^*)-\lambda_{n}(\G^*)\ge\delta$.

\end{proof}

\subsection{Proof of Lemma \ref{lemma:negativeExampleNonSmooth}}
\label{sec:appendix:proofLemma4}
We first restate the lemma and then prove it.

\begin{lemma}
Consider the problem
\begin{align*}
\min_{\X\in\Sn}\{g(\X):=-\left\langle\z\z^{\top}+\z_{\perp}\z_{\perp}^{\top},\X\right\rangle+\frac{1}{2k}\Vert\X\Vert_1\},
\end{align*}
where $\z=(1/\sqrt{k},\ldots,1/\sqrt{k},0,\ldots,0)^{\top}$ is supported on the first $k$ entries,  $\z_{\perp}=(0,\ldots,0,\\ 1/\sqrt{n-k},\ldots,1/\sqrt{n-k})^{\top}$ is supported on the last $n-k$ entries, and $k\le n/4$. Then, $\z\z^{\top}$ is a rank-one optimal solution for which  strict complementarity holds.
However, for any $\eta<\frac{2}{3}$ and any $\v\in\reals^n$ such that $\Vert\v\Vert=1$, $\support(\v)\subseteq\support(\z)$, and $\langle\z,\v\rangle^2=1 - \frac{1}{2}\Vert{\v\v^{\top}-\z\z^{\top}}\Vert_F^2 \ge1-\frac{1}{2k^2}$, it holds that
\begin{align*}
\rank\left(\Pi_{\Sn}[\v\v^{\top}-\eta\G_{\v\v^{\top}}]\right)>1,
\end{align*}
where $\G_{\v\v^{\top}}=-\z\z^{\top}-\z_{\perp}\z_{\perp}^{\top}+\frac{1}{2k}\sign(\v\v^{\top})\in\partial g(\v\v^{\top})$.
\end{lemma}

\begin{proof}
$\z\z^{\top}$ is a rank-one optimal solution for this problem since for the subgradient $k\z\z^{\top}+2k\z_{\perp}\z_{\perp}^{\top}\in\partial \left(\Vert\z\z^{\top}\Vert_1\right)$ the first-order optimality condition holds. Indeed, for all $\X\in\Sn$ 
\begin{align} \label{eq:ExampleFirstOrderOptCondHolds}
\langle\X-\z\z^{\top},-\z\z^{\top}-\z_{\perp}\z_{\perp}^{\top}+\frac{1}{2}\z\z^{\top}+\z_{\perp}\z_{\perp}^{\top}\rangle = \langle\X-\z\z^{\top},-\frac{1}{2}\z\z^{\top}\rangle\ge0.
\end{align}

For the subgradient $-\frac{1}{2}\z\z^{\top}\in\partial g(\z\z^{\top})$ there is a gap $\lambda_{n-1}(-\frac{1}{2}\z\z^{\top})-\lambda_{n}(-\frac{1}{2}\z\z^{\top})=\frac{1}{2}>0$, and as we showed in \eqref{eq:ExampleFirstOrderOptCondHolds} the first order optimality condition holds for $-\frac{1}{2}\z\z^{\top}$. Thus, by Lemma \ref{lemma:kkt} the optimal solution $\z\z^{\top}$ satisfies standard strict complementarity. 

We will show that the projection onto the spectrahedron of a subgradient step from $\v\v^{\top}$ with respect to the natural subgradient of the $\ell_1$-norm $\sign(\v\v^{\top})\in\partial \left(\Vert\v\v^{\top}\Vert_1\right)$ returns a rank-2 solution.

It holds that
\begin{align*}
1-\frac{1}{2k^2}& \le\langle\z\z^{\top},\v\v^{\top}\rangle=\frac{1}{2}\left(\Vert\z\z^{\top}\Vert_F^2+\Vert\v\v^{\top}\Vert_F^2-\Vert\v\v^{\top}-\z\z^{\top}\Vert_F^2\right)
\\ &  = 1- \frac{1}{2}\Vert\v\v^{\top}-\z\z^{\top}\Vert_F^2,
\end{align*}
and equivalently
\begin{align*}
\sum_{i=1}^k\sum_{j=1}^k\left(\frac{1}{k}-(\v\v^{\top})_{ij}\right)^2=\Vert\v\v^{\top}-\z\z^{\top}\Vert_F^2\le\frac{1}{k^2}.
\end{align*}
Therefore, for every $i,j\in\lbrace 1,\ldots,k\rbrace$ it holds that
\begin{align*}
\left\vert(\v\v^{\top})_{ij}-\frac{1}{k}\right\vert\le\frac{1}{k},
\end{align*}
which implies that $0\le(\v\v^{\top})_{ij}\le\frac{2}{k}$. Therefore, $k\z\z^{\top}=\sign(\v\v^{\top})\in\partial \left(\Vert\v\v^{\top}\Vert_1\right)$.

Taking a projected subgradient step from $\v\v^{\top}$ with respect to the subgradient $-\z\z^{\top}-\z_{\perp}\z_{\perp}^{\top}+\frac{1}{2}\z\z^{\top}\in\partial g(\v\v^{\top})$ has the form
\begin{align*}
&\Pi_{\Sn}\left[\v\v^{\top}-\eta\left(-\z\z^{\top}-\z_{\perp}\z_{\perp}^{\top}+\frac{1}{2}\z\z^{\top}\right)\right]
 = \Pi_{\Sn}\left[\v\v^{\top}+\frac{\eta}{2}\z\z^{\top}+\eta\z_{\perp}\z_{\perp}^{\top}\right].
\end{align*}

Since $\support(\v)\subseteq\support(\z)$ it holds that $\left(\v\v^{\top}+\frac{\eta}{2}\z\z^{\top}\right)\perp\z_{\perp}\z_{\perp}^{\top}$. In addition,
$\v\v^{\top}+\frac{\eta}{2}\z\z^{\top}$ is a rank-2 matrix and so we can denote the eigen-decomposition of $\v\v^{\top}+\frac{\eta}{2}\z\z^{\top}+\eta\z_{\perp}\z_{\perp}^{\top}$ as $\v\v^{\top}+\frac{\eta}{2}\z\z^{\top}+\eta\z_{\perp}\z_{\perp}^{\top}=\lambda_1\v_1\v_1^{\top}+\lambda_2\v_2\v_2^{\top}+\eta\z_{\perp}\z_{\perp}^{\top}$, where $\lambda_1\ge\lambda_2$. 
Thus, invoking \eqref{eq:euclidProj} to calculate the projection we need to find the scalar $\lambda\in\reals$ for which the following holds.
\begin{align*}
\max\left\lbrace\lambda_1-\lambda,0\right\rbrace+\max\left\lbrace\lambda_2-\lambda,0\right\rbrace+\max\lbrace\eta-\lambda,0\rbrace
+\sum_{i=4}^n\max \lbrace0-\lambda,0\rbrace = 1.
\end{align*}

$\lambda_1$ is the largest eigenvalue of $\v\v^{\top}+\frac{\eta}{2}\z\z^{\top}+\eta\z_{\perp}\z_{\perp}^{\top}$ since, under our assumption that $\eta < 2/3$, we have that
\[\lambda_1\ge \frac{1}{2}(\lambda_1+\lambda_2)=\frac{1}{2}\trace(\v\v^{\top}+\frac{\eta}{2}\z\z^{\top})=\frac{1}{2}+\frac{\eta}{4}>\eta.\]
Therefore, $\lambda<\lambda_1\le \lambda_1+\lambda_2=\trace(\v\v^{\top}+\frac{\eta}{2}\z\z^{\top})=1+\frac{\eta}{2}$.

In addition, $\max\lbrace\lambda_2,\eta\rbrace\ge\eta$. Therefore, 
\[\lambda_1-\max\lbrace\lambda_2,\eta\rbrace \le \lambda_1 + \lambda_2 - \eta = 1+\frac{\eta}{2}-\eta <1,\]
and so we must have that $\lambda<\max\lbrace\lambda_2,\eta\rbrace\le\lambda_1$. 

This implies that both $\max\left\lbrace\lambda_1-\lambda,0\right\rbrace>0$
and $\max\lbrace\max\lbrace\lambda_2,\eta\rbrace-\lambda,0\rbrace>0$. Thus, using  \eqref{eq:euclidProj} we conclude that 
\[\rank\left(\Pi_{\Sn}\left[\v\v^{\top}-\eta\left(-\z\z^{\top}-\z_{\perp}\z_{\perp}^{\top}+\frac{1}{2}\z\z^{\top}\right)\right]\right)\ge2.\]

\end{proof}

\section{Proofs Omitted from Section \ref{sec:smooth2Saddle}}

\subsection{Proof of Lemma \ref{lemma:connectionSubgradientNonSmoothAndSaddlePoint}}\label{sec:proofSubgradEquiv}
We first restate the lemma and then prove it.
\begin{lemma} 
If $(\X^*,\y^*)$ is a saddle-point of Problem \eqref{problem1} then $\X^*$ is an optimal solution to Problem \eqref{nonSmoothProblem},  $\nabla_{\X}f(\X^*,\y^*)\in\partial g(\X^*)$,  and for all $\X\in\Sn$ it holds that $\langle\X-\X^*,\nabla_{\X}f(\X^*,\y^*)\rangle\ge0$.
Conversely, under Assumption \ref{ass:struct}, if $\X^*$ is an optimal solution to Problem \eqref{nonSmoothProblem}, and $\G^*\in\partial{}g(\X^*)$ which satisfies $\langle\X-\X^*,\G^*\rangle\ge0$  for all $\X\in\Sn$, then there exists $\y^*\in\argmax_{\y\in\mathcal{K}}f(\X^*,\y)$ such that $(\X^*,\y^*)$ is a saddle-point of Problem \eqref{problem1}, and $\nabla_{\X}f(\X^*,\y^*)=\G^*$.
\end{lemma}

\begin{proof}

For the first direction of the lemma, we first observe that for any $\X_1,\X_2\in\Sn$ and $\widetilde{\y}_1\in \argmax_{\y\in\mathcal{K}}f(\X_1,\y)$, $\widetilde{\y}_2\in \argmax_{\y\in\mathcal{K}}f(\X_2,\y)$, using the gradient inequality for $f(\cdot,\widetilde{\y}_2)$, it holds that
\begin{align*}
g(\X_1) & = f(\X_1,\widetilde{\y}_1) \ge f(\X_1,\widetilde{\y}_2) \ge f(\X_2,\widetilde{\y}_2) + \langle\nabla_{\X}f(\X_2,\widetilde{\y}_2), \X_1-\X_2\rangle 
\\ & = g(\X_2)+\langle\nabla_{\X}f(\X_2,\widetilde{\y}_2), \X_1-\X_2\rangle.
\end{align*}
Thus,  $\nabla_{\X}f(\X_2,\widetilde{\y}_2)$ is a subgradient of $g(\cdot)$ at $\X_2$, i.e., $\nabla_{\X}f(\X_2,\widetilde{\y}_2)\in\partial g(\X_2)$.

In particular, for a saddle-point $(\X^*,\y^*)\in{\Sn\times\mathcal{K}}$ it holds that $\y^*\in \argmax_{\y\in\mathcal{K}}f(\X^*,\y)$, and therefore, it follows that $\nabla_{\X}f(\X^*,\y^*)\in\partial g(\X^*)$. In addition, for all $\X\in\Sn$ and $\widetilde{\y}\in \argmax_{\y\in\mathcal{K}}f(\X,\y)$ we have
\[ g(\X^*)=f(\X^*,\y^*)\le f(\X,\y^*)\le f(\X,\widetilde{\y})=g(\X),\] which implies that $\X^*$ is an optimal solution to $\min_{\X\in\Sn}g(\X)$. 

Finally, we need to show that the subgradient $\nabla_{\X}f(\X^*,\y^*)\in\partial g(\X^*)$ indeed satisfies the first-order optimality condition for $g(\cdot)$ at $\X^*$. To see this, we observe that since $\X^*$ is an optimal solution to $\min_{\X\in\Sn}f(\X,\y^*)$, it follows from the first-order optimality condition for the problem $\min_{\X\in\mS_n}f(\X,\y^*)$, that for all $\W\in\Sn$
\begin{align*}
\langle \W-\X^*,\nabla_{\X}f(\X^*,\y^*)\rangle\ge0,
\end{align*}
as needed.

For the second direction, let $\X^*\in\argmin_{\X\in\Sn} g(\X)$ and let $\G^*\in\partial g(\X^*)$ such that $\langle \X-\X^*,\G^*\rangle\ge0$ for all $\X\in\mS_n$. By Assumption \ref{ass:struct} and using Danskin's theorem (see for instance \cite{DanskinTheorem}), the subdifferential set of $g(\X^*) = h(\X^*) + \max_{\y\in\mK}\y^{\top}(\mA(\X^*)-\b)$ can be written as
\begin{align*}
\partial g(\X^*) & = \nabla{}h(\X^*)+\textrm{conv}\left\lbrace\mA^{\top}(\y)\ \Big\vert\ \y\in\argmax_{\y\in\mathcal{K}}\y^{\top}(\mA(\X^*)-\b)\right\rbrace
\\ & = \nabla{}h(\X^*)+\mA^{\top}\left(\textrm{conv}\left\lbrace\y\ \Big\vert\ \y\in\argmax_{\y\in\mathcal{K}}\y^{\top}(\mA(\X^*)-\b)\right\rbrace\right)
\\ & = \nabla{}h(\X^*)+\mA^{\top}\left(\left\lbrace\y\ \Big\vert\ \y\in\argmax_{\y\in\mathcal{K}}\y^{\top}(\mA(\X^*)-\b)\right\rbrace\right)
\\ & = \nabla{}h(\X^*)+\mA^{\top}\left(\left\lbrace\y\ \Big\vert\ \y\in\argmax_{\y\in\mathcal{K}}f(\X^*,\y)\right\rbrace\right),
\end{align*}
where $\textrm{conv}\lbrace\cdot\rbrace$ denotes the convex hull operation and the third equality follows from the convexity of $\mathcal{K}$.

Thus, there exists some $\y^*\in\argmax_{\y\in\mathcal{K}}f(\X^*,\y)$ such that $\G^*=\nabla{}h(\X^*)+\mA^{\top}(\y^*)=\nabla_{\X}f(\X^*,\y^*)$.

Since $\y^*\in \argmax_{\y\in\mathcal{K}}f(\X^*,\y)$, it follows that for all $\y\in\mathcal{K}$, $f(\X^*,\y^*)\ge f(\X^*,\y)$. In addition, using the fact that $\G^*$ satisfies the first-order optimality condition, and using gradient inequality w.r.t. $f(\cdot,\y^*)$, we have that for all $\X\in\Sn$,
\begin{align*}
0\le\langle \X-\X^*,\G^*\rangle=\langle \X-\X^*,\nabla_{\X}f(\X^*,\y^*)\rangle\le f(\X,\y^*)-f(\X^*,\y^*).
\end{align*}
Thus, it follows that $f(\X,\y^*)\ge f(\X^*,\y^*)$. Therefore, $(\X^*,\y^*)$ is indeed a saddle-point of $f$.

\end{proof}

\section{Proofs Omitted from Section \ref{sec:approximatedMP}}
\label{sec:appendix:OmittedApproximatedMP}

\subsection{Proof of Lemma \ref{lemma:convergenceApproxMethod}}
\label{sec:appendix:proofLemma6}
We first restate the lemma and then prove it.

\begin{lemma}
Let $\lbrace(\X_t,\y_t)\rbrace_{t\ge1}$, $\lbrace(\Z_{t},\w_t)\rbrace_{t\ge2}$, $\lbrace\widehat{\X}_t\rbrace_{t\ge2}$, and $\lbrace\widehat{\Z}_{t}\rbrace_{t\ge2}$ be the sequences generated by Algorithm \ref{alg:ApproxMP} with a fixed step-size $\eta_t=\eta\le\min\left\lbrace\frac{1}{\beta_{X}+\beta_{Xy}},\frac{1}{(1+\theta)(\beta_{X}+\beta_{yX})},\frac{1}{\beta_{y}+\beta_{yX}},\frac{1}{\beta_{y}+\beta_{Xy}}\right\rbrace$, such that $\theta=\bigg\lbrace\begin{array}{ll} 0 & if\ \Z_{t+1}=\widehat{\Z}_{t+1}
\\ 1 & otherwise \end{array}$. Then,
\begin{align} \label{ineq:convergenceRateAppend}
& \max_{\y\in\mathcal{K}} f\left(\frac{1}{T}\sum_{t=1}^T \Z_{t+1},\y\right) -  \min_{\X\in\Sn} f\left(\X,\frac{1}{T}\sum_{t=1}^T\w_{t+1}\right) \nonumber
\\ & \le \frac{D^2}{\eta T}
+ \frac{1}{\eta T}\sum_{t=1}^T\max_{\X\in\Sn}\left(\breg_{\X}(\X,\X_{t+1})-\breg_{\X}(\X,\widehat{\X}_{t+1})\right)
 \nonumber
\\ & \ \ \ + \frac{2(\beta_{X}+\beta_{yX})}{T}\sum_{t=1}^T\breg_{\X}(\widehat{\Z}_{t+1},\Z_{t+1})+\frac{\sqrt{2}G}{T}\sum_{t=1}^T\sqrt{\breg_{\X}(\widehat{\Z}_{t+1},\Z_{t+1})},
\end{align}
where $D^2:=\sup_{(\X,\y)\in{\Sn\times\mathcal{K}}}\left(\breg_{\X}(\X,\X_1)+\breg_{\y}(\y,\y_1)\right)$ and  $G=\sup_{(\X,\y)\in{\Sn\times\mathcal{K}}}\Vert\nabla_{\X}f(\X,\y)\Vert_{\mathcal{X^*}}$.
\end{lemma}

\begin{proof}

By the gradient inequality, for any $\X\in\Sn$ it holds that
\begin{align} \label{ineq:convergenceX}
& f(\Z_{t+1},\w_{t+1})-f(\X,\w_{t+1}) \nonumber
\\ & \le \langle \Z_{t+1}-\X,\nabla_{\X}f(\Z_{t+1},\w_{t+1})\rangle \nonumber
\\ & = \langle \widehat{\Z}_{t+1}-\X,\nabla_{\X}f(\Z_{t+1},\w_{t+1})\rangle + \langle\Z_{t+1}-\widehat{\Z}_{t+1},\nabla_{\X}f(\Z_{t+1},\w_{t+1})\rangle \nonumber
\\ & = \langle\widehat{\X}_{t+1}-\X,\nabla_{\X}f(\Z_{t+1},\w_{t+1})\rangle + \langle\widehat{\Z}_{t+1}-\widehat{\X}_{t+1},\nabla_{\X}f(\X_{t},\y_{t})\rangle \nonumber
\\ & \ \ \ + \langle\widehat{\Z}_{t+1}-\widehat{\X}_{t+1},\nabla_{\X}f(\Z_{t+1},\w_{t+1})-\nabla_{\X}f(\X_{t},\y_{t}) \rangle \nonumber
\\ & \ \ \ + \langle\Z_{t+1}-\widehat{\Z}_{t+1},\nabla_{\X}f(\Z_{t+1},\w_{t+1})\rangle .
\end{align}
We will bound the four terms in the RHS of \eqref{ineq:convergenceX} separately.

For the first term, by the definition of $\widehat{\X}_{t+1}$ in Algorithm \ref{alg:ApproxMP},
\begin{align*}
\widehat{\X}_{t+1} & = \argmin_{\X\in\Sn}\lbrace\langle\eta_t\nabla_{\X}f(\Z_{t+1},\w_{t+1}),\X\rangle+\breg_{\X}(\X,\X_t)\rbrace
\\ & = \argmin_{\X\in\mathcal{S}_n}\{\langle\eta_t\nabla_{\X} f(\Z_{t+1},\w_{t+1})-\nabla\omega_{\X}(\X_t),\X\rangle+\omega_{\X}(\X)\}.
\end{align*}

Therefore, by the optimality condition for $\widehat{\X}_{t+1}$, $\forall \X\in\mathcal{S}_n$, 
\[ \langle \eta_t\nabla_{\X} f(\Z_{t+1},\w_{t+1})-\nabla\omega_{\X}(\X_t)+\nabla\omega_{\X}(\widehat{\X}_{t+1}),\X-\widehat{\X}_{t+1}\rangle\ge0. \]

Rearranging and using the three point lemma in \eqref{lemma:threePointLemma}, we obtain that $\forall \X\in\mathcal{S}_n$, 
\begin{align} \label{ineq:convergenceTerm1}
\eta_t\langle\widehat{\X}_{t+1}-\X,\nabla_{\X} f(\Z_{t+1},\w_{t+1})\rangle & \le \langle\nabla\omega_{\X}(\X_t)-\nabla\omega_{\X}(\widehat{\X}_{t+1})),\widehat{\X}_{t+1}-\X\rangle \nonumber 
\\ & = \breg_{\X}(\X,\X_t)-\breg_{\X}(\X,\widehat{\X}_{t+1})-\breg_{\X}(\widehat{\X}_{t+1},\X_t).
\end{align}

For the second term in the RHS of \eqref{ineq:convergenceX}, by the definition of $\widehat{\Z}_{t+1}$ in Algorithm \ref{alg:ApproxMP},
\begin{align*}
\widehat{\Z}_{t+1}=\argmin_{\X\in\mathcal{S}_n}\{\langle\eta_t\nabla_{\X} f(\X_t,\y_t)-\nabla\omega_{\X}(\X_t),\X\rangle+\omega_{\X}(\X)\}.
\end{align*}

Therefore, by the optimality condition for $\widehat{\Z}_{t+1}$, $\forall \X\in\mathcal{S}_n$, 
\[ \langle \eta_t\nabla_{\X} f(\X_t,\y_t)-\nabla\omega_{\X}(\X_t)+\nabla\omega_{\X}(\widehat{\Z}_{t+1}),\X-\widehat{\Z}_{t+1}\rangle\ge0. \]

Rearranging and using the three point lemma in \eqref{lemma:threePointLemma}, we obtain by plugging in $\X=\widehat{\X}_{t+1}$, 
\begin{align} \label{ineq:convergenceTerm2}
\eta_t\langle\widehat{\Z}_{t+1}-\widehat{\X}_{t+1},\nabla_{\X} f(\X_t,\y_t)\rangle & \le \langle\nabla\omega_{\X}(\X_t)-\nabla\omega_{\X}(\widehat{\Z}_{t+1})),\widehat{\Z}_{t+1}-\widehat{\X}_{t+1}\rangle \nonumber 
\\ & = \breg_{\X}(\widehat{\X}_{t+1},\X_t)-\breg_{\X}(\widehat{\X}_{t+1},\widehat{\Z}_{t+1})-\breg_{\X}(\widehat{\Z}_{t+1},\X_t).
\end{align}

For the third term in the RHS of \eqref{ineq:convergenceX}, we obtain
\begin{align} \label{ineq:convergenceTerm3}
& \langle\widehat{\Z}_{t+1}-\widehat{\X}_{t+1},\nabla_{\X}f(\Z_{t+1},\w_{t+1})-\nabla_{\X}f(\X_{t},\y_{t}) \rangle \nonumber
\\ & \underset{(a)}{\le} \Vert\widehat{\Z}_{t+1}-\widehat{\X}_{t+1}\Vert_{\mathcal{X}}\cdot\Vert\nabla_{\X}f(\Z_{t+1},\w_{t+1})-\nabla_{\X}f(\X_{t},\y_{t})\Vert_2 \nonumber
\\ & \le \Vert\widehat{\Z}_{t+1}-\widehat{\X}_{t+1}\Vert_{\mathcal{X}}\cdot\Vert\nabla_{\X}f(\Z_{t+1},\w_{t+1})-\nabla_{\X}f(\X_{t},\w_{t+1})\Vert_2 \nonumber
\\ &  \ \ \ +\Vert\widehat{\Z}_{t+1}-\widehat{\X}_{t+1}\Vert_{\mathcal{X}}\cdot\Vert\nabla_{\X}f(\X_{t},\w_{t+1})-\nabla_{\X}f(\X_{t},\y_{t})\Vert_2 \nonumber
\\ & \underset{(b)}{\le} \left(\beta_{X}\Vert\Z_{t+1}-\X_{t}\Vert_{\mathcal{X}}+\beta_{Xy}\Vert\w_{t+1}-\y_{t}\Vert_{\mathcal{Y}}\right)\cdot\Vert\widehat{\Z}_{t+1}-\widehat{\X}_{t+1}\Vert_{\mathcal{X}} \nonumber
\\ & \le \left(\beta_{X}\Vert\widehat{\Z}_{t+1}-\X_{t}\Vert_{\mathcal{X}}+\beta_{X}\Vert\Z_{t+1}-\widehat{\Z}_{t+1}\Vert_{\mathcal{X}}+\beta_{Xy}\Vert\w_{t+1}-\y_{t}\Vert_{\mathcal{Y}}\right)\cdot\Vert\widehat{\Z}_{t+1}-\widehat{\X}_{t+1}\Vert_{\mathcal{X}} \nonumber
\\ & \underset{(c)}{\le} \frac{\beta_{X}}{2}\left((1+\theta)\Vert\widehat{\Z}_{t+1}-\X_{t}\Vert_{\mathcal{X}}^2+2\Vert\Z_{t+1}-\widehat{\Z}_{t+1}\Vert_{\mathcal{X}}^2+\Vert\widehat{\Z}_{t+1}-\widehat{\X}_{t+1}\Vert_{\mathcal{X}}^2\right) \nonumber
\\ & \ \ \ +\frac{\beta_{Xy}}{2}\left(\Vert\w_{t+1}-\y_{t}\Vert_{\mathcal{Y}}^2+\Vert\widehat{\Z}_{t+1}-\widehat{\X}_{t+1}\Vert_{\mathcal{X}}^2\right) \nonumber
\\ & \underset{(d)}{\le} \beta_{X}\left((1+\theta)\breg_{\X}(\widehat{\Z}_{t+1},\X_t)+2\breg_{\X}(\widehat{\Z}_{t+1},\Z_{t+1})+\breg_{\X}(\widehat{\X}_{t+1},\widehat{\Z}_{t+1})\right) \nonumber
\\ & \ \ \ +\beta_{Xy}\left(\breg_{\y}(\w_{t+1},\y_t)+\breg_{\X}(\widehat{\X}_{t+1},\widehat{\Z}_{t+1})\right),
\end{align}
where (a) follows from H\"{o}lder's inequality, (b) follows from the $\beta_{X}$ and $\beta_{Xy}$ smoothness, (c) follows since for all $a,b\in\reals$ the inequality $2ab\le a^2+b^2$ holds, and (d) follows from \eqref{ineq:strongConvexityOfBregmanDistance}.

For the fourth term in the RHS of \eqref{ineq:convergenceX}, using H\"{o}lder's inequality and \eqref{ineq:strongConvexityOfBregmanDistance}, we obtain
\begin{align} \label{ineq:convergenceTerm4}  
\langle\Z_{t+1}-\widehat{\Z}_{t+1},\nabla_{\X}f(\Z_{t+1},\w_{t+1})\rangle & 
\le \Vert\Z_{t+1}-\widehat{\Z}_{t+1}\Vert_{\mathcal{X}}\cdot\Vert\nabla_{\X}f(\Z_{t+1},\w_{t+1})\Vert_{\mathcal{X^*}} \nonumber
\\ & \le \sqrt{2}\Vert\nabla_{\X}f(\Z_{t+1},\w_{t+1})\Vert_{\mathcal{X^*}}\sqrt{\breg_{\X}(\widehat{\Z}_{t+1},\Z_{t+1})}.
\end{align}

Plugging \eqref{ineq:convergenceTerm1}, \eqref{ineq:convergenceTerm2}, \eqref{ineq:convergenceTerm3}, and \eqref{ineq:convergenceTerm4} into \eqref{ineq:convergenceX} we obtain 
\begin{align} \label{ineq:boundXseiries}
& f(\Z_{t+1},\w_{t+1})-f(\X,\w_{t+1}) \nonumber
\\ & \le \frac{1}{\eta_t}\left(\breg_{\X}(\X,\X_t)-\breg_{\X}(\X,\widehat{\X}_{t+1})\right) +\left((1+\theta)\beta_{X}-\frac{1}{\eta_t}\right)\breg_{\X}(\widehat{\Z}_{t+1},\X_t) \nonumber
\\ & \ \ \ +\left((\beta_{X}+\beta_{Xy})-\frac{1}{\eta_t}\right)\breg(\widehat{\X}_{t+1},\widehat{\Z}_{t+1})+\beta_{Xy}\breg_{\y}(\w_{t+1},\y_t) \nonumber
\\ & \ \ \ +2\beta_{X}\breg_{\X}(\widehat{\Z}_{t+1},\Z_{t+1})+\sqrt{2}\Vert\nabla_{\X}f(\Z_{t+1},\w_{t+1})\Vert_{\mathcal{X^*}}\sqrt{\breg_{\X}(\widehat{\Z}_{t+1},\Z_{t+1})}.
\end{align}

Similarly, by the gradient inequality, for any $\y\in\mathcal{K}$ it holds that
\begin{align} \label{ineq:convergenceY}
& f(\Z_{t+1},\y) - f(\Z_{t+1},\w_{t+1}) \nonumber
\\ & \le \langle \w_{t+1}-\y,-\nabla_{\y}f(\Z_{t+1},\w_{t+1})\rangle \nonumber
\\ & = \langle \y_{t+1}-\y,-\nabla_{\y}f(\Z_{t+1},\w_{t+1})\rangle + \langle \w_{t+1}-\y_{t+1},-\nabla_{\y}f(\X_{t},\y_{t})\rangle \nonumber
\\ & \ \ \ + \langle \w_{t+1}-\y_{t+1},\nabla_{\y}f(\X_{t},\y_{t})-\nabla_{\y}f(\Z_{t+1},\w_{t+1})\rangle.
\end{align}
We will bound the three terms in the RHS of \eqref{ineq:convergenceY} separately.

For the first term, by the definition of $\y_{t+1}$ in Algorithm \ref{alg:ApproxMP}, we have that
\begin{align*}
\y_{t+1}=\argmin_{\y\in\mathcal{K}}\{\langle-\eta_t\nabla_{\y} f(\Z_{t+1},\w_{t+1})-\nabla\omega_{\y}(\y_t),\y\rangle+\omega_{\y}(\y)\}.
\end{align*}

Therefore, by the optimality condition for $\y_{t+1}$, $\forall \y\in\mathcal{K}$,
\begin{align*}
\langle -\eta_t\nabla_{\y} f(\Z_{t+1},\w_{t+1})-\nabla\omega_{\y}(\y_t)+\nabla\omega_{\y}(\y_{t+1}),\y-\y_{t+1}\rangle\ge0.
\end{align*}

Rearranging and using the three point lemma in \eqref{lemma:threePointLemma}, we obtain that $\forall \y\in\mathcal{K}$,
\begin{align} \label{ineq:convergenceTerm1y}
\eta_t\langle \y_{t+1}-\y,-\nabla_{\y}f(\Z_{t+1},\w_{t+1})\rangle & \le \langle\nabla\omega_{\y}(\y_t)-\nabla\omega_{\y}(\y_{t+1})),\y_{t+1}-\y\rangle \nonumber 
\\ & = \breg_{\y}(\y,\y_t)-\breg_{\y}(\y,\y_{t+1})-\breg_{\y}(\y_{t+1},\y_t).
\end{align}

For the second term in the RHS of \eqref{ineq:convergenceY}, by the definition of $\w_{t+1}$ in Algorithm \ref{alg:ApproxMP},
\begin{align*}
\w_{t+1}=\argmin_{\y\in\mathcal{K}}\{\langle-\eta_t\nabla_{\y} f(\X_t,\y_t)-\nabla\omega_{\y}(\y_t),\y\rangle+\omega_{\y}(\y)\}.
\end{align*}

Therefore, by the optimality condition for $\w_{t+1}$, $\forall \y\in\mathcal{K}$, 
\[ \langle -\eta_t\nabla_{\y} f(\X_t,\y_t)-\nabla\omega_{\y}(\y_t)+\nabla\omega_{\y}(\w_{t+1}),\y-\w_{t+1}\rangle\ge0. \]

Rearranging and using the three point lemma in \eqref{lemma:threePointLemma}, we obtain that by plugging in $\y=\y_{t+1}$, 
\begin{align} \label{ineq:convergenceTerm2y}
\eta_t\langle\w_{t+1}-\y_{t+1},-\nabla_{\y} f(\X_t,\y_t)\rangle & \le \langle\nabla\omega_{\y}(\y_t)-\nabla\omega_{\y}(\w_{t+1})),\w_{t+1}-\y_{t+1}\rangle \nonumber 
\\ & = \breg_{\y}(\y_{t+1},\y_t)-\breg_{\y}(\y_{t+1},\w_{t+1})-\breg_{\y}(\w_{t+1},\y_t).
\end{align}

For the third term in the RHS of \eqref{ineq:convergenceY}, we obtain
\begin{align} \label{ineq:convergenceTerm3}
& \langle\w_{t+1}-\y_{t+1},\nabla_{\y}f(\X_{t},\y_{t})-\nabla_{\y}f(\Z_{t+1},\w_{t+1}) \rangle \nonumber
\\ & \underset{(a)}{\le} \Vert\w_{t+1}-\y_{t+1}\Vert_{\mathcal{Y}}\cdot\Vert\nabla_{\y}f(\Z_{t+1},\w_{t+1})-\nabla_{\y}f(\X_{t},\y_{t})\Vert_{\mathcal{Y^*}} \nonumber
\\ & \le \Vert\w_{t+1}-\y_{t+1}\Vert_{\mathcal{Y}}\cdot\Vert\nabla_{\y}f(\Z_{t+1},\w_{t+1})-\nabla_{\y}f(\X_{t},\w_{t+1})\Vert_{\mathcal{Y^*}} \nonumber 
\\ & \ \ \ +\Vert\w_{t+1}-\y_{t+1}\Vert_{\mathcal{Y}}\cdot\Vert\nabla_{\y}f(\X_{t},\w_{t+1})-\nabla_{\y}f(\X_{t},\y_{t})\Vert_{\mathcal{Y^*}} \nonumber
\\ & \underset{(b)}{\le} \left(\beta_{y}\Vert\w_{t+1}-\y_{t}\Vert_{\mathcal{Y}}+\beta_{yX}\Vert\Z_{t+1}-\X_{t}\Vert_{\mathcal{X}}\right)\cdot\Vert\w_{t+1}-\y_{t+1}\Vert_{\mathcal{Y}} \nonumber
\\ & \le \left(\beta_{y}\Vert\w_{t+1}-\y_{t}\Vert_{\mathcal{Y}}+\beta_{yX}\Vert\widehat{\Z}_{t+1}-\X_{t}\Vert_{\mathcal{X}}+\beta_{yX}\Vert\Z_{t+1}-\widehat{\Z}_{t+1}\Vert_{\mathcal{X}}\right)\cdot\Vert\w_{t+1}-\y_{t+1}\Vert_{\mathcal{Y}} \nonumber
\\ & \underset{(c)}{\le} \frac{\beta_{y}}{2}\left(\Vert\w_{t+1}-\y_{t}\Vert_{\mathcal{Y}}^2+\Vert\w_{t+1}-\y_{t+1}\Vert_{\mathcal{Y}}^2\right) \nonumber
\\ & \ \ \ +\frac{\beta_{yX}}{2}\left(\Vert\w_{t+1}-\y_{t+1}\Vert_{\mathcal{Y}}^2+(1+\theta)\Vert\widehat{\Z}_{t+1}-\X_{t}\Vert_{\mathcal{X}}^2+2\Vert\widehat{\Z}_{t+1}-\Z_{t+1}\Vert_{\mathcal{X}}^2\right) \nonumber
\\ & \underset{(d)}{\le} \beta_{y}\left(\breg_{\y}(\w_{t+1},\y_t)+\breg_{\y}(\y_{t+1},\w_{t+1})\right)+\beta_{yX}\breg_{\y}(\y_{t+1},\w_{t+1}) \nonumber
\\ & \ \ \ +\beta_{yX}\left((1+\theta)\breg_{\X}(\widehat{\Z}_{t+1},\X_t)+2\breg_{\X}(\widehat{\Z}_{t+1},\Z_{t+1})\right),
\end{align}
where (a) follows from H\"{o}lder's inequality, (b) follows from the $\beta_{y}$ and $\beta_{yX}$ smoothness, (c) follows since for all $a,b\in\reals$ the inequality $2ab\le a^2+b^2$ holds, and (d) follows from \eqref{ineq:strongConvexityOfBregmanDistance}.

Plugging \eqref{ineq:convergenceTerm1}, \eqref{ineq:convergenceTerm2}, and \eqref{ineq:convergenceTerm3} into \eqref{ineq:convergenceY} we obtain
\begin{align} \label{ineq:boundYseiries}
f(\Z_{t+1},\y) - f(\Z_{t+1},\w_{t+1}) & \le \frac{1}{\eta_t}\left(\breg_{\y}(\y,\y_{t})-\breg_{\y}(\y,\y_{t+1})\right) +\left(\beta_{y}-\frac{1}{\eta_t}\right)\breg_{\y}(\w_{t+1},\y_t) \nonumber
\\ & \ \ \ +\left(\beta_{y}+\beta_{yX}-\frac{1}{\eta_t}\right)\breg_{\y}(\y_{t+1},\w_{t+1}) \nonumber
\\ & \ \ \ +\beta_{yX}\left((1+\theta)\breg_{\X}(\widehat{\Z}_{t+1},\X_t)+2\breg_{\X}(\widehat{\Z}_{t+1},\Z_{t+1})\right).
\end{align}

Summing \eqref{ineq:boundXseiries} and \eqref{ineq:boundYseiries} we obtain for $\eta_t\le\min\left\lbrace\frac{1}{\beta_{X}+\beta_{Xy}},\frac{1}{(1+\theta)(\beta_{X}+\beta_{yX})},\frac{1}{\beta_{y}+\beta_{yX}},\frac{1}{\beta_{y}+\beta_{Xy}}\right\rbrace$ that
\begin{align*}
& f(\Z_{t+1},\y) - f(\X,\w_{t+1})
\\ & \le \frac{1}{\eta_t}\left(\breg_{\X}(\X,\X_t)+\breg_{\y}(\y,\y_t)-\breg_{\X}(\X,\widehat{\X}_{t+1})-\breg_{\y}(\y,\y_{t+1})\right)
\\ & \ \ \ +\left((1+\theta)(\beta_{X}+\beta_{yX})-\frac{1}{\eta_t}\right)\breg_{\X}(\widehat{\Z}_{t+1},\X_t)+\left(\beta_{y}+\beta_{Xy}-\frac{1}{\eta_t}\right)\breg_{\y}(\w_{t+1},\y_t)
\\ & \ \ \ +\left(\beta_{X}+\beta_{Xy}-\frac{1}{\eta_t}\right)\breg_{\X}(\widehat{\X}_{t+1},\widehat{\Z}_{t+1})+\left(\beta_{y}+\beta_{yX}-\frac{1}{\eta_t}\right)\breg_{\y}(\y_{t+1},\w_{t+1})
\\ & \ \ \ +\sqrt{2}\Vert\nabla_{\X}f(\Z_{t+1},\w_{t+1})\Vert_2\sqrt{\breg_{\X}(\widehat{\Z}_{t+1},\Z_{t+1})} +2(\beta_{X}+\beta_{yX})\breg_{\X}(\widehat{\Z}_{t+1},\Z_{t+1})
\\ & \le \frac{1}{\eta_t}\left(\breg_{\X}(\X,\X_t)+\breg_{\y}(\y,\y_t)-\breg_{\X}(\X,\X_{t+1})-\breg_{\y}(\y,\y_{t+1})\right)
\\ & \ \ \ + \frac{1}{\eta_t}\left(\breg_{\X}(\X,\X_{t+1})-\breg_{\X}(\X,\widehat{\X}_{t+1})\right)+2(\beta_{X}+\beta_{yX})\breg_{\X}(\widehat{\Z}_{t+1},\Z_{t+1})
\\ & \ \ \ +\sqrt{2}\Vert\nabla_{\X}f(\Z_{t+1},\w_{t+1})\Vert_{\mathcal{X^*}}\sqrt{\breg_{\X}(\widehat{\Z}_{t+1},\Z_{t+1})} .
\end{align*}

Averaging over $t=1,\ldots,T$, and taking a $\eta_t=\eta$ we get that
\begin{align*} 
& \frac{1}{T}\sum_{t=1}^T f(\Z_{t+1},\y) - \frac{1}{T}\sum_{t=1}^T  f(\X,\w_{t+1}) \nonumber
\\ & \le \frac{1}{\eta T}\left(\breg_{\X}(\X,\X_1)+\breg_{\y}(\y,\y_1)\right)\nonumber
+ \frac{1}{\eta T}\sum_{t=1}^T\left(\breg_{\X}(\X,\X_{t+1})-\breg_{\X}(\X,\widehat{\X}_{t+1})\right) \nonumber
\\ & \ \ \  + \frac{2(\beta_{X}+\beta_{yX})}{T}\sum_{t=1}^T\breg_{\X}(\widehat{\Z}_{t+1},\Z_{t+1}) \nonumber
\\ & \ \ \ +\frac{\sqrt{2}}{T}\sum_{t=1}^T\Vert\nabla_{\X}f(\Z_{t+1},\w_{t+1})\Vert_{\mathcal{X^*}}\sqrt{\breg_{\X}(\widehat{\Z}_{t+1},\Z_{t+1})}.
\end{align*}


Maximizing $\frac{1}{T}\sum_{t=1}^Tf(\Z_{t+1},\y)$ over all $\y\in\mathcal{K}$ and minimizing $\frac{1}{T}\sum_{t=1}^Tf(\X,\w_{t+1})$ over all $\X\in\Sn$ and using the convexity of $f(\cdot,\y)$ and concavity of $f(\X,\cdot)$, we obtain that
\begin{align*}
& \max_{\y\in\mathcal{K}} f\left(\frac{1}{T}\sum_{t=1}^T \Z_{t+1},\y\right) -  \min_{\X\in\Sn} f\left(\X,\frac{1}{T}\sum_{t=1}^T\w_{t+1}\right) 
\\ & \le \max_{\y\in\mathcal{K}} \frac{1}{T}\sum_{t=1}^T f\left( \Z_{t+1},\y\right) -  \min_{\X\in\Sn} \frac{1}{T}\sum_{t=1}^T f\left(\X,\w_{t+1}\right)
\\ & \le \frac{1}{\eta T}\max_{(\X,\y)\in\Sn\times\mK}\left(\breg_{\X}(\X,\X_1)+\breg_{\y}(\y,\y_1)\right) \nonumber
\\ & + \frac{1}{\eta T}\sum_{t=1}^T\max_{\X\in\Sn}\left(\breg_{\X}(\X,\X_{t+1})-\breg_{\X}(\X,\widehat{\X}_{t+1})\right)+ \frac{2(\beta_{X}+\beta_{yX})}{T}\sum_{t=1}^T\breg_{\X}(\widehat{\Z}_{t+1},\Z_{t+1}) \nonumber
\\ & +\frac{\sqrt{2}}{T}\sum_{t=1}^T\Vert\nabla_{\X}f(\Z_{t+1},\w_{t+1})\Vert_{\mathcal{X^*}}\sqrt{\breg_{\X}(\widehat{\Z}_{t+1},\Z_{t+1})}.
\end{align*}

\end{proof}

\subsection{Bounds on the bregman distance of updates in Algorithm \ref{alg:ApproxMP} from saddle-points}
\label{sec:appendix:proofLemma7}
The following two lemmas establish that with an appropriate choice of step-size, the distances between the sequences $\lbrace(\X_t,\y_t)\rbrace_{t\ge1}$ and $\lbrace(\Z_{t},\w_t)\rbrace_{t\ge2}$ in Algorithm \ref{alg:ApproxMP}, and an optimal solution $(\X^*,\y^*)$, do not increase, up to a bounded error. Since our results for replacing the full-rank SVDs with only low-rank SVDs hold in a certain proximity of an optimal solution, it is important to ensure that when using a ``warm-start" initialization, throughout the run of Algorithm \ref{alg:ApproxMP}, the iterates indeed remain sufficiently close to the optimal solution. 
\begin{lemma}
\label{lemma:inequalityForDecreasingRadius}
Let $\lbrace(\X_t,\y_t)\rbrace_{t\ge1}$, $\lbrace(\Z_{t},\w_t)\rbrace_{t\ge2}$, $\lbrace\widehat{\X}_t\rbrace_{t\ge2}$, and $\lbrace\widehat{\Z}_{t}\rbrace_{t\ge2}$ be the sequences generated by Algorithm \ref{alg:ApproxMP} with a step-size $\eta_t\ge0$, and denote  $\theta=\bigg\lbrace\begin{array}{ll} 0 & if\ \Z_{t+1}=\widehat{\Z}_{t+1} ~\forall{}t
\\ 1 & otherwise \end{array}$. Let $(\X^*,\y^*)$ be an optimal solution to Problem \eqref{problem1}. Then for all $t\ge 1$ it holds that
\begin{align*}
& \breg_{\X}(\X^*,\widehat{\X}_{t+1})+\breg_{\y}(\y^*,\y_{t+1}) \nonumber
\\ & \le \breg_{\X}(\X^*,\X_t) +\breg_{\y}(\y^*,\y_t) \nonumber
 \\ & \ \ \ +\left(2(1+\theta)\eta_t^2\left(\beta_{X}^2+\beta_{yX}^2\right)-1\right)\breg_{\X}(\widehat{\Z}_{t+1},\X_t)+\left(2\eta_t^2\left(\beta_{y}^2+\beta_{Xy}^2\right)-1\right)\breg_{\y}(\w_{t+1},\y_t) \nonumber
\\ & \ \ \ +4\eta_t^2\left(\beta_{X}^2+\beta_{yX}^2\right)\breg_{\X}(\widehat{\Z}_{t+1},\Z_{t+1})+ \sqrt{2}\eta_tG\sqrt{\breg_{\X}(\widehat{\Z}_{t+1},\Z_{t+1})}.
\end{align*}
\end{lemma}

\begin{proof}

Invoking the gradient inequality,
\begin{align*}
\langle\nabla_{\X} f(\Z_{t+1},\w_{t+1}),\Z_{t+1}-\X^*\rangle & \ge f(\Z_{t+1},\w_{t+1})-f(\X^*,\w_{t+1}) 
\\ & \ge f(\Z_{t+1},\w_{t+1})-f(\X^*,\y^*),
\end{align*}
and similarly,
\begin{align*}
\langle-\nabla_{\y} f(\Z_{t+1},\w_{t+1}),\w_{t+1}-\y^*\rangle & \ge -f(\Z_{t+1},\w_{t+1})+f(\Z_{t+1},\y^*) 
\\ & \ge -f(\Z_{t+1},\w_{t+1})+f(\X^*,\y^*).
\end{align*}
Combining these two inequalities we obtain
\begin{align*}
\left\langle\left[\begin{array}{c}\nabla_{\X} f(\Z_{t+1},\w_{t+1})\\ -\nabla_{\y} f(\Z_{t+1},\w_{t+1})\end{array}\right],\left[\begin{array}{c}\Z_{t+1} \\ \w_{t+1}\end{array}\right]-\left[\begin{array}{c}\X^*\\ \y^*\end{array}\right]\right\rangle & \ge 0.
\end{align*}
Thus,
\begin{align} \label{ineq:inProofBound1}
&\left\langle\left[\begin{array}{c}\nabla_{\X} f(\Z_{t+1},\w_{t+1})\\ -\nabla_{\y} f(\Z_{t+1},\w_{t+1})\end{array}\right],\left[\begin{array}{c}\widehat{\X}_{t+1} \\ \y_{t+1}\end{array}\right]-\left[\begin{array}{c}\X^*\\ \y^*\end{array}\right]\right\rangle
\\ & \ge \left\langle\left[\begin{array}{c}\nabla_{\X} f(\Z_{t+1},\w_{t+1})\\ -\nabla_{\y} f(\Z_{t+1},\w_{t+1})\end{array}\right],\left[\begin{array}{c}\widehat{\X}_{t+1} \\ \y_{t+1}\end{array}\right]-\left[\begin{array}{c}\Z_{t+1}\\ \w_{t+1}\end{array}\right]\right\rangle.
\end{align}

By the definitions of $\widehat{\X}_{t+1}$ and $\y_{t+1}$ in Algorithm \ref{alg:ApproxMP},
\begin{align*}
\widehat{\X}_{t+1}&=\argmin_{\X\in\Sn}\lbrace\langle\eta_t\nabla_{\X}f(\Z_{t+1},\w_{t+1}),\X\rangle+\breg_{\X}(\X,\X_t)\rbrace
\\ & = \argmin_{\X\in\mathcal{S}_n}\lbrace\langle\eta_t\nabla_{\X} f(\Z_{t+1},\w_{t+1})-\nabla\omega_{\X}(\X_t),\X\rangle+\omega_{\X}(\X)\rbrace,
\\ \y_{t+1}&= \argmin_{\y\in\mathcal{K}}\lbrace\langle-\eta_t\nabla_{\y}f(\Z_{t+1},\w_{t+1}),\y\rangle+\breg_{\y}(\y,\y_t)\rbrace
\\ & = \argmin_{\y\in\mathcal{K}}\lbrace\langle-\eta_t\nabla_{\y} f(\Z_{t+1},\w_{t+1})-\nabla\omega_{\y}(\y_t),\y\rangle+\omega_{\y}(\y)\rbrace.
\end{align*}

Therefore, by the optimality conditions for $\widehat{\X}_{t+1}$ and $\y_{t+1}$, it holds for all $\X\in\mathcal{S}_n$ and $\y\in\mathcal{K}$ that, 
\begin{align*}
\langle \eta_t\nabla_{\X} f(\Z_{t+1},\w_{t+1})-\nabla\omega_{\X}(\X_t)+\nabla\omega_{\X}(\widehat{\X}_{t+1}),\X-\widehat{\X}_{t+1}\rangle & \ge0,
\\ \langle -\eta_t\nabla_{\y} f(\Z_{t+1},\w_{t+1})-\nabla\omega_{\y}(\y_t)+\nabla\omega_{\y}(\y_{t+1}),\y-\y_{t+1}\rangle & \ge0. 
\end{align*}

Using the three point lemma as in \eqref{lemma:threePointLemma},
\begin{align*}
0 & \le \langle \eta_t\nabla_{\X} f(\Z_{t+1},\w_{t+1})-\nabla\omega_{\X}(\X_t)+\nabla\omega_{\X}(\widehat{\X}_{t+1}),\X^*-\widehat{\X}_{t+1}\rangle
\\ & = \langle \eta_t\nabla_{\X} f(\Z_{t+1},\w_{t+1}),\X^*-\widehat{\X}_{t+1}\rangle -\breg_{\X}(\X^*,\widehat{\X}_{t+1})-\breg_{\X}(\widehat{\X}_{t+1},\X_t)+\breg_{\X}(\X^*,\X_t), 
\end{align*}
and
\begin{align*}
0 & \le \langle -\eta_t\nabla_{\y} f(\Z_{t+1},\w_{t+1})-\nabla\omega_{\y}(\y_t)+\nabla\omega_{\y}(\y_{t+1}),\y^*-\y_{t+1}\rangle
\\ & = \langle -\eta_t\nabla_{\y} f(\Z_{t+1},\w_{t+1}),\y^*-\y_{t+1}\rangle -\breg_{\y}(\y^*,\y_{t+1})-\breg_{\y}(\y_{t+1},\y_t)+\breg_{\y}(\y^*,\y_t). 
\end{align*}

Combining the last two inequalities and \eqref{ineq:inProofBound1} we obtain
\begin{align} \label{ineq:inProofBound2}
0 & \le \left\langle\left[\begin{array}{c}\nabla_{\X} f(\Z_{t+1},\w_{t+1})\\ -\nabla_{\y} f(\Z_{t+1},\w_{t+1})\end{array}\right],\left[\begin{array}{c}\Z_{t+1}\\ \w_{t+1}\end{array}\right]-\left[\begin{array}{c}\widehat{\X}_{t+1} \\ \y_{t+1}\end{array}\right]\right\rangle -\breg_{\X}(\X^*,\widehat{\X}_{t+1})-\breg_{\X}(\widehat{\X}_{t+1},\X_t) \nonumber
\\ & \ \ \ +\breg_{\X}(\X^*,\X_t)-\breg_{\y}(\y^*,\y_{t+1})-\breg_{\y}(\y_{t+1},\y_t)+\breg_{\y}(\y^*,\y_t) \nonumber
\\ & = \langle \eta_t\nabla_{\X} f(\Z_{t+1},\w_{t+1})-\nabla\omega_{\X}(\X_t)+\nabla\omega_{\X}(\widehat{\Z}_{t+1}),\widehat{\Z}_{t+1}-\widehat{\X}_{t+1}\rangle \nonumber
\\ & \ \ \  + \langle \eta_t\nabla_{\X} f(\Z_{t+1},\w_{t+1}),\Z_{t+1}-\widehat{\Z}_{t+1}\rangle \nonumber
\\ & \ \ \ -\breg_{\X}(\widehat{\X}_{t+1},\widehat{\Z}_{t+1})-\breg_{\X}(\widehat{\Z}_{t+1},\X_t)-\breg_{\X}(\X^*,\widehat{\X}_{t+1})+\breg_{\X}(\X^*,\X_t) \nonumber
\\ & \ \ \ +\langle -\eta_t\nabla_{\y} f(\Z_{t+1},\w_{t+1})-\nabla\omega_{\y}(\y_t)+\nabla\omega_{\y}(\w_{t+1}),\w_{t+1}-\y_{t+1}\rangle \nonumber
\\ & \ \ \ -\breg_{\y}(\y_{t+1},\w_{t+1})-\breg_{\y}(\w_{t+1},\y_t)-\breg_{\y}(\y^*,\y_{t+1})+\breg_{\y}(\y^*,\y_t), 
\end{align}
where the equality holds from \eqref{lemma:threePointLemma}.

By the definitions of $\widehat{\Z}_{t+1}$ and $\w_{t+1}$ in Algorithm \ref{alg:ApproxMP},
\begin{align*}
\widehat{\Z}_{t+1} & =\argmin_{\X\in\mathcal{S}_n}\lbrace\langle\eta_t\nabla_{\X} f(\X_t,\y_t)-\nabla\omega_{\X}(\X_t),\X\rangle+\omega_{\X}(\X)\rbrace,
\\ \w_{t+1} & = \argmin_{\y\in\mathcal{K}}\lbrace\langle-\eta_t\nabla_{\y} f(\X_t,\y_t)-\nabla\omega_{\y}(\y_t),\y\rangle+\omega_{\y}(\y)\rbrace.
\end{align*}

Therefore, by the optimality conditions for $\widehat{\Z}_{t+1}$ and $\w_{t+1}$, for all $\X\in\mathcal{S}_n$ and $\y\in\mathcal{K}$, 
\begin{align}
\langle \eta_t\nabla_{\X} f(\X_t,\y_t)-\nabla\omega_{\X}(\X_t)+\nabla\omega_{\X}(\widehat{\Z}_{t+1}),\X-\widehat{\Z}_{t+1}\rangle & \ge0, \label{ineq:Zhatopt:inProof}
\\ \langle -\eta_t\nabla_{\y} f(\X_t,\y_t)-\nabla\omega_{\y}(\y_t)+\nabla\omega_{\y}(\w_{t+1}),\y-\w_{t+1}\rangle & \ge0. \label{ineq:Wopt:inProof}
\end{align}

Therefore, taking $\X=\widehat{\X}_{t+1}$ it holds that
\begin{align} \label{ineq:inProofOne}
& \langle \eta_t\nabla_{\X} f(\Z_{t+1},\w_{t+1})-\nabla\omega_{\X}(\X_t)+\nabla\omega_{\X}(\widehat{\Z}_{t+1}),\widehat{\Z}_{t+1}-\widehat{\X}_{t+1}\rangle \nonumber
\\ & = \langle \eta_t\nabla_{\X} f(\X_t,\y_t)-\nabla\omega_{\X}(\X_t)+\nabla\omega_{\X}(\widehat{\Z}_{t+1}),\widehat{\Z}_{t+1}-\widehat{\X}_{t+1}\rangle  \nonumber
\\ & \ \ \ +\eta_t\langle \nabla_{\X} f(\Z_{t+1},\w_{t+1})-\nabla_{\X} f(\X_t,\y_t),\widehat{\Z}_{t+1}-\widehat{\X}_{t+1}\rangle \nonumber
\\ & \underset{(a)}{\le} \eta_t\langle \nabla_{\X} f(\Z_{t+1},\w_{t+1})-\nabla_{\X} f(\X_t,\y_t),\widehat{\Z}_{t+1}-\widehat{\X}_{t+1}\rangle \nonumber
\\ & \underset{(b)}{\le} \eta_t\Vert\nabla_{\X} f(\Z_{t+1},\w_{t+1})-\nabla_{\X} f(\X_t,\y_t)\Vert_{\mathcal{X^*}}\cdot\Vert\widehat{\Z}_{t+1}-\widehat{\X}_{t+1}\Vert_{\mathcal{X}} \nonumber
\\ & \underset{(c)}{\le} \eta_t\left(\beta_{X}\Vert\Z_{t+1}-\X_t\Vert_{\mathcal{X}}+\beta_{Xy}\Vert\w_{t+1}-\y_t\Vert_{\mathcal{Y}}\right)\cdot\Vert\widehat{\Z}_{t+1}-\widehat{\X}_{t+1}\Vert_{\mathcal{X}} \nonumber
\\ & \underset{(d)}{\le} \eta_t^2\beta_{X}^2\Vert\Z_{t+1}-\X_t\Vert_{\mathcal{X}}^2+\eta_t^2\beta_{Xy}^2\Vert\w_{t+1}-\y_t\Vert_{\mathcal{Y}}^2+\frac{1}{2}\Vert\widehat{\Z}_{t+1}-\widehat{\X}_{t+1}\Vert_{\mathcal{X}}^2 \nonumber
\\ & \underset{(e)}{\le} (1+\theta)\eta_t^2\beta_{X}^2\Vert\widehat{\Z}_{t+1}-\X_t\Vert_{\mathcal{X}}^2+2\eta_t^2\beta_{X}^2\Vert\widehat{\Z}_{t+1}-\Z_{t+1}\Vert_{\mathcal{X}}^2+\eta_t^2\beta_{Xy}^2\Vert\w_{t+1}-\y_t\Vert_{\mathcal{Y}}^2 \nonumber \\ & \ \ \  +\frac{1}{2}\Vert\widehat{\Z}_{t+1}-\widehat{\X}_{t+1}\Vert_{\mathcal{X}}^2 \nonumber
\\ & \underset{(f)}{\le} 2(1+\theta)\eta_t^2\beta_{X}^2\breg_{\X}(\widehat{\Z}_{t+1},\X_t)+4\eta_t^2\beta_{X}^2\breg_{\X}(\widehat{\Z}_{t+1},\Z_{t+1})+2\eta_t^2\beta_{Xy}^2\breg_{\y}(\w_{t+1},\y_t) \nonumber \\ & \ \ \ +\breg_{\X}(\widehat{\X}_{t+1},\widehat{\Z}_{t+1}),
\end{align}
where (a) follows from \eqref{ineq:Zhatopt:inProof}, (b) follows from H\"{o}lder's inequality, (c) follows from the $\beta_X$ and $\beta_{Xy}$ smoothness, (d) follows since for all $a,b\in\reals$ the inequality $ab\le a^2+\frac{1}{4}b^2$ holds, (e) holds since when using exact updates then $\widehat{\Z}_{t+1}=\Z_{t+1}$ and $\theta=0$ and then $(1+\theta)\Vert\widehat{\Z}_{t+1}-\X_t\Vert_{\mathcal{X}}^2+2\Vert\widehat{\Z}_{t+1}-\Z_{t+1}\Vert_{\mathcal{X}}^2=\Vert\Z_{t+1}-\X_t\Vert_{\mathcal{X}}^2$ and otherwise $\theta=1$, and (f) follows from \eqref{ineq:strongConvexityOfBregmanDistance}.

Similarly, taking $\y=\y_{t+1}$ it holds that
\begin{align} \label{ineq:inProofTwo}
& \langle -\eta_t\nabla_{\y} f(\Z_{t+1},\w_{t+1})-\nabla\omega_{\y}(\y_t)+\nabla\omega_{\y}(\w_{t+1}),\w_{t+1}-\y_{t+1}\rangle \nonumber
\\ & = \langle -\eta_t\nabla_{\y} f(\X_t,\y_t)-\nabla\omega_{\y}(\y_t)+\nabla\omega_{\y}(\w_{t+1}),\w_{t+1}-\y_{t+1}\rangle \nonumber
\\ & \ \ \ +\eta_t\langle \nabla_{\y} f(\X_t,\y_t)-\nabla_{\y} f(\Z_{t+1},\w_{t+1}),\w_{t+1}-\y_{t+1}\rangle \nonumber
\\ & \underset{(a)}{\le} \eta_t\langle \nabla_{\y} f(\X_t,\y_t)-\nabla_{\y} f(\Z_{t+1},\w_{t+1}),\w_{t+1}-\y_{t+1}\rangle \nonumber
\\ & \underset{(b)}{\le} \eta_t\Vert\nabla_{\y} f(\Z_{t+1},\w_{t+1})-\nabla_{\y} f(\X_t,\y_t)\Vert_{\mathcal{Y^*}}\cdot\Vert\w_{t+1}-\y_{t+1}\Vert_{\mathcal{Y}} \nonumber
\\ & \underset{(c)}{\le} \eta_t\left(\beta_{y}\Vert\w_{t+1}-\y_t\Vert_{\mathcal{Y}}+\beta_{yX}\Vert\Z_{t+1}-\X_t\Vert_{\mathcal{X}}\right)\cdot\Vert\w_{t+1}-\y_{t+1}\Vert_{\mathcal{Y}} \nonumber
\\ & \underset{(d)}{\le} \eta_t^2\beta_{y}^2\Vert\w_{t+1}-\y_t\Vert_{\mathcal{Y}}^2 + \eta_t^2\beta_{yX}^2\Vert\Z_{t+1}-\X_t\Vert_{\mathcal{X}}^2 + \frac{1}{2}\Vert\w_{t+1}-\y_{t+1}\Vert_{\mathcal{Y}}^2 \nonumber
\\ & \underset{(e)}{\le} \eta_t^2\beta_{y}^2\Vert\w_{t+1}-\y_t\Vert_{\mathcal{Y}}^2 + (1+\theta)\eta_t^2\beta_{yX}^2\Vert\widehat{\Z}_{t+1}-\X_t\Vert_{\mathcal{X}}^2 + 2\eta_t^2\beta_{yX}^2\Vert\widehat{\Z}_{t+1}-\Z_{t+1}\Vert_{\mathcal{X}}^2 \nonumber \\ & \ \ \ + \frac{1}{2}\Vert\w_{t+1}-\y_{t+1}\Vert_{\mathcal{Y}}^2 \nonumber
\\ & \le 2\eta_t^2\beta_{y}^2\breg_{\y}(\w_{t+1},\y_t) + 2(1+\theta)\eta_t^2\beta_{yX}^2\breg_{\X}(\widehat{\Z}_{t+1},\X_t)  + 4\eta_t^2\beta_{yX}^2\breg_{\X}(\widehat{\Z}_{t+1},\Z_{t+1}) \nonumber \\ & \ \ \ + \breg_{\y}(\y_{t+1},\w_{t+1}),
\end{align}
where (a) follows from \eqref{ineq:Wopt:inProof}, (b) follows from H\"{o}lder's inequality, (c) follows from the $\beta_y$ and $\beta_{yX}$ smoothness, (d) follows since for all $a,b\in\reals$ the inequality $ab\le a^2+\frac{1}{4}b^2$ holds, (e) holds since when using exact updates then $\widehat{\Z}_{t+1}=\Z_{t+1}$ and $\theta=0$ and then $(1+\theta)\Vert\widehat{\Z}_{t+1}-\X_t\Vert_{\mathcal{X}}^2+2\Vert\widehat{\Z}_{t+1}-\Z_{t+1}\Vert_{\mathcal{X}}^2=\Vert\Z_{t+1}-\X_t\Vert_{\mathcal{X}}^2$ and otherwise $\theta=1$, and (f) follows from \eqref{ineq:strongConvexityOfBregmanDistance}.

Also, using H\"{o}lder's inequality it holds that
\begin{align} \label{ineq:inProofThree}
\langle \eta_t\nabla_{\X} f(\Z_{t+1},\w_{t+1}),\Z_{t+1}-\widehat{\Z}_{t+1}\rangle & \le \eta_t\Vert\nabla_{\X} f(\Z_{t+1},\w_{t+1})\Vert_{\mathcal{X^*}}\cdot\Vert\Z_{t+1}-\widehat{\Z}_{t+1}\Vert_{\mathcal{X}} \nonumber
\\ & \le \sqrt{2}\eta_tG\sqrt{\breg_{\X}(\widehat{\Z}_{t+1},\Z_{t+1})}.  
\end{align}

Plugging the inequalities \eqref{ineq:inProofOne}, \eqref{ineq:inProofTwo}, and \eqref{ineq:inProofThree} into \eqref{ineq:inProofBound2} we obtain that
\begin{align*}
& \breg_{\X}(\X^*,\widehat{\X}_{t+1})+\breg_{\y}(\y^*,\y_{t+1})
\\ & \le \breg_{\X}(\X^*,\X_t) +\breg_{\y}(\y^*,\y_t) +4\eta_t^2\left(\beta_{X}^2+\beta_{yX}^2\right)\breg_{\X}(\widehat{\Z}_{t+1},\Z_{t+1})
\\ & \ \ \ + \sqrt{2}\eta_tG\sqrt{\breg_{\X}(\widehat{\Z}_{t+1},\Z_{t+1})}
+\left(2(1+\theta)\eta_t^2\left(\beta_{X}^2+\beta_{yX}^2\right)-1\right)\breg_{\X}(\widehat{\Z}_{t+1},\X_t) \\ & \ \ \ +\left(2\eta_t^2\left(\beta_{y}^2+\beta_{Xy}^2\right)-1\right)\breg_{\y}(\w_{t+1},\y_t). 
\end{align*}

\end{proof}

%

\begin{lemma}
 \label{lemma:decreasingRadius}
Let $\lbrace(\X_t,\y_t)\rbrace_{t\ge1}$, $\lbrace(\Z_{t},\w_t)\rbrace_{t\ge2}$, $\lbrace\widehat{\X}_t\rbrace_{t\ge2}$, and $\lbrace\widehat{\Z}_{t}\rbrace_{t\ge2}$ be the sequences generated by Algorithm \ref{alg:ApproxMP} with a step-size $\eta_t\ge0$ such that $\eta_t^2\cdot\max\left\lbrace 2(1+\theta)(\beta_{X}^2+\beta_{yX}^2),2(\beta_{y}^2+\beta_{Xy}^2)\right\rbrace-1\le0$, where  $\theta=\bigg\lbrace\begin{array}{ll} 0 & if\ \Z_{t+1}=\widehat{\Z}_{t+1} ~\forall t
\\ 1 & otherwise \end{array}$.  Let $(\X^*,\y^*)$ be an optimal solution to Problem \eqref{problem1}. Then for all $t\ge 1$ it holds that
\begin{align*}
& \breg_{\X}(\X^*,\X_{t+1})+\breg_{\y}(\y^*,\y_{t+1}) 
\\ & \le \breg_{\X}(\X^*,\X_t) +\breg_{\y}(\y^*,\y_t)
+\breg_{\X}(\widehat{\Z}_{t+1},\Z_{t+1}) \\ & \ \ \ + \sqrt{2}\eta_tG\sqrt{\breg_{\X}(\widehat{\Z}_{t+1},\Z_{t+1})}+\left(\breg_{\X}(\X^*,\X_{t+1})-\breg_{\X}(\X^*,\widehat{\X}_{t+1})\right)
\end{align*}
and
\begin{align*}
& \Vert\Z_{t+1}-\X^*\Vert_{\mathcal{X}}^2+\Vert\w_{t+1}-\y^*\Vert_{\mathcal{Y}}^2 
\\ & \le \left(\frac{1}{\gamma_t}+8\right)(\breg_{\X}(\X^*,\X_t) +\breg_{\y}(\y^*,\y_t))
\\ & \ \ \  +\left(\frac{1}{\gamma_t}+8\right)\breg_{\X}(\widehat{\Z}_{t+1},\Z_{t+1})+ \frac{\sqrt{2}\eta_tG}{\gamma_t}\sqrt{\breg_{\X}(\widehat{\Z}_{t+1},\Z_{t+1})},
\end{align*} 
where $\gamma_t = \min\left\lbrace 1-2(1+\theta)\eta_t^2\left(\beta_{X}^2+\beta_{yX}^2\right),1-2\eta_t^2\left(\beta_{y}^2+\beta_{Xy}^2\right)\right\rbrace$.
\end{lemma}

\begin{proof}

By our choice of step-size it holds that $2(1+\theta)\eta_t^2(\beta_{X}^2+\beta_{yX}^2)-1\le 0$ and $2\eta_t^2(\beta_{y}^2+\beta_{Xy}^2)-1\le0$, and so, from Lemma \ref{lemma:inequalityForDecreasingRadius} we obtain that
\begin{align*}
& \breg_{\X}(\X^*,\X_{t+1})+\breg_{\y}(\y^*,\y_{t+1}) 
\\ & \le \breg_{\X}(\X^*,\X_t) +\breg_{\y}(\y^*,\y_t)
\\ & \ \ \ +\breg_{\X}(\widehat{\Z}_{t+1},\Z_{t+1})+ \sqrt{2}\eta_tG\sqrt{\breg_{\X}(\widehat{\Z}_{t+1},\Z_{t+1})}+\breg_{\X}(\X^*,\X_{t+1})-\breg_{\X}(\X^*,\widehat{\X}_{t+1}). 
\end{align*}

In addition, from Lemma \ref{lemma:inequalityForDecreasingRadius}, by our choice of $\eta_t$, and invoking \eqref{ineq:strongConvexityOfBregmanDistance}, it holds that
\begin{align*}
& \breg_{\X}(\X^*,\widehat{\X}_{t+1})+\breg_{\y}(\y^*,\y_{t+1}) 
\\ & \le \breg_{\X}(\X^*,\X_t) +\breg_{\y}(\y^*,\y_t)+\breg_{\X}(\widehat{\Z}_{t+1},\Z_{t+1})+ \sqrt{2}\eta_tG\sqrt{\breg_{\X}(\widehat{\Z}_{t+1},\Z_{t+1})}
 \\ & \ \ \ +\frac{1}{2}\left(2(1+\theta)\eta_t^2\left(\beta_{X}^2+\beta_{yX}^2\right)-1\right)\Vert\widehat{\Z}_{t+1}-\X_t\Vert_{\mathcal{X}}^2
 \\ & \ \ \ +\frac{1}{2}\left(2\eta_t^2\left(\beta_{y}^2+\beta_{Xy}^2\right)-1\right)\Vert\w_{t+1}-\y_t\Vert_{\mathcal{Y}}.
\end{align*}

Since $\breg_{\X}(\X^*,\widehat{\X}_{t+1})+\breg_{\y}(\y^*,\y_{t+1})\ge0$ and by rearranging we obtain that
\begin{align*}
& \frac{1}{2}\left(1-2(1+\theta)\eta_t^2\left(\beta_{X}^2+\beta_{yX}^2\right)\right)\Vert\widehat{\Z}_{t+1}-\X_t\Vert_{\mathcal{X}}^2+\frac{1}{2}\left(1-2\eta_t^2\left(\beta_{y}^2+\beta_{Xy}^2\right)\right)\Vert\w_{t+1}-\y_t\Vert_{\mathcal{Y}}
\\ & \le \breg_{\X}(\X^*,\X_t) +\breg_{\y}(\y^*,\y_t)
+\breg_{\X}(\widehat{\Z}_{t+1},\Z_{t+1})+ \sqrt{2}\eta_tG\sqrt{\breg_{\X}(\widehat{\Z}_{t+1},\Z_{t+1})}.
\end{align*}

Denoting 
$\gamma_t = \min\left\lbrace 1-2(1+\theta)\eta_t^2\left(\beta_{X}^2+\beta_{yX}^2\right),1-2\eta_t^2\left(\beta_{y}^2+\beta_{Xy}^2\right)\right\rbrace$ 
we get that
\begin{align} \label{ineq:WithinProof8}
& \Vert\widehat{\Z}_{t+1}-\X_t\Vert_{\mathcal{X}}^2+\Vert\w_{t+1}-\y_t\Vert_{\mathcal{Y}}^2 \nonumber
\\ & \le \frac{1}{2\gamma_t}\left(\breg_{\X}(\X^*,\X_t) +\breg_{\y}(\y^*,\y_t)+\breg_{\X}(\widehat{\Z}_{t+1},\Z_{t+1})+ \sqrt{2}\eta_tG\sqrt{\breg_{\X}(\widehat{\Z}_{t+1},\Z_{t+1})}\right).  
\end{align}

Therefore, since for all $a,b,c\in\reals$ it holds that $(a+b+c)^2\le4a^2+2b^2+4c^2$, invoking \eqref{ineq:strongConvexityOfBregmanDistance}, and plugging in \eqref{ineq:WithinProof8} we have that
\begin{align*}
& \Vert\Z_{t+1}-\X^*\Vert_{\mathcal{X}}^2+\Vert\w_{t+1}-\y^*\Vert_{\mathcal{Y}}^2 
\\ & \le 4\Vert\Z_{t+1}-\widehat{\Z}_{t+1}\Vert_{\mathcal{X}}^2+2\Vert\widehat{\Z}_{t+1}-\X_t\Vert_{\mathcal{X}}^2+4\Vert\X_t-\X^*\Vert_{\mathcal{X}}^2+2\Vert\w_{t+1}-\y_t\Vert_{\mathcal{Y}}^2\\ & \ \ \ +2\Vert\y_t-\y^*\Vert_{\mathcal{Y}}^2
\\ & \le \left(\frac{1}{\gamma_t}+8\right)(\breg_{\X}(\X^*,\X_t) +\breg_{\y}(\y^*,\y_t))
+\left(\frac{1}{\gamma_t}+8\right)\breg_{\X}(\widehat{\Z}_{t+1},\Z_{t+1})\\ & \ \ \ + \frac{\sqrt{2}\eta_tG}{\gamma_t}\sqrt{\breg_{\X}(\widehat{\Z}_{t+1},\Z_{t+1})}.
\end{align*}

\end{proof}

\section{Proofs omitted from from Section \ref{sec:EucledianCase}}

\subsection{Relationship between full Lipschitz parameter and its components}
\label{sec:appendix:fullLipschitz}

We denote by $\beta$ the full Euclidean Lipschitz parameter of the gradient, that is for any $\X,\tilde{\X}\in\Sn$ and $\y,\tilde{\y}\in\mathcal{K}$, 
\begin{align*}
\Vert(\nabla_{\X}f(\X,\y),-\nabla_{\y}f(\X,\y))-(\nabla_{\X}f(\tilde{\X},\tilde{\y}),-\nabla_{\y}f(\tilde{\X},\tilde{\y}))\Vert \le\beta\Vert(\X,\Y)-(\tilde{\X},\tilde{\y})\Vert.
\end{align*}

To establish the relationship between $\beta$ and $\beta_X,\beta_y,\beta_{Xy},\beta_{yX}$, we can see that for all $\X,\tilde{\X}\in\Sn$ and all $\y,\tilde{\y}\in\mathcal{K}$
\begin{align*}
& \Vert(\nabla_{\X}f(\X,\y),-\nabla_{\y}f(\X,\y))-(\nabla_{\X}f(\tilde{\X},\tilde{\y}),-\nabla_{\y}f(\tilde{\X},\tilde{\y}))\Vert^2 
\\ & = \Vert\nabla_{\X}f(\X,\y)-\nabla_{\X}f(\tilde{\X},\y)\Vert_F^2 + \Vert\nabla_{\y}f(\X,\y)-\nabla_{\y}f(\tilde{\X},\tilde{\y})\Vert_F^2
\\ & \le 2\Vert\nabla_{\X}f(\X,\y)-\nabla_{\X}f(\tilde{\X},\y)\Vert_F^2 + 2\Vert\nabla_{\X}f(\tilde{\X},\y)-\nabla_{\X}f(\tilde{\X},\tilde{\y})\Vert_F^2 \\ &  \ \ \ + 2\Vert\nabla_{\y}f(\X,\y)-\nabla_{\y}f(\tilde{\X},\y)\Vert_F^2 + 2\Vert\nabla_{\y}f(\tilde{\X},\y)-\nabla_{\y}f(\tilde{\X},\tilde{\y})\Vert_F^2
\\ & \le 2(\beta_{X}^2+\beta_{yX}^2)\Vert\X-\tilde{\X}\Vert_F^2 + 2(\beta_{y}^2+\beta_{Xy}^2)\Vert\y-\tilde{\y}\Vert_F^2
\\ & \le 2\max\lbrace\beta_{X}^2+\beta_{yX}^2,\beta_{y}^2+\beta_{Xy}^2\rbrace\Vert(\X,\y)-(\tilde{\X},\tilde{\y})\Vert^2.
\end{align*}

Therefore, $\beta= \sqrt{2}\max\left\lbrace\sqrt{\beta_{X}^2+\beta_{yX}^2},\sqrt{\beta_{y}^2+\beta_{Xy}^2}\right\rbrace$.

\subsection{Proof of Lemma \ref{lemma:convergenceOfXYZWsequences}}
\label{sec:AppendixLemma9}
We first restate the lemma and then prove it.

\begin{lemma}
Let $\lbrace(\X_t,\y_t)\rbrace_{t\ge1}$ and $\lbrace(\Z_{t},\w_t)\rbrace_{t\ge2}$ be the sequences generated by Algorithm \ref{alg:EG} with a step-size $\eta_t\le\frac{1}{\beta}$ where $\beta= \sqrt{2}\max\left\lbrace\sqrt{\beta_{X}^2+\beta_{yX}^2},\sqrt{\beta_{y}^2+\beta_{Xy}^2}\right\rbrace$. Then for all $t\ge1$ it holds that 
\begin{align*}
\Vert(\X_{t+1},\y_{t+1})-(\X^*,\y^*)\Vert & \le \Vert(\X_{t},\y_t)-(\X^*,\y^*)\Vert,
\\ \Vert(\Z_{t+1},\w_{t+1})-(\X^*,\y^*)\Vert & \le \Bigg(1+\frac{1}{\sqrt{1-\eta_{t}^2\beta^2}}\Bigg)\Vert(\X_t,\y_t)-(\X^*,\y^*)\Vert.
\end{align*}
\end{lemma}

\begin{proof}
Invoking Lemma \ref{lemma:inequalityForDecreasingRadius} with exact updates, that is $\X_{t+1}=\widehat{\X}_{t+1}$, $\Z_{t+1}=\widehat{\Z}_{t+1}$ for all $t$, and $\theta=0$, it follows that
\begin{align} \label{ineq:EGconvergenceOfSeries}
 &\Vert(\X_{t+1},\y_{t+1})-(\X^*,\y^*)\Vert^2 \nonumber
 \\ & \le \Vert(\X_{t},\y_t)-(\X^*,\y^*)\Vert^2 - (1-\eta_t^2\beta^2)\Vert(\X_t,\y_t)-(\Z_{t+1},\w_{t+1})\Vert^2.
\end{align}

Since $\eta_t^2\beta^2\le1$ it follows that
\begin{align*}
\Vert(\X_{t+1},\y_{t+1})-(\X^*,\y^*)\Vert \le \Vert(\X_{t},\y_t)-(\X^*,\y^*)\Vert.
\end{align*}

In addition, using \eqref{ineq:EGconvergenceOfSeries}
\begin{align*}
\Vert(\X_t,\y_t)-(\Z_{t+1},\w_{t+1})\Vert & \le \sqrt{(1-\eta_t^2\beta^2)^{-1}}\Vert(\X_{t},\y_t)-(\X^*,\y^*)\Vert.
\end{align*}

Therefore,
\begin{align*}
& \Vert(\Z_{t+1},\w_{t+1})-(\X^*,\y^*)\Vert \le \Vert(\Z_{t+1},\w_{t+1})-(\X_{t},\y_{t})\Vert+\Vert(\X_{t},\y_{t})-(\X^*,\y^*)\Vert
\\ & \le \sqrt{(1-\eta_{t}^2\beta^2)^{-1}}\Vert(\X_{t},\y_{t})-(\X^*,\y^*)\Vert +\Vert(\X_{t},\y_{t})-(\X^*,\y^*)\Vert
\\ & = \Bigg(1+\frac{1}{\sqrt{1-\eta_{t}^2\beta^2}}\Bigg)\Vert(\X_{t},\y_{t})-(\X^*,\y^*)\Vert.
\end{align*}
\end{proof}

\section{Calculating the dual-gap in saddle-point problems} \label{appendixDualGapCalculation}

Set some point $(\widehat{\Z},\widehat{\w})\in{\lbrace\trace(\X)=\tau,\ \X\succeq 0\rbrace\times\mathcal{K}}$. Using the concavity of $f(\widehat{\Z},\cdot)$ and convexity of $f(\cdot,\widehat{\w})$, for all $\y\in\mathcal{K}$ and $\X\in\lbrace\trace(\X)=\tau,\ \X\succeq 0\rbrace$, it holds that
\begin{align*}
f(\widehat{\Z},\y)-f(\widehat{\Z},\widehat{\w}) & \le \langle\widehat{\w}-\y,-\nabla_{\y}f(\widehat{\Z},\widehat{\w})\rangle, 
\\ f(\widehat{\Z},\widehat{\w})-f(\X,\widehat{\w}) & \le \langle\widehat{\Z}-\X,\nabla_{\X}f(\widehat{\Z},\widehat{\w})\rangle.
\end{align*}

By taking the maximum of all $\y\in\mathcal{K}$ we obtain in particular that
\begin{align*}
f(\X^*,\y^*)-f(\widehat{\Z},\widehat{\w}) &  \le f(\widehat{\Z},\y^*)-f(\widehat{\Z},\widehat{\w}) \le \max_{\y\in\mathcal{K}}f(\widehat{\Z},\y)-f(\widehat{\Z},\widehat{\w})
\\ & \le \max_{\y\in\mathcal{K}}\langle\widehat{\w}-\y,-\nabla_{\y}f(\widehat{\Z},\widehat{\w})\rangle,
\end{align*}
and taking the maximum of all $\X\in\lbrace\trace(\X)=\tau,\ \X\succeq 0\rbrace$
\begin{align*}
f(\widehat{\Z},\widehat{\w})-f(\X^*,\y^*) & \le f(\widehat{\Z},\widehat{\w})-f(\X^*,\widehat{\w}) \le f(\widehat{\Z},\widehat{\w})-\min_{\substack{\trace(\X)=\tau,\\ \X\succeq 0}}f(\X,\widehat{\w}) 
\\ & \le \max_{\substack{\trace(\X)=\tau,\\ \X\succeq 0}}\langle\widehat{\Z}-\X,\nabla_{\X}f(\widehat{\Z},\widehat{\w})\rangle.
\end{align*}

Summing these two inequalities, we obtain a bound on the dual-gap at $(\widehat{\Z},\widehat{\w})$ which can be written as
\begin{align*}
g(\widehat{\Z})-g^* & \le \max_{\y\in\mathcal{K}}f(\widehat{\Z},\y)-\min_{\X\in\Sn}f(\X,\widehat{\w}) 
\\ & \le \max_{\substack{\trace(\X)=\tau,\\ \X\succeq 0}}\langle\widehat{\Z}-\X,\nabla_{\X}f(\widehat{\Z},\widehat{\w})\rangle - \min_{\y\in\mathcal{K}} \langle\widehat{\w}-\y,\nabla_{\y}f(\widehat{\Z},\widehat{\w})\rangle.
\end{align*} 

It is easy to see that the maximizer of the first term in the RHS of the above is $\tau\v_n\v_n^{\top}$ where $\v_n$ is the smallest eigenvector of $\nabla_{\X}f(\widehat{\Z},\widehat{\W})$, and the minimizer of the second term is $\Y_{i,j}=\sign(\nabla_{\Y}f(\widehat{\Z},\widehat{\W})_{i,j})$ for $\mathcal{K}=\lbrace\Y\in\reals^{n\times n}\ \vert\ \Vert\Y\Vert_{\infty}\le1\rbrace$ and $\nabla_{\y}f(\widehat{\Z},\widehat{\w})/\Vert\nabla_{\y}f(\widehat{\Z},\widehat{\w})\Vert_2$ for $\mathcal{K}=\lbrace\y\in\reals^{n}\ \vert\ \Vert\y\Vert_{2}\le1\rbrace$.




\end{appendices}


\section{Compliance with Ethical Standards}
Dan Garber is supported by the ISRAEL SCIENCE FOUNDATION (grant No. 2267/22). No other funding was received for this work, directly or indirectly, and there are no competing interests.

\bibliography{sn-bibliography}

\end{document}